\documentclass[a4paper,12pt]{amsart}
\usepackage[usenames,dvipsnames]{xcolor}
\usepackage{hyperref, tikz-cd}
\hypersetup{colorlinks=true,linkcolor={Brown},citecolor={Brown},urlcolor={Brown}}

\usepackage{microtype}
\usepackage{amssymb,amsmath,amsfonts,amsthm,bm,latexsym,amscd,verbatim,url,nicefrac,stmaryrd,enumerate,colonequals,dsfont,appendix,stackrel,mathrsfs,cleveref}
\crefname{exm}{Example}{Examples}
\crefname{cor}{Corollary}{Corollaries}
\crefname{prop}{Proposition}{Propositions}
\crefname{rmk}{Remark}{Remarks}
\crefname{lem}{Lemma}{Lemmata}

\usepackage[color,matrix,arrow]{xy}
\usepackage{float}
\usepackage{times}
\usepackage{graphicx}
\usepackage{fullpage,adjustbox,soul}
\usepackage{array}

\makeatletter
\newcommand{\oset}[3][0ex]{%
	\mathrel{\mathop{#3}\limits^{
			\vbox to#1{\kern-2\ex@
				\hbox{$\scriptstyle#2$}\vss}}}}
\makeatother

\RequirePackage{xcolor}
\providecommand{\alberto}[1]{{\color{blue}{#1}}}

\numberwithin{equation}{section}%
\newtheorem{thm}[equation]{Theorem}
\newtheorem{prop}[equation]{Proposition}
\newtheorem{lemma}[equation]{Lemma}
\newtheorem{cor}[equation]{Corollary}

\theoremstyle{definition}
\newtheorem{dfn}[equation]{Definition}
\newtheorem{notn}[equation]{Notation}
\newtheorem{rmk}[equation]{Remark}
\newtheorem{exm}[equation]{Example}
\newtheorem{assu}[equation]{Assumption}
\newtheorem{cons}[equation]{Construction}

\usepackage{bbold,etoolbox}
\newcommand{\one}{\mathbb{1}}

\newcommand{\Gmd}{\mathsf{M}^{\mathrm{coh}}(\mathbb{G}_{\tiny \textup{m}})} %
\newcommand{\A}{\mathbb{A}}
\newcommand{\B}{\mathbb{B}}
\newcommand{\C}{\mathbb{C}}
\newcommand{\E}{\mathbb{E}}
\newcommand{\F}{\mathbb{F}}
\newcommand{\barFp}{\bar{\F}_{\!p}}
\newcommand{\G}{\mathbb{G}}
\newcommand{\bH}{\mathbb{H}}

\renewcommand{\P}{\mathbb{P}}
\newcommand{\pirnd}[1][]{
	\ifblank{#1}{\rho^{\textup{nil}}}{\rho^{\textup{nil}}_{#1}}}
\newcommand{\pirn}{\pirnd{}}
\newcommand{\Q}{\mathbb{Q}}
\newcommand{\rF}{\varphi_{\rm rel}} %
\newcommand{\rFactual}{\varphi_{\rm rel}} %
\newcommand{\Ftilt}{\varphi^{\sharp}} %

\newcommand{\T}{\mathbb{T}}
\newcommand{\Z}{\mathbb{Z}}

\newcommand{\catC}{\mathscr{C}}
\newcommand{\catH}{\mathscr{H}}
\newcommand{\catD}{\mathscr{D}}

\newcommand{\catT}{\mathbf{T}}

\newcommand{\mcA}{\mathcal{A}}

\newcommand{\mcD}{\mathcal{D}}

\newcommand{\mcK}{\mathcal{K}}

\newcommand{\mcO}{\mathcal{O}}

\newcommand{\mcX}{\mathcal{X}}
\newcommand{\mcY}{\mathcal{Y}}

\newcommand{\del}{\partial}

\newcommand{\xto}[1]{\xrightarrow{#1}}
\newcommand{\isoto}{\xrightarrow{\sim}}
\newcommand{\xfrom}[1]{\xleftarrow{#1}}
\newcommand{\from}{\xleftarrow{}}
\newcommand{\isofrom}{\xleftarrow{\sim}}

\DeclareMathOperator{\alg}{alg}
\DeclareMathOperator{\Alt}{Alt}
\DeclareMathOperator{\an}{an}

\DeclareMathOperator{\An}{An}

\DeclareMathOperator{\CAlg}{CAlg}
\DeclareMathOperator{\coAlg}{coAlg}
\DeclareMathOperator{\car}{char}

\DeclareMathOperator{\colim}{colim}

\DeclareMathOperator{\cpt}{\omega} %

\DeclareMathOperator{\dR}{dR}

\DeclareMathOperator{\Derfp}{\mathcal{D}^{\mathit{b}}_{\varphi}}
\newcommand{\Dfp}[1][]{\ifblank{#1}{\Derfp(K_0)}{\Derfp(#1)}}
\newcommand{\DerfpN}{\mathcal{D}_{(\varphi,N)}^{\mathit{b}}}
\newcommand{\DfpN}[1][]{\ifblank{#1}{\DerfpN(K_0)}{\DerfpN(#1)}}
\DeclareMathOperator{\Derp}{\mathcal{D}_{\varphi}}
\newcommand{\Dp}[1][]{\ifblank{#1}{\Derp(K_0)}{\Derfp(#1)}}
\newcommand{\DerpN}{\mathcal{D}_{(\varphi,N)}}
\newcommand{\DpN}[1][]{\ifblank{#1}{\DerpN(K_0)}{\DerfpN(#1)}}

\DeclareMathOperator{\eff}{eff}

\DeclareMathOperator{\Ext}{Ext}
\DeclareMathOperator{\et}{\acute{e}t}

\DeclareMathOperator{\Fin}{{Fin}}
\DeclareMathOperator{\FDA}{{FDA}}

\DeclareMathOperator{\Frac}{Frac}

\DeclareMathOperator{\Gal}{Gal}

\DeclareMathOperator{\gr}{gr}
\DeclareMathOperator{\Hm}{H}
\DeclareMathOperator{\HK}{HK}
\DeclareMathOperator{\ho}{h}

\DeclareMathOperator{\Hodge}{Hodge}
\DeclareMathOperator{\Hom}{Hom}

\DeclareMathOperator{\id}{id}

\DeclareMathOperator{\Ind}{Ind}

\DeclareMathOperator{\Map}{Map}

\DeclareMathOperator{\MHS}{MHS}
\DeclareMathOperator{\spMap}{map}%

\DeclareMathOperator{\op}{{op}}
\DeclareMathOperator{\perf}{{perf}}
\DeclareMathOperator{\Perf}{Perf}

\newcommand{\Prl}{{\rm{Pr}^{L}}}
\newcommand{\Prlst}{{\rm{Pr}^{st}}}

\newcommand{\Prlmst}{{\rm{CAlg}(\rm{Pr}^{st})}}
\newcommand{\Prlo}{{\rm{Pr}^{L}_\omega}}
\newcommand{\Prlost}{{\rm{Pr}^{st}_\omega}}
\newcommand{\Prloo}{{\rm{CAlg}(\Prlo)}}
\newcommand{\Prloost}{{\rm{CAlg}(\Prlost)}}

\DeclareMathOperator{\proet}{pro\acute{e}t}

\DeclareMathOperator{\rig}{rig}

\DeclareMathOperator{\RigSm}{RigSm}

\DeclareMathOperator{\Sm}{Sm}
\DeclareMathOperator{\Sp}{Sp}
\DeclareMathOperator{\Spa}{Spa}
\DeclareMathOperator{\Spec}{Spec}

\DeclareMathOperator{\Spf}{Spf}

\DeclareMathOperator{\Sym}{Sym}
\DeclareMathOperator{\Tot}{Tot}

\DeclareMathOperator{\uhom}{\underline{Hom}}

\DeclareMathOperator{\Alg}{Alg}

\DeclareMathOperator{\iCat}{{Cat}_{\infty}}
\DeclareMathOperator{\Ch}{{Ch}}

\DeclareMathOperator{\DA}{{{DA}}}
\DeclareMathOperator{\DAet}{{{DA}_{\et}}}

\DeclareMathOperator{\UDA}{{UDA}}

\DeclareMathOperator{\coMod}{{coMod}}
\DeclareMathOperator{\Mod}{{Mod}}

\DeclareMathOperator{\QCoh}{{QCoh}}

\DeclareMathOperator{\RigDA}{{RigDA}}

\DeclareMathOperator{\cyc}{{cyc}}

\DeclareMathOperator{\SH}{{SH}}

\DeclareMathOperator{\Fun}{{Fun}}
\DeclareMathOperator{\Funadd}{{Fun_{add}}}
\DeclareMathOperator{\Funex}{{Fun_{ex}}}
\DeclareMathOperator{\Funaddo}{{Fun^\otimes_{add}}}
\DeclareMathOperator{\Funaddlax}{{Fun^{lax}_{add}}}
\DeclareMathOperator{\Funexo}{{Fun^\otimes_{ex}}}

\DeclareMathOperator{\FunLo}{{Fun^{\otimes}_{L}}}
\DeclareMathOperator{\FunLlax}{{Fun^{lax}_{L}}}

\newcommand{\heart}[1]{\mathcal{H}_{#1}}

\DeclareMathOperator{\Ho}{Ho} 
\DeclareMathOperator{\fib}{fib}
\DeclareMathOperator{\cof}{cofib}

\newcommand{\LF}[2]{{#1}^{#2}_{\textup{lax}}}
\newcommand{\LFnil}[2]{{#1}^{#2}_{\textup{nil}}}
\newcommand{\DAN}{\DA_N(k)}
\newcommand{\hF}[2]{{#1}^{#2}}

\newcommand{\schop}[1][]{
	\ifblank{#1}{\operatorname{Sch}^{\textup{op}}}{\operatorname{Sch}_{#1}^{\textup{op}}}}
\providecommand{\base}{B}
\providecommand{\uC}{C^{\textup{uni}}}
\providecommand{\into}{\hookrightarrow}
\providecommand{\augm}{\varepsilon}
\providecommand{\free}{\mathrm{Free}}
\providecommand{\aid}{\mathrm{I}}
\usepackage{xspace}
\newcommand{\icat}{$\infty$\nobreakdash-category\xspace}
\newcommand{\icats}{$\infty$\nobreakdash-categories\xspace}
\newcommand{\icategorical}{$\infty$\nobreakdash-catego\-ri\-cal\xspace}

\newcommand{\iops}{$\infty$\nobreakdash-operads\xspace}
\newcommand{\subicat}{sub-$\infty$-category\xspace}
\newcommand{\subicats}{sub-$\infty$-categories\xspace}

\usepackage{array}

\makeatletter %
\newcount\SOUL@minus
\makeatother  %

\begin{document}
	\title{Motivic Monodromy and $p$-adic Cohomology Theories}
	\author{Federico Binda}
 \address{Dipartimento di Matematica ``F. Enriques'' - Universit\`a degli Studi di Milano (Italy)}
 \email{federico.binda@unimi.it}
	\author{Martin Gallauer}
  \address{Mathematics Institute - University of Warwick (United Kingdom)}
  \email{martin.gallauer@warwick.ac.uk}
	\author{Alberto Vezzani}
  \address{Dipartimento di Matematica ``F. Enriques'' - Universit\`a degli Studi di Milano (Italy)}
  \email{alberto.vezzani@unimi.it}

	\thanks{
	The first and third authors are partially supported by the {\it Ministero dell'Universit\`a e della Ricerca} (MUR), 
	project   PRIN-20222B24AY. The second author thanks the Max Planck Institute and the Hausdorff Research Institute for Mathematics in Bonn for their hospitality. }
	
	\begin{abstract}
We build a unified framework for the study of monodromy operators and weight filtrations of cohomology theories for varieties over a local field. As an application, we give a streamlined definition of Hyodo--Kato cohomology without recourse to log-geometry, as predicted by Fontaine, and we produce an induced Clemens--Schmid chain complex.
	\end{abstract}

	\maketitle
	\setcounter{tocdepth}{1}
	\tableofcontents
	
	\section{Introduction}
	\label{sec:introduction}
	\subsection{Clemens--Schmid for Hodge and $\ell$-adic cohomology}
	In complex algebraic geometry, a generically smooth family of proper varieties over a curve $X\to S$ induces, around each point~$0$ of~$S(\C)$, a proper smooth morphism $ U\to\Delta^*$ over the punctured analytic disk, for some open subset $U\subseteq X(\C)$. Up to  pullback along a finite cover of the disk, the relation between the singular cohomology of the (possibly singular) special fiber~$X_0$, equipped with its natural mixed Hodge structure, and the cohomology of a generic fiber~$X_t$ for~$t\neq 0$ is encoded in the strictly exact Clemens--Schmid chain complex
	$$
	\cdots\to \Hm^*(X_0)\to \Hm^*_{\lim}(X)\to \Hm^*_{\lim}(X)(-1)\to \Hm_{2d-*}(X_0)\to \Hm^{*+2}(X_0)\to \cdots
	$$
	where $d$ is the relative dimension of $X/S$ and $\Hm^*_{\lim}(X)$ is the ``limit cohomology'', which is a mixed Hodge structure $V$  equipped with a nilpotent monodromy operator $N\colon V\to V(-1)$,  to be thought of as the limit of the Hodge structures on~$\Hm^*(X_t)$ as $t\to 0$ (see \cite{schmid, steenbrink}). It can also be described by the formula $R\Gamma_{\lim}(X)=R\Gamma(X_0,\Psi(\Q))$ with $\Psi$ the (derived) nearby cycle functor. 
	The monodromy operator~$N$ is the logarithm of the action of a generator for the fundamental group of~$\Delta^*$. If $X_0$ is itself smooth then $N=0$ and $V$ is pure. In general, the exactness of the above chain complex follows from the weight-monodromy conjecture, proved by Saito~\cite{saito-mdhp,g-na}.%
	
	This situation has been translated to the arithmetic setting by work of Grothendieck and Deligne \cite{SGAVII1, weil2} (see also \cite{illusie}): Let $K$ be a local field with finite residue field~$k$ of characteristic~$p$ and let $\ell$ be a prime different from $p$. Then to a generically smooth proper map $X\to S=\Spec(\mcO_K)$ we can attach a Galois representation $\Hm^*(X_{\bar{K}},\Q_\ell)$ which, up to replacing $K$ by a finite extension,  amounts to the datum of a Weil--Deligne representation: that is,  a $\Gal(k)$-representation $V$ (which is the action of a topological generator $\varphi$ of $\Gal(k)$) and a $\varphi$-equivariant morphism $N\colon V\to V(-1)$ which encodes the action of the inertia group (by Grothendieck's monodromy theorem) and its relation to~$\varphi$. Here also one has a formula $R\Gamma(X_K,\Q_\ell)=R\Gamma( X_0,\Psi(\Q_\ell))$ and a Clemens--Schmid chain complex. %
	If $X_0$ is smooth then $N=0$ and $V$ is pure \cite{weil1}. In general, for a smooth projective variety $X/K$,  the weight-monodromy conjecture \cite{del-hodge1} predicts that the weight filtration  coincides (up to a shift) with the monodromy filtration, and hence that the Clemens--Schmid sequence is exact (see e.g.\ \cite[\S 5]{chiar-tsu-CS}).
	\subsection{The $p$-adic situation}
	In the $\ell=p$ situation, Jannsen and Fontaine \cite{Jan87, Fon94b} conjectured the existence of a cohomology theory $\Hm^*_{\HK}(X_K)$ (the Hyodo--Kato cohomology) also equipped with a weight and a monodromy filtration. More precisely, up to a finite field extension of $K$, it should be given by the following data (where $K_0=\Frac(W(k))$ and $\sigma$ is the lift of Frobenius):
	\begin{enumerate}
		\item\label{HK1} a $K_0$-vector space~$\Hm^*_{\HK}(X_K)$;
		\item a $K_0$-linear (``Frobenius'') isomorphism $\varphi\colon \Hm^*_{\HK}(X_K)\otimes_{\sigma}K_0\xto{\sim}\Hm^*_{\HK}(X_K)$;
		\item\label{HK3} a $\varphi$-equivariant (``monodromy'') map   $N\colon \Hm^*_{\HK}(X_K)\to \Hm^*_{\HK}(X_K)(-1)$ where the negative twist corresponds to the multiplication by~$p$, i.e., it satisfies $N\varphi=p\varphi N$;
		\item\label{HKlast} if $\car K=0$, an isomorphism $\Hm^*_{\HK}(X_K)\otimes_{K_0}K\cong \Hm^*_{\dR}(X_K)$.
	\end{enumerate}
	In other words, the HK-cohomology endows $\Hm^*_{\dR}(X_K)$ with a $(\varphi,N)$-module structure over~$K_0$.
	Such a cohomology theory has been defined by various authors, including Hyodo--Kato~\cite{HK94}, Mokrane~\cite{mokrane}, Gro\ss e-Kl\"onne~\cite{GK05}, Beilinson~\cite{Bei-crys}, Ertl--Yamada~\cite{EY19} and Colmez--Nizio{\l}~\cite{CN19}. In the most general case, one can even assume $X_K$ to be a rigid analytic variety (in this case, $\Hm_{\dR}(X_K)$ must be thought of as its overconvergent de\,Rham cohomology). This definition is a somewhat intricate generalization of crystalline cohomology: one can reduce to the case of a suitable model of semi-stable reduction, and use some version of the log-de\,Rham complex of its log-special fiber. It is by no means easy to show that this construction and its comparison with de\,Rham can be made independent of choices, functorially on the generic fiber. %
	Not even the existence (let alone exactness) of the Clemens-Schmid chain complex has been known in this context, apart from the equi-characteristic case (see \cite{chiar-tsu-CS}), in low dimension (\cite{chiar-duke}) or in special situations (\cite{cmn}). 
	
	\subsection{Executive summary}
	In the present paper we explain how the ``limit'' structures in all three situations above are nothing but different manifestations of a single structure existing at the motivic level.
	This involves notably two things: producing motivic monodromy operators (for which we use   homotopy theory of rigid analytic varieties) and identifying their realizations in the three contexts (for which we use additionally, and chiefly, weight arguments).
	
	In particular, we provide a more conceptual definition of Hyodo--Kato cohomology that is ultimately built on the (classical) rigid cohomology of smooth varieties over the residue field (even, smooth proper ones: see \Cref{rmk:mononWCHK}). We do not use log-geometry, nor choices of semi-stable models. As such, our definition only depends on the generic fiber, as predicted by Fontaine~\cite[Remarque 6.2.11]{FontaineRepSem}.
	Moreover, we observe that HK-cohomology can be understood, just as in the more classical contexts, as the cohomology of (a version of) the nearby cycles sheaf.
	Finally, we are able to construct the Clemens--Schmid chain complex, thereby resolving~\cite[Conjecture 7.15]{flach-morin}.
	
	\medskip
	All of this will now be discussed in more detail. %
	In what follows, we let $\RigDA(S)$ (resp.\ $\DA(S)$) be the category of analytic (resp.\ algebraic) \'etale motives with rational coefficients over a base~$S$.\footnote{Intuitively, one should think of them as analytic (resp.\ algebraic) varieties ``modulo'' \'etale localization, and homotopies parametrized by~$\B^1$ (resp.~$\A^1$).}
	
	\begin{rmk}
		Certain limit structures at the motivic level have previously been considered, by Spitzweck~\cite{Spitzweck_lim}, Ayoub~\cite{ayoub-th2} and Levine~\cite{MR2302525}.
		We go beyond these works in identifying (and studying) the structure of monodromy operators, as well as in providing general tools in order to identify their realizations.
		And the application to $p$-adic cohomology theories is entirely new.
	\end{rmk}
	
	\begin{rmk}
		We can motivate the use of motivic homotopy theory by drawing a parallel between the complex punctured disk considered above (in the Hodge situation) and the Fargues--Fontaine curve. Let $\Delta$ be the adic space $\mcY_{(0,\infty]}(\C_p^\flat)$ (following the notation of \cite[Section 12.2]{berkeley}) i.e. the locus $p\neq0$ in~$(\Spf\A_{\inf})^{\rig}$. Let $x_\infty$ be its closed point $\Spa\breve{\Q}_p\to\Delta$ and $\Delta^*$ its complement. 
		In \cite{LBV} it is shown that if $C$ is $\C_p$ or $\C_p^\flat$, one can associate canonically to any variety $X$ over $C$ a $\varphi$-motive $\mcD(X)$ over $\Delta^*$ by means of a functor
		$$
		\mcD\colon \RigDA(C)\to\RigDA_{\varphi}(\Delta^*)
		$$
		so that $\mcD(X)$ is then analogous to the generically smooth family $X(\C)$ considered above, except that it is a \emph{motive}, not an actual variety (in general). Under this analogy, the Hyodo--Kato cohomology is the ``limit'' of the relative de\,Rham cohomology groups $\Hm^*_{\dR}(X_{C'}/C')$ as $x_{C'}\to x_\infty$, considered with its structure of $\varphi$-module.\footnote{Indeed, any $\varphi$-module on $\Delta^*$ (such as the de\,Rham-Fargues-Fontaine cohomology of \cite{LBV}) extends uniquely to $\Delta$, see \cite{kedlaya-m}, \cite[Theorem 13.4.1]{berkeley}.} It is also equipped with a  monodromy operator, as we now explain.%
		
	\end{rmk}
	\subsection{Analytic motives as nearby cycles with monodromy}
	In the sequel, $k$ denotes the residue field of the non-archimedean valued field under consideration.
	We construct a compactly generated~\icat $\DA_N(k)$ whose compact objects are given by pairs $(M,N_M\colon M\to M(-1))$ where $M\in\DA(k)_{\cpt}$ is a compact algebraic motive, and $M(-1)=M\otimes\one(-1)$ denotes its (negative) Tate twist. %
		For weight reasons all such maps are nilpotent. %
		We show directly the following identification (see \Cref{cor:RigDAisDAN}).
		
		\begin{thm}\label{thm:intro1}  Let $C$ be $\C_p$,  $\C_p^\flat$ or the complete algebraic closure of $\C(\!(T)\!)$\footnote{More generally, $C$ may be a complete algebraically closed non-archimedean valued field whose valuation group is $\Q$, e.g. spherical completions of these fields.}. Let $k$ be its residue field. 
				There is an  equivalence of symmetric monoidal \icats   $$\RigDA(C)\simeq\DA_N(k),\qquad M\mapsto (\Psi M,N_{\Psi M}).$$ 
		\end{thm}
		If $M$ is the analytification of an algebraic motive~$M'$ over~$C$ then $\Psi M$ is the motivic nearby cycles of~$M'$, and we interpret~$N_{\Psi M}$ as a \emph{motivic monodromy operator} on these.
		The proof of \Cref{thm:intro1} uses the comparison between rigid and algebraic motives from~\cite{agv}
		and some version of the monadicity theorem of Barr--Beck--Lurie.  We actually prove a more general statement, which is an instance of Koszul duality applied to coefficient systems. For details, see \Cref{thm:summary-Mod-coMod-lax} and \Cref{cor:unip-sameas-monodromy}.
		
		\begin{rmk}
			The equivalence depends \textsl{a priori} on a choice of a pseudo-uniformizer of~$C$.
			If $C$ is $\C_p$ or $\C_p^\flat$, another choice has the sole effect of scaling the monodromy operator by a (uniform) non-zero rational constant: see \Cref{rmk:changepi?} and also \Cref{rmk:choice?} for a choice-free equivalence in this case.
		\end{rmk}
		\Cref{thm:intro1} can be used to produce a canonical ``Frobenius lift'' for rigid analytic varieties. More precisely, let  $\DA_{(\varphi,N)}(\barFp)$ be the \icat of  ``motivic $(\varphi,N)$-modules'', i.e., the compactly generated \icat whose compact objects are given by (note the analogy with \eqref{HK1}--\eqref{HK3} before):
		\begin{enumerate}
			\item a compact algebraic motive $M\in \DA(\barFp)_{\cpt}$;
			\item an equivalence $\alpha\colon M\xto{\sim}\varphi_*M$ (where $\varphi$ is the Frobenius on~$\barFp$);
			\item a morphism $N\colon M\to M(-1)$ together with a homotopy $N\circ\alpha\simeq p\cdot\alpha\circ N$.
		\end{enumerate}
		We then show the following (see \Cref{prop:K0phinew}):
		\begin{thm}\label{thm:intro2}
			Let $C$ be $\C_p$ or $\C_p^\flat$. There is a canonical symmetric monoidal functor
			$$
			\RigDA(C)\to \DA_{(\varphi,N)}(\barFp).
			$$
			Informally, it sends $M$ to $\left(\varphi_*\Psi M\xfrom{\sim}\Psi M\to \Psi M(-1)\right)$ where the first map is induced by  the relative Frobenius.
		\end{thm}
		
		\subsection{(Co)homological realizations}
		These results simplify the task of defining cohomology theories on the generic fiber, starting from (motivic) cohomology theories on the special fiber. More precisely, let $\catD$ be  a symmetric monoidal stable compactly generated \icat. As before, we fix a ``Tate twist'' operator~$(-1)$ on $\catD$ and an endofunctor~$\varphi$, and we let~$\catD_N$ resp.~$\catD_{(\varphi,N)}$ be the \icat of $N$-modules resp. of $(\varphi,N)$-modules in $\catD$, respectively. In the presence of a colimit-preserving, compact-preserving monoidal functor $R\Gamma\colon\DA(k)\to\catD$ we will always take the twisting object to be the image of the Tate object~$\one(-1)$.  \Cref{thm:intro1} and \Cref{thm:intro2} immediately imply the following (see also Corollaries \ref{cor:RGammahat} and \ref{cor:RGammahat2}):
		\begin{cor} Let $C$ be $\C_p$,  $\C_p^\flat$ or the complete algebraic closure of $\C(\!(T)\!)$.
			\begin{enumerate}
				\item Every  functor $R\Gamma\colon\DA(k)\to \catD$   as above extends to
				\[\widehat{R\Gamma}\colon \RigDA(C)\to {\catD_N}, \qquad M\mapsto R\Gamma(\Psi M).\]
				\item If $\car k=p>0$,   every  $\varphi$-equivariant  functor $R\Gamma\colon \DA(k)\to \catD$  as above extends to \[\widehat{R\Gamma}\colon \RigDA(C)\to \catD_{(\varphi,N)}, \qquad M\mapsto R\Gamma(\Psi M).\]  
			\end{enumerate}
		\end{cor}
		By taking the Hodge resp. $\ell$-adic \'etale realization we recover the limit Hodge resp. the $\ell$-adic \'etale cohomology of the generic fiber (for the latter, see in particular \cite[Section 11]{AyoubEt}). In this article, we focus more on the $p$-adic situation, obtaining the following (see \Cref{cor:HKiso}, \Cref{cor:HKisHK} and \Cref{rmk:mononWCHK}).
		\begin{thm}\label{thm:intro3}
			Let $C$ be $\C_p$ or $\C_p^\flat$. The rigid realization
			$$
			R\Gamma_{\rig}\colon \DA(\barFp)\to \mcD(\breve{\Q}_p)
			$$
			induces a symmetric monoidal realization to the derived \icat of $(\varphi,N)$-modules
			$$
			\widehat{R\Gamma}_{\rig}\colon\RigDA(C)\to \mcD_{(\varphi,N)}(\breve{\Q}_p), \qquad\widehat{R\Gamma}_{\rig}(X)=R\Gamma_{\rig}(\Psi X)
			$$
			satisfying:    \begin{enumerate}
				\item \label{HKiso} %
				If $C=\C_p$ there is a canonical equivalence  $$\widehat{R\Gamma}_{\rig}\otimes C\simeq R\Gamma_{\dR}^\dagger.$$
				\item\label{WS} There is a canonical weight spectral sequence converging to $\widehat{\Hm}{}^n_{\rig}(X)$ giving rise to the weight filtration, and the monodromy operator restricts to $$N\colon \gr_i\widehat{\Hm}{}^n_{\rig}(X)\to \gr_{i-2}\widehat{\Hm}{}^n_{\rig}(X)(-1).$$
				\item\label{isotous} There is an equivalence with the ``classical'' Hyodo--Kato cohomology realization $$\widehat{R\Gamma}_{\rig}\simeq R\Gamma_{\HK}.$$
			\end{enumerate}
		\end{thm}
		We note in particular that \eqref{HKiso} and \eqref{isotous} give an alternative and independent proof of the ``classical'' Hyodo--Kato comparison isomorphism. 
		
		\medskip
		The proofs of \eqref{HKiso} and \eqref{WS} are based on the motivic definition of the (overconvergent) de\,Rham cohomology of \cite{vezz-MW} and the presence of a weight structure on $\RigDA(C)$ in the sense of Bondarko, whose heart is generated by motives of smooth and proper varieties \emph{of good reduction}. The latter allows one to detect functors by their restriction to the heart, and gives a way of attaching functorially to any motive over $K$ a (weight) complex with pure terms (giving rise to the classical weight spectral sequence) equipped with a monodromy operator. More precisely, we give the following criterion for a cohomology theory on the generic fiber to come from the special fiber, see \Cref{cor:RGammas}.
		\begin{prop}
			Let $R\Gamma\colon\RigDA(C)\to \catD$ be a functor as above, and let $R\Gamma'$ be its restriction along the generic fiber functor $\xi\colon\DA(k)\to \RigDA(C)$. If $R\Gamma'$ factors through the weight complex functor, then $R\Gamma\simeq R\Gamma'\circ\Psi$.
		\end{prop}
		The condition above is fulfilled whenever a ``purity'' statement \`a la Deligne holds (e.g.\ in all our main examples).
		
		\medskip
		The proof of \eqref{isotous} on the other hand, is based on the observation that the ``monodromy component'' of a realization functor is determined essentially by its value on the monodromy of the \emph{Kummer motive} which is a non-trivial extension of~$\one$ by~$\one(-1)$. The crucial geometric input is then the identification of the Kummer motive as  the $\Hm^1$-term of (the motive of) a Tate curve (\Cref{thm:KummerisTate}). %

		\begin{rmk}
			Even if in this introduction we focus mainly on the ``geometric case'', arithmetic versions of the theorems above (that is, replacing $C$ with a field $K$ which is not algebraically closed) can easily be deduced using Galois descent. One may restate them as they are by restricting to ``unipotent'' motives   over $K$ (as we do in the main body of the text), or by enriching realizations with values in $(\varphi,N,\Gamma)$-modules with $\Gamma=\Gal(K)$. Note also that, as $\RigDA(C)=\varinjlim_{[L\colon K]<\infty}\RigDA(L)$, any compact motive becomes unipotent, up to a finite field extension.
		\end{rmk}
		\subsection{Further applications}
		Not only does the approach above provide a unified, polished approach to monodromy and Hyodo--Kato cohomology, but it also  allows us to deduce new facts on limit cohomologies. For example, we prove a conjecture of Flach and Morin~\cite[Conjecture 7.15]{flach-morin}, see \Cref{cor:CS_realization}.
		\begin{thm}
			Let $\mcX$ be a regular proper flat scheme over $\mcO_C$ of relative dimension $d$. There are two canonical Clemens-Schmid chain complexes (for even and for odd indices)
			$$%
			\cdots \to \Hm^*_{\rig}(\mcX_k)\to {\Hm}_{\HK}^*(\mcX_C)\xto{N} {\Hm}_{\HK}^*(\mcX_C)(-1) \to \Hm_{2d-*}^{\rig}(\mcX_k)(-d-1) \to  \Hm^{*+2}_{\rig}(\mcX_k)\to\cdots 
			$$%
	\end{thm}
	Such complexes are exact whenever the weight-monodromy conjecture holds for $\mcX_C$ (eg. if it is a smooth projective hypersurface over a local field, see \cite{pWM}).
	We remark that for a proper semistable family of varieties over a curve in characteristic~$p$ (the equi-characteristic setting),  the existence of a ``Clemens-Schmid type'' long exact sequence for log-crystalline cohomology was shown by Chiarellotto and Tsuzuki~\cite{chiar-tsu-CS}. The analogous result in the mixed-characteristic setting was originally conjectured in~\cite{chiar-duke} where it is proved for curves and surfaces.
	
	\medskip
	Moreover, the comparison between the Hyodo--Kato cohomology of a variety over~$C$ and the one of its ``motivic tilt'' over~$C^\flat$ becomes straighforward using our construction. This yields a little shortcut in the proof of the main theorem of~\cite{pWM}, see \Cref{cor:shortcut}.
	
	Note that the definition of the Hyodo--Kato cohomology via (the motivic) $\Psi$  still allows one to make explicit calculations (recovering those of Mokrane and Gro\ss e-Kl\"onne). We give some concrete examples  in \Cref{sec:CS}.
	We also point out that our approach is more easily generalized to the relative setting, where one considers motives over a more general base~$S$ (not necessarily a field). 
	Our construction is obviously related to the de Rham--Fargues--Fontaine cohomology (with values in vector bundles over the Fargues--Fontaine curve) of \cite{LBV}. 
	We also expect to be able to deduce a general comparison theorem between the Hyodo--Kato and  the \'etale $p$-adic cohomology for varieties over $K$ (that is, the $C_{pst}$-conjecture) from the crystalline case and the study of periods of the Tate curve. This is the subject of future research.
	
	\subsection{Organization}In \Cref{sec:split-square-zero} we recall some instances of Koszul duality for modules over split square-zero extensions, we prove their relation to monodromy maps and to unipotent  objects in coefficients systems. In \Cref{sec:detection-principles} we prove that the monodromy of the Kummer object determines (in an appropriate sense) the monodromy of any other object. We also recall the definition of hearts of weight structures and prove that exact symmetric monoidal functors on the heart admit a unique extension to the whole \icat. In \Cref{sec:realizations} we  move to our main example of this formalism, namely the \icat of rigid analytic motives. In particular, we  prove \Cref{thm:intro1} in \Cref{sec:RigDA-monodromy}, \Cref{thm:intro2} in \Cref{sec:motivicphiN}  and  \Cref{thm:intro3} in \Cref{sec:dR-realization,sec:hyodo-kato-iso,sec:comp-with-classical}. Finally, in \Cref{sec:CS} we produce the Clemens--Schmid chain complexes, and we sketch some explicit computations for the weight complexes of $\chi M$ and $\Psi M$ in the case of a motive~$M$ of a smooth proper rigid analytic variety with semi-stable reduction, comparing with the classical formulas by Steenbrink, Rapoport--Zink and Mokrane. 
	
	We strongly advise the reader who is interested in $p$-adic cohomologies to skip directly to \Cref{sec:realizations} and refer back to the previous sections whenever needed.

		\subsection*{Acknowledgments}
		We thank Joseph Ayoub for detailed discussions which sparked our interest in Koszul duality and its relation to the motivic monodromy operator.
		We thank Emil Jacobsen for suggesting \Cref{rmk:derivation}, and Pierre Colmez for asking a question that led to the development of \Cref{sec:KummerinRig}.
		Shaul Barkan was helpful in explaining his work on square-zero extensions.
		We  are grateful  to  Bruno Chiarellotto, Farid Mokrane and Wies\l awa Nizio{\l}  for their interest in this project, and to the anonymous referees for a detailed reading and very helpful corrections and suggestions.

		\subsection*{Notation and conventions}
		\label{sec:notation-conventions}%
		We use freely the language of \icats, mostly following the notation of~\cite{lurie, HA}. If $X,Y$ are objects in a (resp.\ stable) \icat, we let $\Map(X,Y)$ (resp.\ $\spMap(X,Y)$) be the mapping space (resp.\ mapping spectrum) and $\Hom(X,Y)$ the set $\pi_0\Map(X,Y)$.
		By a monoidal structure on an \icat we always mean a \emph{symmetric} monoidal structure; similarly for functors between (symmetric) monoidal \icats.
		Even if we don't mention it explicitly, the tensor product in stable monoidal \icats is always assumed to be exact (preserve finite colimits) in both variables separately.
		Similarly, the tensor product in presentable (resp.\ compactly generated) monoidal \icats is assumed to preserve small colimits in both variables separately (resp.\ and restrict to compact objects); in other words, these are objects of $\CAlg(\Prl)$ and $\CAlg(\Prlo)$, respectively.
		If $\catC$ is a monoidal \icat and $A$ an object in $\CAlg(\catC)$ we denote by $\Mod_{A}(\catC)$ or $\Mod_A$  the \icat of $A$-modules. We assume some familiarity with the \icats of (algebraic, rational) \'etale motives $\DAet(S)\colonequals \DAet(S,\Q)=\SH_{\et}(S,\Ch\Q)$. We recall their definition in \Cref{sec:RigDA-monodromy}. See also \cite{ayoub-ICM, AyoubEt,  cd, MR3477640}.
		
		\section{Modules, comodules and monodromy operators}
		\label{sec:split-square-zero} We begin with some categorical background material. The results in this section are probably all known in some form or other: we record them here for lack of appropriate references.
		
		Let $\catC\in\Prloost$  be a compactly generated monoidal stable \icat, and let $A=\one\oplus t[-1] \in \CAlg(\catC)$ be the split square-zero extension of the tensor unit~$\one$ by some~$t[-1]\in\catC$, see \Cref{sec:Mod-laxFix}, equipped with its augmentation $\varepsilon\colon A\to \one$. 
		In \Cref{sec:lax-fixed-points,sec:nilpotent-op} we recall the \icat~$\LF{\catC}{-\otimes t}$ of lax fixed points of the endofunctor $\catC\to \catC$, $X\mapsto X\otimes t$, as well as the full \subicat~$\LFnil{\catC}{-\otimes t}$ of ind-nilpotent operators.
		Its objects are pairs $(X,f)$, where $X$ is an object of~$\catC$ and $f\colon X\to X\otimes t$ is a (respectively ind-nilpotent) morphism.
                It has a natural monoidal structure.
Using this, we describe the \icat~$\Mod_A$ in  relatively simple terms, that we summarize in the following theorem.
\begin{thm}[See \Cref{sta:Mod-fix,sta:Mod-coMod}]
			\label{thm:summary-Mod-coMod-lax}
			Assume that $\catC$ is $\Q$-linear and $\Alt^2(t)\simeq 0$. Let $\bH =  \mathrm{Bar}(A,\epsilon)\simeq\bigoplus_{n\geq 0} t^{\otimes n}$ be the Bar construction of $A$ with its canonical coalgebra structure, and let $\nu$ be its canonical coaugmentation. Then there are equivalences of monoidal \icats and commutative diagrams
			\[
			\begin{tikzcd}[column sep=5em, row sep=3em]
				\Mod_A  \arrow[rd,  shift right , "\varepsilon^*", swap]\arrow[r, leftrightarrow, "\sim"] & \coMod_\bH \arrow[r, leftrightarrow, "\sim"] \arrow[d, shift right, "\mathrm{forget}", swap] & \LFnil{\catC}{-\otimes t} \arrow[ld, swap, "\pi"]  \\
				& \catC \arrow[lu,  shift right, "\varepsilon_*", swap] \arrow[ru, shift right = 2, swap, "\pirn"] \arrow[u, shift right, "\mathrm{cofree}", swap]& 
			\end{tikzcd}\]
			where $\pi$ is the forgetful functor $(X,f)\mapsto X$ and $\pirn$ its right adjoint. 
			
   Letting $\nu^* \colon \catC \to  \coMod_\bH$ be the co-restriction functor of comodules along $\nu$, the equivalences above also render the following diagrams commutative
			\[
			\begin{tikzcd}[column sep=5em, row sep=3em]
				\Mod_A  \arrow[rd,  shift left , "\mathrm{forget}"]\arrow[r, leftrightarrow, "\sim"] & \coMod_\bH \arrow[r, leftrightarrow, "\sim"] \arrow[d, shift left, "\nu_*"] & \LFnil{\catC}{-\otimes t} \arrow[ld, shift left, "\mathrm{fib}(N)"]  \\
				& \catC \arrow[lu,  shift left, "\mathrm{free}"] \arrow[ru, shift left , "(N=0)"] \arrow[u, shift left, "\nu^* "]& 
			\end{tikzcd}\]
			where $(N=0)$ is the functor given by sending $X$  to the pair $(X,0)$, and $\mathrm{fib}(N)$ is the functor sending $(X,f)$ to the fiber of~$f$. 
		\end{thm}
		
		\begin{rmk}
  For details on the Bar construction of an augmented  algebra in a monoidal \icat and its properties, we refer to~\cite[\S\,5.2.2]{HA} and~\cite{DAN-COHEN-HOREV}. 
                The equivalence~$\Mod_A\simeq\coMod_\bH$ fits into the common framework of Koszul duality.
                The further equivalence with~$\LFnil{\catC}{-\otimes t}$ exhibits~$\bH$ as the cofree conilpotent coalgebra on~$t$.
                For more general statements about the correspondence between associative augmented algebras and conilpotent coaugmented coassociative coalgebras, we refer the reader to~\cite[Corollary~3.3.5 and Section~3.3.6]{FG} and   \cite[II, Chapter~6, Section~2.4.1]{Gait-Rozen-II}.
		\end{rmk}

                \begin{rmk}
                The assumption $\Alt^2(t)\simeq 0$ simply means that the canonical transposition of factors on~$t\otimes t$ is homotopic to the identity.
                The reader can find the equivalence $\Mod_A\simeq\LFnil{\catC}{-\otimes t}$ already in~\cite[Proposition~3.2.2]{raskin-dgm} and~\cite[Theorem~2.34]{barkan:square-zero}, without assuming this nor $\Q$-linearity. 
                On the other hand, these assumptions allow us to give self-contained alternative proofs and promote these known equivalences to ones between \emph{monoidal} \icats.
                (We don't claim that they are necessary even to this aim, however.)
                \end{rmk}
		
		\begin{rmk} We are going to apply \Cref{thm:summary-Mod-coMod-lax} to $\catC = \DAet(k;\Q)$. %
			In this case, the relevant square-zero extension~$A$ is the cohomological motive~$\mathsf{M}_k^{\mathrm{coh}}(\G_m)$   of the multiplicative group~$\G_{m}$. The augmentation ideal (suspended once)~$t$ is then simply the Tate object~$\one(-1)$, and the relevant \icat of lax fixed points consists of pairs $(X, X\xrightarrow{N_X} X(-1))$, where we interpret the map $N_X$ as a monodromy operator. 
			
			Note that in this setting, the \icat $\Mod_{\mathsf{M}^{\mathrm{coh}}(\G_m)}$ is also equivalent to the full \subicat $\UDA(k)$ of $\DAet(\G_{m,k}; \Q)$ of so-called unipotent motives, introduced by Ayoub and Spitzweck. We will generalize this description to rather general coefficient systems in \Cref{sec:unipotent-setting}.  %
		\end{rmk}
		
		\subsection{Lax fixed points}
		\label{sec:lax-fixed-points}
		Let $\catC$ be a stable \icat and $T\colon\catC\to\catC$ an exact endofunctor.
		We recall a particular instance of the lax equalizer construction~\cite[\S\,II.1.4]{Nikolaus-Scholze}.
		\begin{cons}
			\label{cons:lax-equalizer}
			Define the \icat of \emph{lax fixed points} of~$T$ as the following pullback in~$\iCat$:
			\begin{equation}
				\label{eq:C^T}
				\begin{tikzcd}
					\LF{\catC}{T}
					\ar[r, "N"]
					\ar[d, "\pi"]
					&
					\Fun(\Delta^1,\catC)
					\ar[d, "{\mathrm{ev}_0,\mathrm{ev}_1}"]
					\\
					\catC
					\ar[r, "{\id, T}"]
					&
					\Fun(\partial\Delta^1,\catC).
				\end{tikzcd}
			\end{equation}
		\end{cons}
		
		\begin{rmk}\label{rmk:map_laxeq}
			\label{rmk:C^T-explicit}
			Thus, objects $\mathcal{X}$ of $\LF{\catC}{T}$ are given by  pairs $(X, f)$, where $X=\pi(\mathcal{X})$ is an object of $\catC$ and $f=N(\mathcal{X})$ is a map $f\colon X\to TX$.
			If $(X,f)$ and $(Y,g)$ are two objects of $\LF{\catC}{T}$, then the space  $\Map_{\LF{\catC}{T}}((X,f), (Y,g))$ is given by the equalizer of the two maps:
			\begin{equation}\label{eq:mapping_space_laxeq}
				\begin{tikzcd}
					\Map_{\catC}(X,Y)
					\arrow[r, shift left, "g_*\id_{\catC} "]
					\arrow[r, shift right, swap, "f^*T "]
					&
					\Map_{\catC}(X,TY).
				\end{tikzcd}
			\end{equation}
		\end{rmk}
		\begin{rmk}
			In~\cite[Definition II.5.1]{Nikolaus-Scholze}, objects in~$\LF{\catC}{T}$ are called ``$T$-coalgebras''.
		\end{rmk}

		\begin{lemma}[{\cite[Proposition~II.1.5]{Nikolaus-Scholze}}]
			\label{sta:C^T-basics}
			Assume that $\catC$ is a stable presentable \icat and that $T$ preserves colimits.
			Then:
			\begin{enumerate}
				\item $\LF{\catC}{T}$ is a stable presentable \icat and \eqref{eq:C^T} is a pullback square in $\Prlst$.
				\item
				\label{it:C^T-basics.2}
				The functor~$\pi$ is conservative.
				\item If $T$ preserves limits then so does~$\pi$.
			\end{enumerate} 
		\end{lemma}

		\begin{cons}
			\label{cons:pi_rho}
			Assume that $T$ preserves products.
			Consider the functor $\rho\colon\catC\to\LF{\catC}{T}$ that sends $X\in\catC$ to $(\prod_{n\geq 0}T^nX, \mathrm{proj})$, where
			\[
			\mathrm{proj}\colon\prod_{n\geq 0}T^nX\to\prod_{n> 0}T^nX
			\]
			is the canonical projection. Using the description given in Remark \ref{rmk:map_laxeq} for the mapping spaces in $\LF{\catC}{T}$, for every 
			$(X,f)\in\LF{\catC}{T}$ we can construct a morphism
			\[
			\eta_{(X,f)}\colon (X,f)\to \rho(X)
			\]
			that on the $n$th component is given by $f^n:X\to T^nX$.
			It is not difficult to show that $\eta\colon\id\to \rho\circ\pi$ defines a natural transformation.
		\end{cons}
		
		\begin{lemma}
			\label{sta:pi_rho}
			The natural transformation~$\eta$ exhibts~$\pi$ as left adjoint to~$\rho$.
		\end{lemma}
		\begin{proof}
			Fix $(X,f)\in\LF{\catC}{T}$ and $Y\in\catC$.
			We need to show that the composite
			\[
			\Map_{\catC}(X,Y)\xto{\rho}\Map_{\LF{\catC}{T}}(\rho(X),\rho(Y))\xto{\eta_{(X,f)}}\Map_{\LF{\catC}{T}}((X,f),\rho(Y))
			\]
			is an equivalence of spaces.
			By \Cref{sta:C^T-basics}, the target is the equalizer of
			\[
			\begin{tikzcd}
				\Map_{\catC}(X,\prod_{n\geq 0}T^nY)
				\arrow[r, shift left, "\mathrm{proj}\circ(-)"]
				\arrow[r, shift right, swap, "T(-)\circ f"]
				&
				\Map_{\catC}(X,\prod_{n>0}T^nY),
			\end{tikzcd}
			\]
			that is, equivalently, of
			\begin{equation*}%
				\begin{tikzcd}
					\prod_{n\geq 0}\Map_{\catC}(X,T^nY)
					\arrow[r, shift left, "\mathrm{proj}"]
					\arrow[r, shift right, swap, "(T(-)\circ f)"]
					&
					\prod_{n>0}\Map_{\catC}(X,T^nY).
				\end{tikzcd}
			\end{equation*}
This may be identified with the limit~$\varprojlim_n\Map_{\catC}(X,T^nY)$ with transition maps $T(-)\circ f$. 
			The indexing diagram has an initial object so that this limit is $\Map_{\catC}(X,Y)$.
			It is easy to see that~$\eta$ provides this identification.
		\end{proof}
		
		We will use the following observation later on.
		\begin{lemma}
			\label{sta:pi-compacts}
			Assume that $\catC\in\Prlst$ and $T$ preserves colimits.
			Let $(X,f)\in\LF{\catC}{T}$.
			If $X\in\catC$ is compact then so is $(X,f)$.
		\end{lemma}
		\begin{proof}
			This follows immediately from the description of mapping spaces in \Cref{rmk:C^T-explicit} and the fact that filtered colimits commute with finite limits.
		\end{proof}

		\subsection{Nilpotent operators}
		\label{sec:nilpotent-op}
		We now want to restrict to those lax fixed points which are (ind\nobreakdash-)nilpotent.
		The discussion here follows~\cite{raskin-dgm} and~\cite{barkan:square-zero}.
		\begin{dfn}%
			\label{dfn:LFnil}
			Assume that $\catC$ is stable and $T\colon\catC\to\catC$ an exact functor.
			We define $\LFnil{\catC}{T}$ as the full \subicat of $\LF{\catC}{T}$ spanned by objects $(X,f)$ with \emph{ind-nilpotent} %
			morphism~$f$, that is, such that
			\[
			X[f^{-1}]:=\colim\left(X\xto{f}TX\xto{f}T^2X\xto{f}\cdots\right)\simeq 0.
			\]
		\end{dfn}
		
		\begin{exm}
			There is an obvious section to $\pi\colon\LFnil{\catC}{T}\to\catC$ which sends~$X$ to~$(X,0)$. %
			We sometimes denote it by $(N=0)$.
		\end{exm}

		\begin{lemma}
			\label{sta:N=0-fib}
			The section $(N=0)$ admits a right adjoint $\fib(N)\colon\LF{\catC}{T}\to\catC$ which is given on objects by
			\[
			(X,f)\ \longmapsto\ \fib(f).
			\]
		\end{lemma}
		\begin{proof}
			Let $(X,f)\in\LFnil{\catC}{T}$ and $Y\in\catC$.
			The mapping space $\Map_{\LF{\catC}{T}}((Y,0),(X,f))$ is described in~\eqref{eq:mapping_space_laxeq} as the fiber of
			\[
			\Map_{\catC}(Y,X)\xto{f_*}\Map_{\catC}(Y,TX)
			\]
			which is nothing but~$\Map_{\catC}(Y,\fib(f))$.
			We leave it to the reader to exhibit this identification as part of the claimed adjunction.
			Or see~\cite[Proposition~2.24]{barkan:square-zero}.
		\end{proof}
		
		\begin{lemma}
			\label{sta:C^Tnil-compactly-generated}
			Assume that $\catC\in\Prlost$ and that $T$ preserves colimits.
			Then $\LFnil{\catC}{T}$ is compactly generated by the image of $(N=0)\colon\catC_{\cpt}\to\LFnil{\catC}{T}$.
		\end{lemma}
		\begin{proof}
			By \Cref{sta:pi-compacts}, the functor $(N=0)\colon\catC\to\LFnil{\catC}{T}$ preserves compact objects.
			(This is because the inclusion $\LFnil{\catC}{T}\into\LF{\catC}{T}$ preserves filtered colimits.)
			It therefore suffices to show that the image of $\catC_{\cpt}$ under this functor generates $\LFnil{\catC}{T}$ under colimits.
			Or, equivalently, that the right adjoint $\fib(N)\colon\LFnil{\catC}{T}\to\catC$ is conservative (\Cref{sta:N=0-fib}).
			But if $\fib(f\colon X\to TX)\simeq 0$ then $f$ is an equivalence as well as ind-nilpotent: $X\simeq X[f^{-1}]\simeq 0$.
			We conclude with \Cref{sta:C^T-basics}.\eqref{it:C^T-basics.2}.
		\end{proof}

		\begin{rmk}
			We will say that $(X,f)$ is \emph{nilpotent} if $f^N $ is homotopic to $0$ for a sufficiently large $N$. By the previous lemma, and the fact that nilpotent objects are closed under finite colimits and retracts, we deduce that compact objects of $\LFnil{\catC}{T}$ are all nilpotent.
		\end{rmk}

		\begin{lemma}
			\label{sta:pi_rho^nil}
			Assume that $\catC\in\Prlost$ and that $T\colon\catC\to\catC$ preserves colimits.
			The functor $\pi\colon\LFnil{\catC}{T}\to\catC$ admits a right adjoint~$\pirn$ given on objects by
			\[
			X\mapsto \left(\bigoplus_{n\geq 0}T^nX,\mathrm{proj}\right).
			\]
		\end{lemma}
		\begin{proof}
			If $(X,f)\in\LFnil{\catC}{T}$ is compact then $X\in\catC$ is compact, by \Cref{sta:C^Tnil-compactly-generated}.
			In that case the natural transformation of \Cref{cons:pi_rho},
			\[
			\eta_{(X,f)}\colon (X,f)\to\rho(X)=\left(\prod_{n\geq 0}T^nX,\mathrm{proj}\right),
			\]
			factors through $(\oplus_{n\geq 0}T^nX,\mathrm{proj})=\pirn(X)$.
			By the equivalence $\Fun^{\omega}(\LFnil{\catC}{T},\LFnil{\catC}{T})\simeq\Fun((\LFnil{\catC}{T})_{\cpt},\LFnil{\catC}{T})$ of~\cite[Proposition~5.3.5.10]{lurie}, the natural transformation $\eta\colon\id\to\pirn\circ\pi$ on the right hand side extends uniquely to one on the left hand side (which we denote by the same symbol).
			
			Now let $(X,f)\in\LFnil{\catC}{T}$ and $Y\in\catC$ arbitrary.
			In order to show that the morphism $\Map_{\catC}(X,Y)\xto{\eta_{(X,f)}^*\circ\pirn}\Map_{\LFnil{\catC}{T}}((X,f),\pirn(Y))$ is an equivalence we write $(X,f)$ as a filtered colimit of compact objects and thereby reduce to $(X,f)$ compact, and even, using \Cref{sta:C^Tnil-compactly-generated}, to $f=0$. In this case, the morphism  $\eta_{(X,f)}^*\circ\pirn$ is the identification between $\Map_{\catC}(X,Y)$ and the fiber of
   $$
   \bigoplus_{n\geq0}\Map_{\catC}(X,T^nY)\xto{\mathrm{proj} }\bigoplus_{n>0}\Map_{\catC}(X,T^nY).
   $$
		\end{proof}

                \begin{rmk}\label{lem:CTNil-vsIndCat}
                Assume that $\catC\in\Prlost$ and that $T\colon\catC\to\catC$ preserves compact objects. The diagram~\eqref{eq:C^T} defining~$\LF{\catC}{T}$ is clearly a diagram in $\Prlost$, and as such one could compute its limit in $\Prlost$. This amounts to the \icat $\Ind\left(\LF{(\catC_{\cpt})}{T}\right)$, where $\catC_{\cpt}$ is the \subicat of compact objects. We have a canonical functor 
				\[\iota\colon\Ind\left(\LF{(\catC_{\cpt})}{T}\right)\to\LF{\catC}{T}\]
				which can be easily seen to be fully faithful.
				In fact, the functor $\iota$   identifies $\LFnil{\catC}{T}$ with a full \subicat of $\Ind\left(\LF{(\catC_{\cpt})}{T}\right)$, so that we have a chain of fully faithful functors
				\[
				\LFnil{\catC}{T}\subseteq \Ind\left(\LF{(\catC_{\cpt})}{T}\right)\subseteq \LF{\catC}{T}.
				\]
				It is immediate to check that the former two \subicats coincide if the following condition holds for a set of compact generators~$X\in\catC$:
				\[
				\Hom_{\catC}(X,T^{n}X)=0, \quad n\gg 0.
				\]
				Indeed, this implies that the monodromy map $Y\to TY$ is nilpotent whenever~$Y\in\catC$ is compact. 
			\end{rmk}

                        \subsection{Monoidal structure}
\label{sec:monoidal-structure}

                When $\catC$ is equipped with a monoidal structure and $T=-\otimes t$ is tensoring with a fixed object~$t\in\catC$ then	there should be a monoidal structure on~$\LF{\catC}{-\otimes t}$ that is informally given by the formula
                \begin{align}
                \label{eq:LEq-tensor}
                  (X,f)\boxtimes(X',f')=(X\otimes X', f\boxtimes f'),\\\intertext{where}\notag
                  f\boxtimes f'\ \colon\ X\otimes X'\xto{f\otimes \id + \id\otimes f'}X\otimes t\otimes  X'.
                \end{align}
                We will presently make this precise and discuss monoidal aspects of the constructions in the preceding sections.

                \begin{cons}
                \label{cons:LEq-monoidal}
                Assume that $\catC^\otimes$ is a monoidal \icat and let $\catT\colon \catC^\otimes  \to \catC^\otimes$ be a lax monoidal endofunctor (see~\cite[A.2]{Nikolaus-Scholze}, \cite[Definition 2.1.2.7]{HA}).
			Consider the oriented fiber product $\catC\overset{\to}{\times}_{\catT}\catC$, i.e.,  the pullback of $\Fun(\Delta^1, \catC) \xto{\mathrm{ev}_1}\catC\xleftarrow{\catT}\catC$ in~$\iCat$. See \cite[\href{https://kerodon.net/tag/01KE}{Tag 01KE}]{kerodon}. %
			Its objects are triples $(X,Y,f)$ with $X,Y$ objects of $\catC$ and $f\colon X\to \catT Y$. Observe that the functor \icat $\Fun(\Delta^1,\catC)$ can be naturally endowed with the structure of monoidal \icat (with respect to the pointwise monoidal structure, see~\cite[Remark~2.1.3.4]{HA}) such that the evaluation functor $\mathrm{ev}_1$ is a cocartesian fibration, in particular monoidal.  Using this we can endow $\catC\overset{\to}{\times}_{\catT}\catC$  with the structure of a monoidal \icat
			\[	
			(\catC\overset{\to}{\times}_{\catT}\catC)^\boxtimes\colonequals \lim\left(\Fun(\Delta^1, \catC)^{\otimes}\xto{ev_1}\catC^{\otimes}\xleftarrow{{\catT}}\catC^{\otimes}\right).
			\]
			Indeed, this limit can also be computed in the \icat $\mathrm{Op}_\infty$ of \iops~\cite[Definition~2.1.4.1]{HA}. (This follows from~\cite[Proposition~2.1.4.6]{HA} and the fact that~$\mathrm{ev}_1$ is a fibration of \iops.)
It then suffices to note that the induced functor $(\catC\overset{\to}{\times}_{\catT}\catC)^\boxtimes\to  \Fin_*$ is again a cocartesian fibration.

Informally, we have $(X,Y,f)\boxtimes(X',Y',f')=(X\otimes X', Y\otimes Y',f\boxtimes f')$ with $f\boxtimes f'$ being the composition $$X\otimes X'\xto{f\otimes f'}{\catT}Y\otimes {\catT}Y'\to {\catT}(Y\otimes Y').$$
                \end{cons}
                        
		\begin{rmk}\label{rmk:monoidal-structure-lax-cat}\label{LF-tensor} 	
                We apply \Cref{cons:LEq-monoidal} in the case where $\catC^{\otimes}\in\Prlmst$ and where ${\catT}$ is the lax monoidal functor ${\catT}= (\one\oplus t) \otimes (-)$ 
                given by tensoring with a split square-zero extension in~$\catC$, cf.\ \Cref{rmk:ssze}. If we let $T$ be the functor $-\otimes t$ we can observe that~$\LF{\catC}{-\otimes t}$ is (equivalent to) a full \subicat of $\catC\overset{\to}{\times}_{\catT}\catC$ (using the functor $(X,f)\mapsto (X,X,\id+f)$).
                As this \subicat contains the unit object and is closed under tensor products, \cite[Proposition~2.2.1.1, Remark~2.2.1.2]{HA} imply that the monoidal structure on the oriented fiber product restricts to a monoidal structure $(\LF{\catC}{-\otimes t})^\boxtimes$.
                Note that it indeed is given by the formula in~\eqref{eq:LEq-tensor}.
		\end{rmk}

                \begin{rmk}
                \label{rmk:monoidal-structure-LFnil}
                The full \subicat $\LFnil{\catC}{-\otimes t}$ of $\LF{\catC}{-\otimes t}$ contains the unit object $(\one,0)$ and is stable under tensor products. Therefore, it inherits the structure of a monoidal \icat from the one on~$\LF{\catC}{-\otimes t}$.
                \end{rmk}
                
                In the rest of this section we work in the setting of \Cref{LF-tensor}.
                We will always endow $\LF{\catC}{-\otimes t}$ and $\LFnil{\catC}{-\otimes t}$ with the respective monoidal structures just constructed.
                \begin{lemma}
                \label{sta:pi-rho-projection-formula}%
                The forgetful functor $\pi\colon\LFnil{\catC}{-\otimes t}\to\catC$ can be promoted to a monoidal functor and the adjoints $\pi\dashv \pirn$ satisfy the projection formula
                \[
                \pirn(Y)\boxtimes (X,f)\isoto\pirn(Y\otimes X).
                \]
                Explicitly, it is given as follows, after applying $\pi$, in $\Ho(\catC)$:
                \[
                \bigoplus_{n\geq 0}t^{\otimes n}\otimes Y\otimes X\xrightarrow[
                \left(\binom{m}{m-n}f^{m-n}
                \right)]{}\bigoplus_{m\geq 0}t^{\otimes m}\otimes Y\otimes X.
                \]
                \end{lemma}
             
                \begin{proof}
                The forgetful functor is the composite of the inclusion $\LFnil{\catC}{-\otimes t}\into\catC\overset{\to}{\times}_{\catT}\catC$ and the projection onto~$\catC$ both of which are monoidal.

                The projection formula is obtained  by  the following composite of unit, monoidal structure and counit of the adjunction:
  \begin{equation}\label{eq:proj_formula}  
  \pirn(Y)\boxtimes (X,f)\xto{\eta}\pirn\pi(\pirn(Y)\boxtimes (X,f))\simeq \pirn(\pi\pirn(Y)\otimes X) \to \pirn(Y\otimes X)
    \end{equation}
  The nilpotent operator~$N$ for the domain is (by \Cref{rmk:monoidal-structure-lax-cat})
  \[
 \bigoplus_{n\geq 0}t^{\otimes n}\otimes Y\otimes X\simeq(\oplus_{n}T^nY)\otimes X\xrightarrow[\left(
 \begin{smallmatrix}
    f&1&0&\cdots\\
    0&f&1&\\
   0&0&f&\\
   \vdots&&&\ddots  \end{smallmatrix}\right)
  ]{} T\left((\oplus_{n}T^nY)\otimes X\right)\simeq\bigoplus_{n>0}t^{\otimes n}\otimes Y\otimes X
  \]
  so that the unit
  \[
  \eta\colon\bigoplus_{n\geq 0}t^{\otimes n}\otimes Y\otimes X\to \bigoplus_{m\geq 0}t^{\otimes m}\otimes \bigoplus_{\ell\geq 0}t^{\otimes \ell}\otimes Y\otimes X
  \]
  on the $m$th component is given by
  \[
  N^m=
  \begin{pmatrix}
    f^m&\binom{m}{m-1}f^{m-1}&\binom{m}{m-2}f^{m-2}&\cdots\\
    0&f^m&\binom{m}{m-1}f^{m-1}\\
    \vdots&&\ddots
  \end{pmatrix}.
  \]
  The counit of the adjunction projects onto the $(\ell=0)$-component thus the explicit description of the projection formula in the statement.
  As this matrix is invertible the lemma is proved.
  \end{proof}
                
               \begin{rmk}\label{lemma:ciucciaN}
                Note that by putting $Y=\one$ we immediately deduce an equivalence \[
                \pirn\ \simeq\ (N=0)\boxtimes\pirn\one.
                \]
                \end{rmk}
                \begin{lemma}
                \label{sta:N-fib-projection-formula}
                The zero section $(N=0)\colon\catC\to\LFnil{\catC}{-\otimes t}$ can be promoted to a monoidal functor and the adjoints $(N=0)\dashv \fib(N)$ satisfy the projection formula
                $$
                \fib(X\xto{f} X\otimes t)\otimes Y \isoto \fib(X\otimes Y\xto{f\boxtimes0} X\otimes Y\otimes t).
                $$
                \end{lemma}
                \begin{proof}
                Let, as in \Cref{LF-tensor}, $\one\oplus t$ be the split square-zero extension and $\catT\colon\catC\to\catC$ the lax monoidal functor given by tensoring with $\one\oplus t$.
                The unit $\one\to\one\oplus t$ produces a natural transformation $\id\to\catT$ which may be viewed as a lax monoidal functor $\catC^\otimes\to\Fun(\Delta^1,\catC)^\otimes$.
                Postcomposing with~$\mathrm{ev}_1$ yields~$\catT$ so that we have produced a lax monoidal functor $\catC^\otimes\to(\catC\overset{\to}{\times}_{\catT}\catC)^\boxtimes$.
                It factors through the full \subicat $\LFnil{\catC}{-\otimes t}$ and gives a (lax) monoidal enhancement of~$(N=0)$.
                But note that postcomposing with the conservative functor~$\pi$ we get the identity on~$\catC$ and we conclude that $(N=0)$ is monoidal.

                To prove the second statement let $(X,f)\in\LFnil{\catC}{-\otimes t}$ and $Y\in\catC$.
                Abbreviating $F=(N=0)$ and $G=\fib(N)$, we need to show that the composite
                \[
                G(X,f)\otimes Y\to GF(G(X,f)\otimes Y)\simeq G\left(FG(X,f)\boxtimes F(Y)\right)\to G\left((X,f)\boxtimes F(Y)\right)
                \]
                is an equivalence.
                By \Cref{sta:C^Tnil-compactly-generated} we may assume $f=0$ so that $G(X,f)=X\oplus (X\otimes t[-1])$ and similarly for the codomain.
                In other words, the map is given by a $2\times 2$-matrix and one can check that it is the identity matrix under this identification.
                \end{proof}
		
		\subsection{Modules over split square-zero extensions}\label{sec:Mod-laxFix} We are now ready to relate the \icat of lax fixed points with a suitable \icat of modules. Fix $\catC\in\Prloost$ a stable, compactly generated monoidal \icat.
		\begin{rmk}
			\label{rmk:ssze}
			Recall from~\cite[\S\,7.3.4]{HA} the notion of a split square-zero extension in~$\catC$.
			We specialize to extensions of the unit~$\one$.
			The free commutative algebra functor induces an adjunction
			\[
			\free\colon \catC\rightleftarrows \CAlg(\catC)_{/\one}\colon \aid,
			\]
			where the right adjoint~$\aid$ is the augmentation ideal functor~\cite[Proposition~7.3.4.5]{HA}.
			The latter induces an equivalence~\cite[Theorem~7.3.4.7]{HA} $\Sp(\CAlg(\catC)_{/\one})\isoto\Sp(\catC)\isoto\catC$ so that we may define a functor as the inverse composed with the infinite loops functor:
			\[
			\one\oplus -\colon \catC\isofrom\Sp(\CAlg(\catC)_{/\one})\xto{\Omega^\infty}\CAlg(\catC)_{/\one}.
			\]
			The commutative algebra $\one\oplus M$ (with its canonical augmentation) is the \emph{split square-zero extension} of~$\one$ by~$M$.
		\end{rmk}
		
		We now give one simple criterion to check that a given augmented algebra is a split square-zero extension.
		For this let $A\in\CAlg(\catC)_{/\one}$ with augmentation ideal~$t[-1]=\aid(A)\in\catC$.
                The inclusion $t[-1]\into A$ and $t[-1]\into\one\oplus t[-1]$, respectively, induce by adjunction two morphisms in $\CAlg(\catC)_{/\one}$:
		\begin{equation}
			\label{eq:free-ssze}
			A\from \free(t[-1])\to \one\oplus t[-1].
		\end{equation}

		\begin{lemma}
			\label{sta:free=ssze}
                        If  $\catC$ is $\Q$-linear and $\Alt^2(t)\simeq 0$, both arrows in~\eqref{eq:free-ssze} are equivalences.
			In particular, $A$ is a split square-zero extension.
		\end{lemma}
		\begin{proof}
			By~\cite[Lemma~7.3.4.11]{HA}, the augmentation ideal functor is conservative.
			By~\cite[Proposition~7.3.4.5]{HA}, we have $I(\free(t[-1]))=\coprod_{n>0}\Sym^n(t[-1])$.
                        By~\cite[Lemma~5.19]{MR4133164}, the symmetric powers may be computed at the level of homotopy categories.
                        Using our assumption on~$t$ we see that $\Sym^n(t[-1])\simeq 0$ for all $n\geq 2$ so that $I(\free(t[-1]))\simeq t[-1]$.
			The claim follows.
		\end{proof}
		
                As shown in \cite{raskin-dgm} and~\cite{barkan:square-zero}, modules over split-square zero extensions are intimately related to nilpotent operators.
                Here we promote their observation to a monoidal equivalence.
                The extra assumptions are probably superfluous for this.
                \begin{prop}
                \label{sta:Mod-fix}
                If $\catC$ is $\Q$-linear and $\Alt^2(t)\simeq 0$, the functor $\fib(N)\colon\LFnil{\catC}{-\otimes t}\to\catC$ is monadic and factors through the following  monoidal equivalence.
                \[
                \begin{tikzcd}
                \LFnil{\catC}{-\otimes t}
                \ar[rr, "\sim"]
                \ar[dr, "\fib(N)" swap]
                &&
                \Mod_{\one\oplus t[-1]}(\catC)
                \ar[ld, "\textup{forget}"]
                \\
                &
                \catC
                \end{tikzcd}
                \]
                \end{prop}
                \begin{proof}
                By \Cref{sta:C^Tnil-compactly-generated}, $\fib(N)$ is conservative and preserves all colimits.
                Its left adjoint $(N=0)$ is monoidal and they satisfy the projection formula, by \Cref{sta:N-fib-projection-formula}.
                It then follows from a version of Barr--Beck--Lurie~\cite[Proposition~5.29]{MR3570153} that we have a monoidal equivalence as claimed with modules over the algebra~$\fib(N)(\one,0)$.
                It remains to identify the latter with the split square-zero extension~$\one\oplus t[-1]$.
                The unit of the adjunction $\fib(N)(\one,0)\to\fib(N)\pirn\pi(\one,0)\simeq\one$ exhibits this algebra as augmented.
                The claim now follows from \Cref{sta:free=ssze}.
                \end{proof}

                \begin{rmk}
                \label{cons:square-zero}
                One can describe the equivalence of \Cref{sta:Mod-fix} informally as follows.
                (We write $A=\one\oplus t[-1]$.)
                Starting with $(X,f)\in\LFnil{\catC}{-\otimes t}$, the corresponding $A$-module is the fiber~$\fib(f)$  with the $A$-module action induced by the projection formula and adjunction:
                \[
                \fib(N)(X,f)\otimes \fib(N)(N=0)(\one)\simeq\fib(N)(N=0)\fib(N)(X,f)\to\fib(N)(X,f).
                \]
                Conversely, an $A$-module~$M$ is sent to the pair $(M\otimes_A\one,M\otimes_A\one\to M\otimes_At)$ where the map is induced by the cofiber sequence $A\xto{\augm}\one\to t$, where $\augm$ denotes the augmentation.

                Note in particular that under this identification, the forgetful functor~$\pi$ corresponds to the base change functor~$\augm^*$ (and thus~$\pirn$ to~$\augm_*$).
			\end{rmk}

		\subsection{Comodules over cofree coalgebras}
		\label{sec:comodules}
		While the contents of this section will not be used in the sequel, we include it to complete the picture described in \Cref{thm:summary-Mod-coMod-lax} and to explain how the equivalence of \Cref{sta:Mod-fix} may be viewed as a form of Koszul duality.

                We continue with the previous setting. Fix $\catC\in\Prloost$ and $t\in\catC$. %
                Let us denote the corresponding split square-zero extension by~$A=\one\oplus t[-1]$, with augmentation~$\augm\colon A\to\one$.
		An important object in~$\Mod_A$ is $\augm_*\one$, that is, $\one\in\catC$ viewed as an $A$-module via the augmentation.
		We have
		\[
		\bH:=\bH(t):=\augm^*\augm_*\one=\one\otimes_A\one=\mathrm{Bar}(A,\augm)\in\catC,
		\]
		the Bar construction of the augmented algebra $\augm\colon A\to\one$ in the sense of~\cite[\S\,5.2.2]{HA}.
		As such it is canonically a commutative bialgebra, by which we mean an~$(\E_\infty,\E_1)$-bialgebra in~$\catC$.\footnote{That is, an object of $\coAlg_{\E_1}(\CAlg(\catC))$, where $\CAlg(\catC)$ has the monoidal structure of~\cite[Example~3.2.4.4]{HA}.}

		\begin{rmk}
			\label{rmk:coMod-monoidal}
			The \icat of (left) $\bH$-comodules admits a monoidal structure such that the forgetful functor $\coMod_\bH\to\catC$ is monoidal.
			Informally, it can be described as follows.
			If $X,Y\in\catC$ with coaction maps $a:X\to \bH\otimes X$, $b:Y\to \bH\otimes Y$ then their tensor product is the object $X\otimes Y$ with coaction map (see e.g.~\cite[Corollaire 1.25]{ayoub-h1} for ordinary categories)
			\[
			X\otimes Y\xto{a\otimes b}(\bH\otimes X)\otimes (\bH\otimes Y)\simeq (\bH\otimes \bH)\otimes (X\otimes Y)\to \bH\otimes (X\otimes Y)
			\]
			where the last map is induced by the multiplication on~$\bH$.
			
			This is made precise in~\cite[Proposition~3.16]{beardsley:coHA} which shows more generally the existence of coCartesian fibrations\footnote{We write $\coMod$ for the \icat of left comodules.}
			\[
			\coMod(\catC)^\otimes\xto{p}\coAlg(\catC)^{\otimes}\xto{q}\E_\infty^\otimes
			\]
			such that pulling back~$p$ along $\bH:\E_\infty^\otimes\to\coAlg(\catC)^\otimes$ produces the monoidal structure on~$\coMod_\bH(\catC)$.
                        It endows the forgetful functor $\coMod_{\bH}(\catC)\to\catC$ with a monoidal structure.
		\end{rmk}

		\begin{rmk}
                        Under the equivalence of \Cref{sta:Mod-fix}, the functor $\varepsilon^*\colon \Mod_A \to \catC$ corresponds to the projection~$\pi$, and thus its right adjoint~$\varepsilon_*$ corresponds to~$\pirn$, see \Cref{cons:square-zero}. 
			By \Cref{sta:pi_rho^nil}, the functor $\pirn$ sends~$\one$ to
			\[
			\pirn\one\simeq\left(\oplus_{n\geq 0}t^{\otimes n},\oplus_{n\geq 0}t^{\otimes n}\xto{\text{proj}}\oplus_{n>0}t^{\otimes n}\right) \in\LFnil{\catC}{-\otimes t}
			\]
			so that
			\[
			\bH\simeq\bigoplus_{n\geq 0}t^{\otimes n}.
			\]
			As a coalgebra, it agrees with the cofree conilpotent coalgebra  on~$t$, in the sense of \cite[\S 3]{FG}.
		\end{rmk}

  Also the commutative algebra structure on $\bH$ is well understood, in the situation we are mostly interested in.

  \begin{prop}
\label{lem:H-free}%
Assume $\catC$ is $\Q$-linear and $\Alt^2(t)\simeq 0$. The canonical inclusion of the summand $t\to \bH$ exhibits the latter as the free commutative algebra generated by~$t$.
\end{prop}
 \begin{proof}
 By~\cite[Proposition~3.1.3.13]{HA}, it suffices to show that for each $n\geq 0$ (using~\cite[Notation~3.1.3.10]{HA}), the induced map~$\Sym^n(t)\to\bH=\oplus_{n\geq 0}t^{\otimes n}$ is equivalent to the inclusion of the summand~$t^{\otimes n}$.
 Explicitly, this is the composite
 \[
 t^{\otimes n}/\Sigma_n\to \bH^{\otimes n}/\Sigma_n\xto{\mu_n} \bH,
 \]
 where~$\Sigma_n$ acts by permuting the factors and~$\mu$ is the multiplication on~$\bH$.
 Now, by our assumption on~$t$, the canonical map $t^{\otimes n}\to t^{\otimes n}/\Sigma_n$ is an equivalence.
 (One may work at the level of homotopy categories, by~\cite[Lemma~5.19]{MR4133164}.)
 It remains to understand the multiplication on~$\bH$.
 We will show that the composite $t^{\otimes n}\to\bH^{\otimes n}\xto{\mu_n}\bH$ factors through the summand~$t^{\otimes n}\into\bH$ and is multiplication by~$n!$ on~$t^{\otimes n}$.
 Given that~$\catC$ is $\Q$-linear this will complete the proof.

 We show that claim by induction on~$n$, with~$n=0,1$ being clear.
 For $n\geq 2$ we use the explicit description of the projection formula in \Cref{sta:pi-rho-projection-formula}.
 It tells us that multiplication
 \[
 \mu_2\colon\pirn\one\otimes\pirn\one\to\pirn\pi\pirn\one\xto{\text{counit}}\pirn\one
 \]
 on the summand $t\otimes\pi\pirn\one$ is multiplication by the ``total degree'':
 \[
 \resizebox{4.4em}{!}{$
\begin{pmatrix}
    0&0&0&\cdots\\
   1&0&0\\
   0&2&0\\
   0&0&3\\
   \vdots&&&\ddots
\end{pmatrix}
$ }\colon \bigoplus_{\ell\geq 0}t\otimes t^{\otimes \ell}\to\bigoplus_{m\geq 0}t^{\otimes m}.
 \]
 (Recall that the counit of the adjunction projects onto the zeroth summand.)
 By induction, the composite
 \[
 \mu_n\colon t^{\otimes n}=t\otimes t^{\otimes (n-1)}\xrightarrow[=(n-1)!]{\mu_{n-1}}t\otimes  t^{\otimes (n-1)}\xrightarrow[=n]{\mu_2}t^{\otimes n}
 \]
 is indeed multiplication by~$n!$.
 \end{proof}

We remark that by the associativity of tensor products, for any $M\in\catC$ we have
$$
\one\otimes_A\augm_*M \simeq (\one\otimes_A\one)\otimes M
$$
inducing an equivalence of endofunctors $\varepsilon^*\varepsilon_*\simeq\bH\otimes-$. This can be promoted to an equivalence of comonads, using the very definition of the coalgebra structure on $\bH=\mathrm{Bar}(A,\varepsilon)=\one\otimes_A\one$.  We include a sketch of a proof of this fact. See also~\cite[Proposition 3.28]{bcn:pd-operads-lie}.

\begin{lemma}\label{sta:comonads}
There is a canonical equivalence of comonads $\varepsilon^*\varepsilon_*\simeq \bH\otimes-$.
\end{lemma}
\begin{proof}
    The construction of the first part of \cite[Corollary 8.4.3]{DAN-COHEN-HOREV} defines a factorization of the functor $\varepsilon^*$ as
    $$
    \Mod_A\xto{\widetilde{\varepsilon}^*}\coMod_{\bH}\xto{\text{forget}}\catC.
    $$
    Since the forgetful functor reflects colimits, we deduce that $\widetilde{\varepsilon}^*$ preserves them and therefore (use e.g.~\cite[Remark~5.5.2.10]{lurie}) admits a right adjoint $\widetilde{\varepsilon}_*$. We then obtain a map of comonads $\varepsilon^*\varepsilon_*\to (\mathrm{forget})\circ(\mathrm{cofree})\simeq \bH\otimes-$ induced by the counit of the adjunction $(\widetilde{\varepsilon}^*,\widetilde{\varepsilon}_*)$. This is easily seen to be an equivalence by what we remarked above.
    \end{proof}
    Given any $(X,f)\in \LFnil{\catC}{-\otimes t}$, we can extend uniquely  the map~$f$   to a coaction map $\hat{f}\colon X\to \bH\otimes X$.
			It is then reasonable to expect that $\LFnil{\catC}{-\otimes t}\simeq\coMod_\bH$ as we now show.
		 \begin{prop}
		 	\label{sta:Mod-coMod}
		 	Assume $\catC$ is $\Q$-linear and $\Alt^2(t)\simeq 0$. The functor $\pi:\LFnil{\catC}{-\otimes t}\to\catC$ is comonadic and factors through the following monoidal equivalence.
		   	\[
		 	 \begin{tikzcd}
                         \LFnil{\catC}{-\otimes t}
			 	\ar[rr, "\sim"]
			  	\ar[dr, "\pi" below]
			  	&&
			  	\coMod_\bH
			  	\ar[dl, "\mathrm{forget}"]
			  	\\
			  	&
			  	\catC
			  \end{tikzcd}
		   	\]
		 \end{prop}

		 \begin{proof}
   We first show that the functor $\pi$ is comonadic. By \Cref{sta:pi_rho^nil}, it is a left-adjoint, and it follows from \Cref{sta:C^T-basics} that it is conservative. As the functor $\LF{\catC}{-\otimes t}\to\catC$ creates limits \cite[Proposition II.1.5(v)]{Nikolaus-Scholze}, we are left to show that the limit $(Y,g)$ in $\LF{\catC}{-\otimes t}$ of any $\pi$-split cosimplicial object $(X^\bullet,f^\bullet)$ in $\LFnil{\catC}{-\otimes t}$, lies in $\LFnil{\catC}{-\otimes t}$. We shall prove that any map $K\to Y[g^{-1}]$ in~$\catC$ with $K$ compact is nullhomotopic. By compactness of~$K$, such a map factors through some~$Y\otimes t^{\otimes n}$. As the composite $K\to Y[g^{-1}]\to X^0[(f^0)^{-1}]\simeq 0$ is null-homotopic, we may choose~$n$ large enough so that the composite map $K\to Y\otimes t^{\otimes n}\to X^0\otimes t^{\otimes n}$ is null-homotopic already. However, by our assumption on the diagram, the second map is a split monomorphism and we deduce that $K\to Y\otimes t^{\otimes n}\to Y[g^{-1}]$ is also null-homotopic, as wanted.

              We then deduce an equivalence of \icats $\LFnil{\catC}{-\otimes t}\simeq\coMod_{\pi\pirn}\simeq\coMod_{\augm^*\augm_*}\simeq\coMod_\bH$ where the second equivalence follows from \Cref{sta:Mod-fix} and the last one from \Cref{sta:comonads}.
              We will now show that this equivalence underlies a monoidal equivalence by exhibiting a monoidal inverse.\footnote{We proved in particular that the functor $\tilde{\augm}^*\colon\Mod_A\to\coMod_\bH$ in the previous proof is an equivalence.
                   At this point we would like to conclude that the equivalence~$\tilde{\augm}^*$ is monoidal, for formal reasons.
                   However, we don't have a reference for that and instead proceed by hand.}
            
                Given an $\bH$-comodule~$M$ we may extract the coaction map $M\to \bH\otimes M$ in~$\catC$.
                This defines a lax monoidal functor $\coMod_{\bH}(\catC)^\otimes\to\Fun(\Delta^1,\catC)^\otimes$.
                By construction, postcomposing with $\mathrm{ev}_0$ yields the forgetful functor, while postcomposing with $\mathrm{ev}_1$ yields the forgetful functor composed with the lax monoidal functor~$\catT'=\bH\otimes-$.
                We deduce a lax monoidal functor $\coMod_{\bH}(\catC)^\otimes\to(\catC\overset{\to}{\times}_{\catT'}\catC)^\boxtimes$ which is automatically monoidal (see \Cref{cons:LEq-monoidal}). 

                As $\bH$ is a free commutative algebra on $t$ by \Cref{lem:H-free}, it is equipped with a canonical map of commutative algebras $\bH\to(\one\oplus t)$ which induces a natural transformation of lax monoidal functors $\catT'\to \catT=(\one\oplus t)\otimes-$ and hence, a monoidal functor $\catC\overset{\to}{\times}_{\catT'}\catC\to \catC\overset{\to}{\times}_{\catT}\catC$. The composition of the (monoidal) functors above $$\coMod_{\bH}(\catC)^\otimes\to (\catC\overset{\to}{\times}_{\catT'}\catC)^\boxtimes\to (\catC\overset{\to}{\times}_{\catT}\catC)^\boxtimes$$  lands inside $(\LFnil{\catC}{-\otimes t})^\boxtimes$ and can be seen to provide an inverse, on underlying \icats, to $\LFnil{\catC}{-\otimes t}\isoto\coMod_\bH$ above.
                 		\end{proof}
                   
                   \begin{rmk}
                        In the previous proof, in order to show that $\pi$ is comonadic, we could have just as well proved that the base change functor~$\augm^*$ induced by the augmentation of the split square-zero extension $\augm\colon A\to\one$ (see  \Cref{sta:free=ssze}) is comonadic (using \Cref{sta:Mod-fix} and \Cref{cons:square-zero}). This can be deduced from \cite[Proposition 3.22]{AkhilGalois} together with \cite[Example 4.5 and Proposition 4.7]{MR3570153}.
                   \end{rmk}
                
		\begin{rmk}
			Similarly, there is a canonical equivalence
			\[
			\LF{\catC}{-\otimes t}\simeq\coMod_{\hat{\bH}(t)},
			\]
			where
			\[
			\hat{\bH}(t)=\pi\rho\one\simeq\prod_{n\geq 0}t^{\otimes n}.
			\]
			Note that the comonadicity of $\pi\colon \LF{\catC}{-\otimes t} \to \catC$ follows  from 
			\Cref{sta:C^T-basics}(2) and \cite[Proposition~II.1.5(v)]{Nikolaus-Scholze}. We can  view this equivalence as defining the cofree coalgebra on~$t$. 
			See  \cite[Chapter 6, Section 2.7]{Gait-Rozen-II} and \cite[3.5]{FG}. %
		\end{rmk}
		
		\subsection{The derived category of \texorpdfstring{$(\varphi, N)$}{(phi,N)}-modules}
		\label{sec:phiNmod}
		We now give an example of a well-known \icat which can also be expressed in terms of modules/comodules/objects with monodromy using the equivalences above. This \icat is the natural target of the Hyodo--Kato realization functor that we will construct in \Cref{sec:dR-realization}.
		
		\begin{notn}
			\label{notn:K_0}
			Throughout this section we fix a perfect field $k$ of characteristic~$p>0$ and we let~$K_0=\mathrm{Frac}(W(k))$ denote the fraction field of its Witt vectors.
		\end{notn}
		
		Recall the following definition (see~\cite[\S\,4.2]{FontaineRepSem} or~\cite[\S\,1.15]{Bei-crys}).
		
		\begin{dfn}\label{dfn:phimod} \label{dfn:D-phi-N}
			Let $\varphi\colon K_0\to K_0$ denote the arithmetic Frobenius (so that it induces the absolute Frobenius $x\mapsto x^p$ on the residue field). 
			\begin{enumerate}
				\item
				A \emph{$\varphi$-module over}~$K_0$ is a finite-dimensional $K_0$-vector space~$D$ equipped with a $\varphi$-semilinear automorphism $\varphi_D\colon D\to D$, called the Frobenius too.
				Morphisms of $\varphi$-modules are $K_0$-linear maps that commute with the Frobenius.
				We denote by~${M}_{\varphi}(K_0)$ the abelian category of $\varphi$-modules over~$K_0$,  by~$\Dfp$ its bounded derived \icat and by $\Dp=\Ind(\Dfp)\in\Prlost$ the corresponding Ind-completion.
				\item
				A \emph{$(\varphi,N)$-module over}~$K_0$ is a $\varphi$-module~$(D, \varphi_D)$ over~$K_0$ equipped with a~$K_0$-linear endomorphism~$N$ such that $N\varphi_D = p \varphi_D N$.
				Again, morphisms of $(\varphi,N)$-modules are $K_0$-linear maps that commute with Frobenius and~$N$. Note that the condition $N\varphi_D = p \varphi_D N$ implies that~$N$ is nilpotent. %
				We denote by~$M_{(\varphi, N)}(K_0)$ the  category of $(\varphi,N)$-modules over~$K_0$,  by~$\DfpN$ its bounded derived \icat and by $\DpN=\Ind(\DfpN)\in\Prlost$ the corresponding Ind-completion.
				\item \label{pure}
				Assume that $k$ is a finite field of cardinality $q=p^a$. We say that a  $\varphi$-module $D$ over $K_0$ is \emph{pure of weight $i$ }if it is finite dimensional and such that each  eigenvalue $\lambda$ of the $K_0$-linear map  $\varphi_D^a\colon D\to D$ is an algebraic integer such that $|\lambda|=q^{i/2}$ under each field inclusion $\bar{\Q}\subset\C$. 
			\end{enumerate}
		\end{dfn}
		
		\begin{exm}\label{ex:phimods}
			\begin{enumerate}
				\item
				The unit for the tensor structure in~$M_{\varphi}(K_0)$ is the 1-dimensional vector space~$K_0$ with endomorphism the Frobenius lift~$\varphi$.
				\item
				The \emph{Tate twist}~$K_0(n)\in M_{\varphi}(K_0)$ for $n\in\Z$ is the same 1-dimensional vector space~$K_0$ with `twisted' endomorphism~$p^{-n}\varphi$.
				\item 
				Let~$D$ be a $\varphi$-module over~$K_0$.
				We denote by~$D(n)$ the tensor product $D\otimes K_0(n)$.
			\end{enumerate}
		\end{exm}
		
		\begin{rmk}
                \label{rmk:MphiN-tensor}
			Both $M_{\varphi}(K_0)$ and $M_{(\varphi, N)}(K_0)$ are $\Q_p$-linear Tannakian categories for which the functor sending~$(D, \varphi_D)$ to~$D\in \Mod_{K_0}$ (resp.~$(D, \varphi_D, N)\mapsto D$) is a fiber functor, that is, a faithful, exact, $\Q_p$-linear tensor functor.
			The tensor product of $\varphi$-modules is the tensor product on underlying $K_0$-vector spaces with Frobenius $\varphi_{D\otimes E}=\varphi_D\otimes\varphi_E$.
			For $(\varphi,N)$-modules, the tensor product on underlying $\varphi$-modules is as just described, and $N_{D\otimes E}=N_D\otimes\id+\id\otimes N_E$, see also \Cref{rmk:MphiN=LF} below.
		\end{rmk}

		\begin{rmk}
			\label{rmk:MphiN=LF}
			With this notation we may describe a $(\varphi,N)$-module as a $\varphi$-module~$D$ together with a morphism $N\colon D\to D(-1)$.
			In other words, $M_{(\varphi,N)}(K_0)$ is the category of lax fixed points of the functor~$D\mapsto D(-1)$ \emph{computed in $1$-categories}.
			In fact, this can be promoted to an equivalence of tensor categories, with the tensor structure on the lax fixed points of \Cref{LF-tensor}.
		\end{rmk}

                \begin{rmk}
                The tensor product described in \Cref{rmk:MphiN-tensor} may be derived (see~\cite[Proposition~3.2]{Drew:verdier} and~\cite[Corollary~4.8.1.14]{HA}) so that $\Dp$ and $\DpN$ become objects in~$\Prloost$.
		\end{rmk}

                \begin{rmk}
                \label{rmk:phi^*-K0-vs}
                The ring automorphism $\varphi\colon K_0\to K_0$ induces an adjoint equivalence of monoidal functors $\varphi^*\dashv\varphi_*$ on $K_0$-vector spaces (as well as on $\mcD(K_0)$), where $\varphi^*(D)=D\otimes_{K_0,\varphi}K_0$ can be identified with the vector space~$D$ on which $K_0$ acts through $\varphi^{-1}$.
                Similarly then, $\varphi_*(D)$ is the vector space~$D$ on which $K_0$ acts through~$\varphi$.
                It follows that a $\varphi$-module is the datum of a finite-dimensional $K_0$-vector space~$D$ equipped with a $K_0$-linear isomorphism $D\isoto\varphi_*(D)$.
                \end{rmk}
                
		\begin{rmk}
			\label{rmk:phi-modules-coh-dim}
			The abelian category~$M_{\varphi}(K_0)$ has cohomological dimension~$\leq1$.
			The Ext$^i$-groups (for $i=0,1$) from~$(D,\varphi_D)$ to~$(E,\varphi_E)$ are the cohomology groups of the 2-term complex (in cohomological degrees~$[0,1]$)
			\begin{equation}
				\label{eq:RHom-phi}
				\Hom_{K_0}(D,E)\xto{\delta}\Hom_{K_0}(D,\varphi_*E)
			\end{equation}
			where the map~$\delta$ is given by $\delta(\xi)=\varphi_E\xi-\xi\varphi_D$, see~\cite[\S\,1.14]{Bei-crys} (or~\cite[Lemma~2.3]{DN}).
			
                      Suppose that $k=\F_q$. Then, note that the $\Q_p$-vector space $\Hom_{K_0}(D,E)$ is finite-dimensional. Hence, as soon as there are no non-zero morphisms $D\to E$ in~$M_{\varphi}(K_0)$ then there are no non-zero extensions from~$D$ to~$E$ either, as in that case the map~$\delta$ is injective hence bijective.
			In particular, if $D$ and $E$ are pure of different weights, then $\Ext^{i}(D,E)=0$ for all~$i$. 
		\end{rmk}

		\begin{prop}
			\label{prop:DphiN}
			Let  $\hF{\mathcal{D}(K_0)}{\varphi}$ be  the \icat of fixed points for the functor~$\varphi^*$, that is, the equalizer of the pair of functors $\id,\varphi^*\colon \mathcal{D}(K_0)\to\mathcal{D}(K_0)$, computed in $\Prloost$ and let $T\colon \Dp\to\Dp$ be the endofunctor $D\mapsto D(-1)$. There are canonical equivalences in $\Prloost$:
			$$\Dp\isoto \hF{\mathcal{D}(K_0)}{\varphi}\qquad \DpN\isoto\LFnil{\Dp}{T}.$$
		\end{prop}
		\begin{proof}
                Restricting to compact objects, the first equivalence amounts to an equivalence between $\Dfp$ and the equalizer of the pair of functors $\id,\varphi^*\colon \mathcal{D}^b(K_0)\to\mathcal{D}^b(K_0)$ computed in monoidal \icats.
                Instead of the latter equalizer we may just as well consider the equalizer of $\varphi_*$ and $\varphi_*\varphi^*\simeq\id$.
                Then, the first equivalence  follows  from the explicit description of the mapping complex in  $\Dfp$ provided by~\cite[\S\,1.14]{Bei-crys}, as in \Cref{rmk:phi-modules-coh-dim}. 
			
			By \Cref{rmk:MphiN=LF} we have an essentially commutative square of abelian tensor categories and lax monoidal functors
			\[
			\begin{tikzcd}
				M_{(\varphi,N)}^\boxtimes
				\ar[r]
				\ar[d]
				&
				(M_\varphi^{\Delta^1})^\otimes
				\ar[d, "\mathrm{ev}_1"]
				\\
				M_\varphi^\otimes
				\ar[r, "\mathbf{T}"]
				&
				M_\varphi^\otimes
			\end{tikzcd}
			\]
			where $\mathbf{T}$ is given by tensoring with the split square-zero extension~$K_0\oplus K_0(-1)$, cf. \Cref{LF-tensor}.
			The arguments in~\cite{Drew:verdier} %
			allow us to apply $\mathcal{D}^b$ to get a commutative square in \iops:
			\begin{equation}
				\label{eq:DfpN-oriented-fiber-product-operads}
				\begin{tikzcd}
					\DfpN^\boxtimes
					\ar[r]
					\ar[d]
					&
					\mathcal{D}^b(M_\varphi^{\Delta^1})^\otimes
					\ar[d, "\mathrm{ev}_1"]
					\\
					\Dfp^\otimes
					\ar[r, "\mathbf{T}"]
					&
					\Dfp^\otimes.
				\end{tikzcd}
			\end{equation}
			We observe that the right vertical functor factors through monoidal functors
			\[
			\mathcal{D}^b(M_\varphi^{\Delta^1})^\otimes\to(\Dfp^{\Delta^1})^\otimes\xto{\mathrm{ev}_1}\Dfp^\otimes
			\]
			with the pointwise monoidal structure on the middle term.
			From this we deduce a lax monoidal functor $\DfpN^\boxtimes\to (\Dfp\overset{\to}{\times}_{\mathbf{T}}\Dfp)^\boxtimes$  which factors through~$(\LF{\Dfp}{T})^\boxtimes$.
			We will first prove that the induced functor $\Phi\colon \DfpN\to\LF{\Dfp}{T}$ on underlying \icats is fully faithful.
			
			For this, let $D,E$ be bounded complexes of $(\varphi,N)$-modules.
			By~\cite[1.15]{Bei-crys} (or~\cite[Lemma 2.5]{DN}) together with \Cref{rmk:phi-modules-coh-dim}, the  complex~$\spMap_{\DfpN}(D,E)$ is the fiber of
			\[
			\spMap_{\Dfp}(D,E)\xto{(N_E)_*-(N_D)^*}\spMap_{\Dfp}(D,E(-1))
			\]
			which identifies with the mapping complex in~$\LF{\Dfp}{T}$, see~\eqref{eq:mapping_space_laxeq}.
			We conclude that~$\Psi$ is fully faithful.
			
			It follows from~\cite[Proposition~5.3.5.11]{lurie} that the functor
			\[
			\Ind(\Phi)\colon\Ind(\DfpN)\to\Ind(\LF{\Dfp}{T})
			\]
			on Ind-objects remains fully faithful.
			We now claim that the target of this functor is nothing but~$\LFnil{\Dfp}{T}$.
			Indeed, for any $D\in M_\varphi(K_0)$ we observe that for $n\gg 0$, there are no non-zero morphisms $D\to D(n)$ in $M_{\varphi}(K_0)$.
			Now the claim follows from \Cref{lem:CTNil-vsIndCat}.
			
			Finally, by \Cref{sta:C^Tnil-compactly-generated}, $\Ind(\Phi)$ is essentially surjective and thus an equivalence.
			As $\DfpN$ is idempotent complete, it follows that $\Phi$ is an equivalence too. 
			To show that $\Ind(\Phi)$ is monoidal, we may first restrict to compact objects, and then compose with the conservative monoidal forgetful functor $(\LF{\Dfp}{T})^{\boxtimes}\to \Dfp^{\otimes}$.
			By construction, this composite $(\DfpN)^{\boxtimes}\to\Dfp^{\otimes}$ is nothing but the left vertical arrow in~\eqref{eq:DfpN-oriented-fiber-product-operads} which is indeed monoidal.
			This completes the proof.
		\end{proof}

                \begin{rmk}
                \label{rmk:choice-of-phi^*}
                The diligent reader will have wondered why in \Cref{prop:DphiN} we chose fixed points for the functor~$\varphi^*$ and not~$\varphi_*$, given that the latter is the one actually used in relating the fixed points to $\varphi$-modules.
  The reason is that the natural lift of a $\varphi$-module structure to algebraic and rigid-analytic geometry is given by the relative Frobenius - which is a natural transformation $\id\to \varphi^*$.
  Our chosen convention will therefore lead to a more streamlined construction of realization functors in \Cref{sec:realizations}.
  In any case, as remarked in the proof of \Cref{prop:DphiN}, the resulting \icats of fixed points are canonically equivalent.
  An isomorphism $D\isoto \varphi_*D$ corresponds to its inverse isomorphism $D\isoto\varphi^*D$.
                \end{rmk}
		
		\subsection{Unipotent coefficient systems}
		\label{sec:motivic-monodromy}
		\label{sec:unipotent-setting}Let
		$q\colon\G_{m,S}\to S$ be the  multiplicative group over a scheme $S$. 
		A result of Spitzweck~\cite[Corollary~15.14]{spitzweck:thesis} gives a description of the full \subicat of~$\DAet(\G_{m,S};\Q)$ generated under colimits by ``constant'' motives, in terms of   modules over the ring~$q_*\one$.
		
		We now reinterpret this result in the more general setting of an arbitrary  compactly generated \emph{coefficient system} $C:\schop[\base]\to\Prloost$ over some base scheme~$\base$ (that we leave implicit most of the time) in the sense of \cite[Definition 3.1]{cs-exponentiation}. 
		This is an \icategorical analogue of a stable homotopy 2-functor in the sense of~\cite{ayoub-th1} or a motivic triangulated category in the sense of~\cite{cd}, and therefore underlies a six-functor formalism.
		The reader should feel free to take $C=\DAet(-;\Q)$, the \'{e}tale-local $\A^1$-derived \icat of motivic sheaves with rational coefficients, as studied in~\cite{ayoub-th2,AyoubEt}.
		This is the universal example to which our discussion in the present section applies (\Cref{rmk:universal-Q-et}). %
		Other pertinent examples are  $\ell$-adic sheaves, holonomic $\mathcal{D}$-modules (the ``de\,Rham coefficient system''), complex-analytic sheaves (the ``Betti coefficient system'') or mixed Hodge modules. See \cite[\S\,3.2]{cs-exponentiation} and \cite{cs-intro} for other examples and further discussion.
		
		\begin{dfn}
			\label{dfn:unipotent-motives}%
			Let $S$ be a ($\base$-)scheme.
			Let $\G_m=\G_{m,S}$ be the multiplicative group over~$S$ with structure morphism $q\colon\G_m\to S$. 
			We denote by $\uC(S)$ the smallest full \subicat of $C(\G_{m,S})$ that contains all constant objects~$q^*M$ (for $M\in C(S)$) and is closed under colimits.
			$\uC(S)$ is the \icat of \emph{unipotent coefficients over~$S$}.
			For   $C=\DAet(-;\Q)$, we denote this \icat by $\UDA(S)$.
		\end{dfn}
		
		\begin{rmk}
			By construction, $\uC(S)$ is a compactly generated stable \subicat of~$C(\G_m)$.
			As $q^*\colon C(S)\to C(\G_m)$ is monoidal and the tensor product in $C(\G_m)$ commutes with colimits in both variables, it follows that the monoidal structure on~$C(\G_m)$ restricts to~$\uC(S)$.
		\end{rmk}
		
		\begin{rmk}
			\label{rmk:ind-unipotent}
			In \Cref{dfn:unipotent-motives}, we are allowing, instead of ``extensions'' between constant coefficients in~$C(\G_m)$, arbitrary colimits of such.
			Strictly speaking, the coefficients in~$\uC(S)$ should therefore be called \emph{ind-}unipotent.
			For the sake of consistency with the literature we stick to the shorter terminology nonetheless.
		\end{rmk}

		\begin{prop}
			\label{sta:uC-Mod}
			The functor~$q^*\colon C(S)\to \uC(S)$ factors through the following equivalence of monoidal \icats.
			\begin{equation}
				\label{eq:uC-Mod}
				\begin{tikzcd}
					\Mod_{q_*\one}(C(S))
					\ar[rr, "{\sim}"]
					&&
					\uC(S)
					\\
					&
					C(S)
					\ar[lu, "{\textup{free}}"]
					\ar[ru, "q^*" below]
				\end{tikzcd}
			\end{equation}
		\end{prop}
		\begin{proof}
			As $q^*$ preserves compact objects between compactly generated \icats, the right adjoint~$q_*$ preserves filtered colimits hence, by exactness, all colimits.
			Fix $M\in C(S)$ and let us prove that the canonical projection map
			\[
			M\otimes q_*N\to q_*(q^*M\otimes N)
			\]
			is an equivalence for all~$N\in \uC(S)$.
			Since all functors in sight preserve colimits, we may assume that $N=q^*N'$ for some $N'\in C(S)$.
			
			Consider the inclusion $j\colon\G_{m}\into\P^1:=\P^1_S$ with closed complement $i\colon S\amalg S\xto{0,\infty}\P^1$.
			We denote by $\pi\colon\P^1\to S$ the structure map so that $q=\pi\circ j$.
			By the proper projection formula (applied to $\pi$) we reduce to showing that the canonical map
			\[
			\pi^*M\otimes j_*j^*\pi^*N'\to j_*j^*(\pi^*M\otimes\pi^*N')
			\]
			is an equivalence.
			By the localization triangle, this is true if and only if the similar map for the closed complement is an equivalence:
			\[
			\pi^*M\otimes i_!i^!\pi^*N'\to i_!i^!(\pi^*M\otimes\pi^*N').
			\]
			This follows from the purity isomorphism of~\cite[\S\,1.6.1\,ff]{ayoub-th1}. %
			
			Finally, by definition of~$\uC(S)$, the functor $q_*\colon\uC(S)\to C(S)$ is conservative.
			The result now follows from a monoidal version of Barr--Beck--Lurie~\cite[Proposition~5.29]{MR3570153}.
		\end{proof}
		
		\begin{rmk}
			The equivalence in \eqref{eq:uC-Mod} is the composite
			\[
			\Mod_{q_*\one}(C(S))\xto{q^*}\Mod_{q^*q_*\one}(\uC(S))\to\uC(S)
			\]
			where the second functor is base change along the counit of the adjunction~$q^*q_*\one\to \one$.
			A quasi-inverse takes an object $M\in\uC(S)$ to $q_*M$ with the module structure induced by the lax monoidality of~$q_*$ so that the action map is
			\[
			q_*\one\otimes q_*M\to q_*(\one\otimes M)\simeq q_*M.
			\]
		\end{rmk}
		
		\begin{rmk}
			As a consequence of \Cref{sta:uC-Mod} we see that the construction $S\mapsto\uC(S)$ underlies another coefficient system $\uC:\schop[\base]\to\Prloost$.
			Moreover, the functors $q^*\colon C(S)\to \uC(S)$ assemble to a morphism of coefficient systems $C\to \uC$.
		\end{rmk}

		We now show that in good cases the algebra $q_*\one$ is in fact a split square-zero extension. This will allow us to apply the results of \Cref{sec:Mod-laxFix,sec:comodules}.

		\begin{notn}
			\label{notn:q_*1-split}
			We choose the ``universal section'' $1\colon S\to\G_m$. 
			Then the cofiber sequence in $C(S)$,
			\[
			q_*\one\xto{\tilde{1}}\one\to t,
			\]
			splits where the first arrow $\tilde{1}:q_*q^*\one \to q_*1_*1^*q^*\one\simeq\one$ is the adjunction unit. The  cofiber~$t$ is the Tate twist~$\one(-1)$ in~$C(S)$. Note that~$\tilde{1}$ is a morphism of commutative algebra objects. 
		\end{notn}
		
		\begin{assu}
			\label{hyp:Q-et}
			We will now assume that the coefficient system~$C$
			\begin{enumerate}
				\item is $\Q$-linear, that is, it takes values in $\Prloost_{\Mod_{\Q}/}$; and
				\item satisfies \'{e}tale descent.
			\end{enumerate}
		\end{assu}
		
		\begin{rmk}
			\label{rmk:universal-Q-et}
			The coefficient system~$\DAet(-;\Q)$ is the universal example of a coefficient system satisfying \Cref{hyp:Q-et}.
			This follows from the universality of stable motivic $\A^1$-homotopy theory~\cite{cs-universal} or can also be proved analogously.
			The other examples mentioned at the beginning of \Cref{sec:unipotent-setting} also satisfy \Cref{hyp:Q-et}.
		\end{rmk}

		\begin{prop}\label{prop:Gmissplit}
			Under \Cref{hyp:Q-et}, $\tilde{1}:q_*\one\to\one$ in $C(S)$ is the split square-zero extension $\one\oplus\one(-1)[-1]$.
			Moreover, $q_*\one$ is the free commutative algebra on $\one(-1)[-1]$.
		\end{prop}
		\begin{proof}
   \Cref{hyp:Q-et} buys us that the canonical transposition on~$t^{\otimes 2}$ is homotopic to the identity, see for example~\cite[Lemme 11.5]{AyoubEt}. We conclude by applying the criterion of \Cref{sta:free=ssze}.
		\end{proof}
		\Cref{sta:Mod-fix} together with \Cref{sta:uC-Mod} and \Cref{prop:Gmissplit} imply the following.
		\begin{cor}
			\label{cor:unip-sameas-monodromy}{Under \Cref{hyp:Q-et}}	the functor $1^*\colon \uC(S) \to C(S)$ factors through an equivalence of monoidal \icats
			\[\begin{tikzcd}
				\uC(S)
				\ar[rr, "\sim" ]
				\ar[dr, "1^*" below]
				&&
				\LFnil{C(S)}{-\otimes\one(-1)}
				\ar[dl, "\pi"]
				\\
				&
				C(S)
			\end{tikzcd}
			\]
		\end{cor}
		\begin{rmk}
                \label{rmk:tabsummary}
                Consider the section-retraction $S\xto{1}\G_{m,S}\xto{q}S$. Let $\overline{1}_*$ be a right adjoint to $1^*\colon \uC(S) \to C(S)$ and note that $q_*\overline{1}_*=\id$.
                We may associate to it an augmentation (resp.\ coaugmentation)
                \begin{align*}
                  \tilde{1}\colon q_*q^*\one_S\to q_*\overline{1}_*1^*q^*\one_S=\one_S,&&\tilde{q}\colon\one_S=1^*q^*q_*\overline{1}_*\one_S\to 1^*\overline{1}_*\one_S.
                \end{align*}
                We can summarize the equivalences of \icats studied so far (see \Cref{sta:Mod-fix}, \Cref{sta:Mod-coMod}, \Cref{cor:unip-sameas-monodromy}) and the corresponding adjunctions in the following table.  Note that by \Cref{lemma:ciucciaN} we can identify the right adjoint functor $\pirn$ with $(N=0)\boxtimes\pirn\one$.
                \renewcommand*{\arraystretch}{1.5}
			\begin{table}[ht]
				\centering
				\begin{tabular}{| >{\centering\arraybackslash}m{3cm} | >{\centering\arraybackslash}m{3cm} | >{\centering\arraybackslash}m{3cm} | >{\centering\arraybackslash}m{3cm} |}
                                  \hline
					$\uC(S)$&$\Mod_{q_*q^*\one_S}$&$\LFnil{\catC}{-\otimes \one(-1)}$&$\coMod_{1^*\overline{1}_*\one_S}$\\
                                  \hline
&&&\\[-2.5ex]
                                  {$\begin{aligned}
                                     q^*&\dashv q_* \\
                                     1^*&\dashv \overline{1}_*
                                   \end{aligned}$}
                                  &
                                    {$\begin{aligned}
                                       \textup{free} &\dashv \textup{forget}\\
                                       \tilde{1}^*&\dashv\tilde{1}_*
                                     \end{aligned}$}
                                  &
                                    {$
                                    \begin{aligned}
                                      (N=0) &\dashv \fib(N)\\
                                      \pi&\dashv\pirn
                                    \end{aligned}$}
                                  &
                                    {$
                                    \begin{aligned}
                                      \tilde{q}^*&\dashv\tilde{q}_*\\
                                      \textup{forget}& \dashv \textup{cofree}
                                    \end{aligned}$}\\\hline
				\end{tabular}
                                \end{table}
                                \renewcommand*{\arraystretch}{1}
		\end{rmk}
		
		\section{Detection principles for realization functors}
		\label{sec:detection-principles}
		In subsequent sections we will be interested in (realization) functors $\LFnil{\catC}{-\otimes t}\to\catD$ into various \icats~$\catD$.
		In particular, we will want to be able to say when two such functors are equivalent.
		In the present section we develop two tools to detect such functors in simpler terms.
		
		\subsection{The Kummer object}
		We continue with the notation of \Cref{sec:split-square-zero}, with $\catC^{\otimes}\in\Prloost$ and $A=\one\oplus t[-1]$ a split square-zero extension in~$\catC$.
		\begin{cons}\label{def:kummerobj}
			Consider the canonical split inclusion $t[-1]\into A=\fib(N)(\one)$ in~$\catC$ (recall that $\fib(N)$ denotes the right adjoint to the section $\catC\to \LFnil{\catC}{-\otimes t}$ sending $X$ to~$(X,0)$).
			By adjunction, it induces a map $(t[-1],0)\to \one$ in $\LFnil{\catC}{-\otimes t}$ and we define the \emph{Kummer object} $\mcK:=\mcK_t\in\LFnil{\catC}{-\otimes t}$ to be the cofiber of this map.
			Applying~$\pi$, note that we get a cofiber sequence in~$\catC$,
			\begin{equation}
				\label{eq:Kummer-underlying}
				t[-1]\to\one\to\pi(\mcK).
			\end{equation}
			By construction, the first map is the composite in the bifiber sequence $t[-1]\to A\to\one$ in~$\catC$ hence homotopic to~$0$,
			so that $\pi(\mcK)\simeq\one\oplus t$.
		\end{cons}
		
		\begin{lemma}
			\label{lem:monodromy_Kummer_motive}
			\label{rmk:monodromy_Kummer_motive}
			Let $\mcK=(\pi(\mcK),f_{\mcK})\in\LFnil{\catC}{-\otimes t}$ be the Kummer object.
			In $\Ho(\catC)$ we have $\pi(\mcK)=\one\oplus t$, and
			\[
			f_{\mcK}\colon \one\oplus t\xto{
				\left(\begin{smallmatrix}
					0&\id_t\\0&0
				\end{smallmatrix}\right)}t\oplus t^{\otimes2}.
			\]
		\end{lemma}
		\begin{proof}
			Recall that $\Map_{\LFnil{\catC}{-\otimes t}}((t[-1],0),(\one,0))$ is the fiber of the zero map
			\[
			\Map_{\catC}(t[-1],\one)\xto{0}\Map_{\catC}(t[-1],t).
			\]
			The map $(t[-1],0)\to (\one,0)$ used in defining the Kummer object in this description is given by the composite $t[-1]\to A\to \one$ and a loop in $\Map_{\catC}(t[-1],t)$ based at~$0$, that is, an element of
			\[
			\pi_1\Map_{\catC}(t[-1],t)=\pi_0\Map_{\catC}(t,t).
			\]
			By construction, this is the identity~$\id_{t}$.
			
			The functor $\LFnil{\catC}{-\otimes t}\to\catC^{\Delta^1}\to\Ho(\catC)^{\Delta^1}$ applied to the cofiber sequence defining the Kummer object gives rise to a commutative diagram in the homotopy category
			\[
			\begin{tikzcd}
				\one
				\ar[r]
				\ar[d, "0"]
				&
				\one\oplus t
				\ar[r]
				\ar[d, "f_{\mcK}"]
				&
				t
				\ar[d, "0"]
				\\
				t
				\ar[r]
				&
				t\oplus t^{\otimes 2}
				\ar[r]
				&
				t^{\otimes 2}
			\end{tikzcd}
			\]
			which shows the vanishing of the three entries in the matrix in the statement.
			For the remaining entry note that $f_{\mcK}$ factors through~$t$.
			As a morphism $\cof(t[-1]\xto{0}\one)\to t$ it is described as an element in the fiber of
			\[
			\Map_{\catC}(\one,t)\xto{0}\Map_{\catC}(t[-1],t),
			\]
			namely $0\colon\one\to t$ together with the loop~$\id_t$ from before.
			This completes the proof.
		\end{proof}

		\begin{rmk}
			\label{rmk:Kummer-nonsplit}
			Note that $\one$ is not a direct factor of $\mcK$ in $\LFnil{\catC}{-\otimes t}$ unless $t\simeq 0$, and similarly  there is no natural map from $(N=0)(t)$ to $\mcK$ as the two compositions $t\to t\oplus t^2$ give different results (zero in one case, $(\id,0) $ in the other).
			
			More generally, if $c,d\in\catC$ then the map induced by the forgetful functor
			\[
			\pi\colon\Map_{\LFnil{\catC}{-\otimes t}}((c,0),(d,0))\to\Map_\catC(c,d)
			\]
			is a split retraction with fiber over~$0$ the space~$\Map_{\catC}(c[1],d\otimes t)$.
			As seen in the proof of \Cref{lem:monodromy_Kummer_motive}, the monodromy of the cofiber of a map $\alpha\colon (c,0)\to (d,0)$ with $\pi(\alpha)\simeq 0$ extracts the element $c[1]\to d\otimes t$ in this fiber.
		\end{rmk}

		\begin{rmk}
			\label{rmk:Sym1}
			We may identify $\Alg_{\E_0}(\catC)$ with the \icat $\catC_{\one/}$ consisting of maps $\one \to R$ in $\catC$ (objects of $\catC$ equipped with a unit map but no multiplication that it is the unit for, cf.~\cite[2.1.3.10]{HA}). There is an obvious forgetful functor $\CAlg(\catC) = \Alg_{\E_\infty}(\catC)\to \Alg_{\E_0}(\catC)$, and we denote by $\free_{\one/}(-)$ its left adjoint: it exists by \cite[Corollary 3.1.3.7]{HA}. 
			
                        Recall that the cofiber sequence~\eqref{eq:Kummer-underlying} is split.
                        Any retraction $\pi(\mcK)\to\one$ induces, by adjunction, a map $\mcK\to\pirn(\one)$ of $\E_0$-algebras thus a map $\free_{\one/}(\mcK)\to\pirn(\one)$ of $\E_\infty$-algebras in $\LFnil{\catC}{-\otimes t}$.
		\end{rmk}

		\begin{prop}\label{prop:HisKummer}
                Assume $\catC$ is $\Q$-linear and $\Alt^2(t)\simeq 0$.
                The morphism of commutative algebras in~$\LFnil{\catC}{-\otimes t}$ (constructed in \Cref{rmk:Sym1}) is an equivalence:
			\[
			\free_{\one/}(\mcK) \isoto \pirn(\one).
			\]
		\end{prop}
		\begin{proof}
			As the functor~$\pi\colon\LFnil{\catC}{-\otimes t}\to\catC$ is conservative, it suffices to show that the morphism $\pi\free_{\one/}(\mcK)\to\bH$ is an equivalence in~$\CAlg(\catC)$.
			By~\cite[Remark~3.1.3.8]{HA}, the domain of this map identifies canonically with $\free_{\one/}(\pi(\mcK))$ where we view $\pi(\mcK)\in\Alg_{\E_0}(\catC)$.
                        The fiber $t\to\pi(\mcK)$ of the retraction $\pi(\mcK)\to\one$ exhibits~$\pi(\mcK)$ as the free~$\E_0$-algebra generated by~$t$.
			(Indeed, the left adjoint to the forgetful functor~$\catC_{\one/}\to\catC$ is given by $\one\amalg-$.)
			It follows that $t\to\pi(\mcK)\to\free_{\one/}(\pi(\mcK))$ exihibts the latter as the free commutative algebra generated by~$t$.
			It is a little exercise to deduce from \Cref{lem:monodromy_Kummer_motive} and the natural transformation of \Cref{cons:pi_rho} that the composite $t\to\free_{\one/}(\pi(\mcK))\to\bH\simeq\oplus_{n\geq 0}t^{\otimes n}$ is the canonical inclusion.
			We conclude the proof by virtue of \Cref{lem:H-free}.
		\end{proof}
		\begin{rmk}\label{rmk:kummerUDA}
			Let us justify the terminology. Let $S$ be a regular scheme, and let $\UDA(S)\simeq\LFnil{\DA(S)}{\otimes\one(-1)}$ be as in \Cref{dfn:unipotent-motives}. %
			We can compute the mapping spectrum as follows,
			\begin{align*}
				\spMap_{\UDA(S)}(\one, \one(n)[m]) &\simeq  \spMap_{\DA(S)} (\one, q_*\one(n)[m]) \\ &\simeq \spMap_{\DA(S)} (\one,  \one(n)[m]) \oplus \spMap_{\DA(S)} (\one, \one(n-1)[m-1]),
			\end{align*}
			induced by the splitting $q_*\one\simeq\one\oplus\one(-1)[-1]$ of \Cref{notn:q_*1-split}.
			So for $n=m=1$ there is a distinguished element~$(0,\id)$ in the $\pi_0$~of the right-hand side, corresponding to a map $e\colon \one(-1)[-1] \to \one$ in $\UDA(S)$ inducing a bifiber sequence
			\[ \one \to \mathcal{K} \to{\one}(-1). \]
			Note that the morphism $e$ corresponds to the indeterminate $\varpi\in \mathcal{O}^\times(\G_m)$.
			This is the triangle considered in \cite[Definition 3.6.22 and Lemme 3.6.28]{ayoub-th2} and called Kummer extension. Clearly, $1^* \mathcal{K} = \one\oplus \one(-1)$ splits in $\DA(S)$, but $\mathcal{K}$ is a non-trivial extension in $\UDA(S)$. As explained in \cite[Lemme 11.22]{AyoubEt}, the \'etale realization of the Kummer motive is the classical Kummer extension of $\Z/\ell^n(-1)$ by $\Z/\ell^n$.
			
			In~\cite[D\'efinition~3.6.29]{ayoub-th1} (and~\cite[D\'efinition 11.6]{AyoubEt}), the object $\free_{\one/}(\mcK)$ in $\UDA(S)$ (or rather its image in the homotopy category) is called the \emph{logarithm}~$\mathcal{L}og^{\vee}$.
			In this setting, \Cref{lem:H-free,prop:HisKummer} apply and we may identify the logarithm with~$\pirn(\one)$.%
		\end{rmk}

                \begin{rmk}
                \label{rmk:derivation}
 The canonical `monodromy' map $N\colon \pirn(\one)\to \pirn(t)$ is a derivation.
 Recall this means that $(\id,N)\colon \pirn(\one)\to\pirn(\one)\oplus\pirn(t)$ defines a section to the structure map of this split square-zero extension.
 To see this we may apply \Cref{prop:HisKummer} which identifies such sections with morphisms $\mathcal{K}\to\mathcal{K}\oplus\pirn(t)$ in $\Alg_{\E_0}(\catC)_{/\mathcal{K}}$.
 We may take the morphism $(\id,\mathcal{K}\to\pirn(\one)\xto{N}\pirn(t))$.

 In particular, at the level of homotopy classes of maps $\pirn(\one)\otimes\pirn(\one)\to\pirn(t)$ we have the formula
 \[
N(ab)=N(a)b+aN(b).
\]
This can also be verified very explicitly using the description of the projection formula in \Cref{sta:pi-rho-projection-formula} and the multiplication on~$\pirn(\one)$ given in \Cref{lem:H-free}.
 \end{rmk}

		\subsection{Detection via monodromy}
		\label{sec:detection-monodromy}
		One of our main goals in this paper is to show that our abstractly defined Hyodo-Kato cohomology, as an enriched version of the rigid cohomology with a Frobenius and a monodromy operator, agrees with the classically defined Hyodo-Kato cohomology as $(\varphi, N)$-modules. This is in fact a special instance of a comparison result that holds in the $\ell$-adic and complex analytic setting as well.
		All of these comparison results will follow from the first detection principle that we now discuss.

		\begin{notn}\label{notation:comparison_monod}
			Let $F\colon \catC\to \catD$ be a morphism in $\Prlmst$, and denote by $u=F(t)$ the image of~$t$.
			Any such functor $F$ induces a monoidal functor $\widehat{F}\colon \LF{\catC}{-\otimes t} \to \LF{\catD}{-\otimes u}$, restricting to a monoidal functor $\widehat{F}\colon \LFnil{\catC}{-\otimes t} \to \LFnil{\catD}{-\otimes u}$.
			Informally, it sends the object $(X,f)$ to $(F(X),F(f))$. 
		\end{notn}
		
		\begin{assu}
			\label{hyp:F-package}
			We will make the following assumptions:
			\begin{enumerate}
   \item $F$ is a morphism in $\Prloost$.
				\item The right adjoint~$G$ to~$F$ satisfies the projection formula:
				\[
				Gd\otimes c\isoto G(d\otimes Fc).
				\]
				\item $\catC$ is $\Q$-linear and $\Alt^2(t)\simeq 0$.
			\end{enumerate}
   We view $\LFnil{\catC}{-\otimes t}$ (resp. $\LFnil{\catD}{-\otimes u}$)  as an object of $\Prloost_{\catC^\otimes/-/\catD^\otimes}$ via the functors $\catC^\otimes \xto{(N=0)} \LFnil{\catC}{-\otimes t}$ and  $\LFnil{\catC}{-\otimes t} \xto{F\circ \pi} \catD^\otimes$ (resp.~via the functors $\catC^\otimes \xto{(N=0)\circ F} \LFnil{\catD}{-\otimes u}$ and $\LFnil{\catD}{-\otimes u} \xto{\pi} \catD^\otimes$).
		\end{assu}
		
		\begin{rmk}\label{rmk:hypholds}
			The projection formula is automatically satisfied in case the compact objects of~$\catC$ are strongly dualizable: see e.g.~\cite[Lemme~2.8]{ayoub-h1}.   
		\end{rmk}
		
		\begin{prop}
			\label{sta:classification-monodromy}
			Under \Cref{hyp:F-package} there is a canonical homotopy equivalence of spaces:
			\[
			\Map_{\Prloost_{\catC^\otimes/-/\catD^\otimes}}(\LFnil{\catC}{-\otimes t},\LFnil{\catD}{-\otimes u})\isoto\Map_{\catD}(u,u).
			\]
		\end{prop}
		\begin{proof}
			Let $\catD'\subseteq \catD$ be the full \subicat generated under colimits by the image of~$F$.
			Note that the monoidal structure restricts to~$\catD'$ so that $(\catD')^{\otimes}\in\Prloost$.
			If $F'\colon \catC\to\catD'$ denotes the corestriction of~$F$ and we prove the statement for~$F'$ then the original statement follows from the canonical equivalences
			$$\begin{aligned}
				\Map_{\Prloost_{\catC^\otimes/-/(\catD')^\otimes}}(\LFnil{\catC}{-\otimes t},\LFnil{(\catD')}{-\otimes u})&\simeq 
				\Map_{\Prloost_{\catC^\otimes/-/\catD^\otimes}}(\LFnil{\catC}{-\otimes t},\LFnil{\catD}{-\otimes u}),\\
				\Map_{\catD'}(u,u)&\simeq \Map_{\catD}(u,u).
			\end{aligned}$$
			In other words, from now on we may assume that the image of~$F$ generates~$\catD$ under colimits.
			Equivalently, its right adjoint $G$ is conservative. Note also that, as $F$ sends compact objects to compact objects, $G$ also preserves colimits. 
			
			Write $B=\one\oplus u[-1]$ for the associated split square-zero extension in~$\catD$.
			By the projection formula and the assumptions just made, we deduce that the functor $F\colon\catC\to\catD$, and hence also the functor $B\otimes F(-)\colon\catC\to\Mod_B(\catD)$, satisfy the conditions (a)-(e) of \cite[Corollary~4.8.5.21]{HA}. This implies that the latter lies in the essential image of the fully faithful embedding
			\[
			\CAlg(\catC)\into\Prloost_{\catC^{\otimes}/}
			\]
			induced by $R\mapsto (\catC\to\Mod_R(\catC))$. 
			In particular, we get a homotopy equivalence of spaces
			\[
			\Map_{\CAlg(\catC)_{/G\one}}(A,G(B))\isoto\Map_{\Prloost_{\catC^{\otimes}/-/\catD^\otimes}}(\Mod_A(\catC),\Mod_B(\catD)).
			\]
			The left-hand side can further be identified with~$\Map_{\CAlg(\catD)_{/\one}}(F(A),B)$ via the adjunction $F\colon\CAlg(\catC)\rightleftarrows\CAlg(\catD)\colon G$, see~\cite[Lemma~5.5.5.12]{lurie}.
			
			We will conclude the proof by showing that the augmentation ideal functor induces a homotopy equivalence
			\[
			\Map_{\CAlg(\catD)_{/\one}}(F(A),B)\xrightarrow[\sim]{\aid}\Map_{\catD}(u[-1],u[-1]).
			\]
			Indeed, \Cref{hyp:F-package} and \Cref{sta:free=ssze} imply that $F(A)\simeq F(\free(t[-1]))\simeq\free(u[-1])$ so that the left-hand side identifies with
			\[
			\Map_{\CAlg(\catD)_{/\one}}(\free(u[-1]),\one\oplus u[-1])\simeq\Map_{\catD}(u[-1],\aid(\one\oplus u[-1]))\simeq\Map_{\catD}(u[-1],u[-1])
			\]
			induced by the augmentation ideal functor.
		\end{proof}
		
		\begin{rmk}
			Let $\tilde{F}\colon\LFnil{\catC}{-\otimes t}\to\LFnil{\catD}{-\otimes u}$ be a morphism in~$\Prloost_{\catC^\otimes/-/\catD^\otimes}$.
			Let us denote by $\fib_t\colon \LFnil{\catC}{-\otimes t}\to\catC$ the functor~$\fib(N)$ of \Cref{sta:N=0-fib}, and similarly for~$\fib_u$.
			The equivalence $(N=0)\circ F\simeq \tilde{F}\circ (N=0)$ induces a map in~$\CAlg(\catD)$:
			\[
			F(A)\simeq F\fib_t\one\to \fib_u\tilde{F}\one\simeq B.
			\]
			Being a morphism \emph{over}~$\catD^\otimes$ it underlies a morphism $u_{\tilde{F}}\colon F(A)\to B$ in~$\CAlg(\catD)_{/\one}$.
			Unwinding the constructions in~\cite[\S\,4.8.5]{HA} one sees that the homotopy equivalence of \Cref{sta:classification-monodromy} associates to $\tilde{F}$ the map $\aid(u_{\tilde{F}})[1]\colon u\to u$ under the augmentation ideal functor (suspended once).
		\end{rmk}
		Our goal now is to describe this map more explicitly in terms of the monodromy on the Kummer object.
		
		\begin{cons}
			Let $\tilde{F}\colon\LFnil{\catC}{-\otimes t}\to\LFnil{\catD}{-\otimes u}$ be a morphism in~$\Prloost_{\catC^\otimes/-/\catD^\otimes}$.
			Recall the map $e_t\colon (t[-1],0)\to (\one,0)$ in~$\LFnil{\catC}{-\otimes t}$ whose cofiber is the Kummer object.
			As $\pi\tilde{F}(e_t)\simeq F\pi(e_t)\simeq 0$, we see that $\tilde{F}(e_t)\in\Map_{\LFnil{\catD}{-\otimes u}}((u[-1],0),(\one,0))$ corresponds to an element in $\Map_{\catD}(u,u)$, see \Cref{rmk:Kummer-nonsplit}.
			This construction gives rise to a map of sets
			\begin{equation}
				\label{eq:e}
				e\colon\pi_0\Map_{\Prloost_{\catC^\otimes/-/\catD^\otimes}}(\LFnil{\catC}{-\otimes t},\LFnil{\catD}{-\otimes u})\to\pi_0\Map_{\catD}(u,u).
			\end{equation}
		\end{cons}
		
		\begin{rmk}
			As observed in \Cref{rmk:Kummer-nonsplit}, the map $e(\tilde{F})\colon u\to u$ is the only possibly non-zero entry in the monodromy matrix of the cofiber of $\tilde{F}(e_t)$.
			In other words, the monodromy of~$\tilde{F}(\mcK_t)$ has the shape $
			\left(\begin{smallmatrix}
				0&e(\tilde{F})\\
				0&0
			\end{smallmatrix}\right)$.
		\end{rmk}
		
		\begin{lemma}\label{lemma:!N}
			The homotopy equivalence of \Cref{sta:classification-monodromy} induces on connected components the map~$e$ of~\eqref{eq:e}.
		\end{lemma}
		\begin{proof}
			Let $\tilde{F}\colon\LFnil{\catC}{-\otimes t}\to\LFnil{\catD}{-\otimes u}$ be a morphism in~$\Prloost_{\catC^\otimes/-/\catD^\otimes}$.
			We have the following commutative diagram in $\Ho(\catD)$.
			\[
			\begin{tikzcd}
				Ft[-1]
				\ar[r, "\text{unit}"]
				\ar[d, "\text{incl}" swap]
				&
				F\fib_t(t[-1],0)
				\ar[ld, "F\fib_t(e_t)"]
				\ar[r, "u_{\tilde{F}}"]
				&
				\fib_u\tilde{F}(t[-1],0)
				\ar[ld, "\fib_u\tilde{F}(e_t)"]
				\\
				F\fib_t\one
				\ar[r, "u_{\tilde{F}}" swap]
				&
				\fib_u\tilde{F}\one
			\end{tikzcd}
			\]
			The top horizontal arrow is the inclusion of the direct summand $u[-1]\into u[-1]\oplus u^{\otimes 2}[-2]$ so that the path ``right-then-down'' is $
			\left(
			\begin{smallmatrix}
				0\\ e(\tilde{F})[-1]
			\end{smallmatrix}
			\right)\colon u[-1]\to \one\oplus u[-1]$.
			This shows the claim.
		\end{proof}

		\begin{cor}\label{prop:unicity}
			Under \Cref{hyp:F-package}, let $\tilde{F}\colon \LFnil{\catC}{-\otimes t}\to \LFnil{\catD}{-\otimes u}$ be a morphism in $\Prloost$ and assume:
			\begin{enumerate}[(i)]
				\item \label{cond:kummer0}$\tilde{F}$ is compatible with $(N=0)$ and $\pi$, i.e., there are  commutative squares of monoidal functors
				\[
				\begin{tikzcd}
					\catC
					\ar[d, "F"]\ar[r, "N=0"]
					&
					\LFnil{\catC}{-\otimes t}
					\ar[d,"\tilde{F}"]\ar[r,"\pi"]
					&
					\catC
					\ar[d, "F"]
					\\
					\catD
					\ar[r, "N=0"]
					&
					\LFnil{\catD}{-\otimes u}
					\ar[r,"\pi"]
					&
					\catD.
				\end{tikzcd}
				\]
				\item\label{cond:kummer}The monodromy matrix of $\tilde{F}(\mcK_t)$ in $\Ho(\catD)$ is given by
				\[
				\one\oplus u\xto{
					\left(\begin{smallmatrix}
						0&\id_u\\0&0
					\end{smallmatrix}\right)}u\oplus u^{\otimes2}.
				\]
			\end{enumerate}
			Then $\tilde{F}\simeq \widehat{F}$ (of \Cref{notation:comparison_monod}) as monoidal functors.
		\end{cor}
		\begin{proof}
			Condition \eqref{cond:kummer0} ensures that  $\tilde{F}$ lies in $\Map_{\Prloost_{\catC^\otimes/-/\catD^\otimes}}(\LFnil{\catC}{-\otimes t},\LFnil{\catD}{-\otimes u})$. Note that under the equivalence $\pi_0\Map_{\Prloost_{\catC^\otimes/-/\catD^\otimes}}(\LFnil{\catC}{-\otimes t},\LFnil{\catD}{-\otimes u})\simeq\pi_0\Map_{\catD}(u,u)$ of \Cref{sta:classification-monodromy} both functors correspond to the element $\id_u$, by \Cref{lemma:!N}.
		\end{proof}
		\begin{rmk}
			Another possible extension of $F$ is defined by the following composition: $(N=0)\circ F\circ \pi$. In other words, the functor which ``kills the monodromy datum''. Note that it does not satisfy condition \eqref{cond:kummer} above.
		\end{rmk}
		
		\begin{rmk}
			It is possible to give a more ``hands-on'' proof of the previous result, which is closer to the argument of~\cite[Th\'eor\`eme 11.17]{AyoubEt}: in order to verify that $\tilde{F}(f) = F(f)$, it is enough to check this on co-free objects. To see this, note that 
			one can recover the monodromy $f$ of $X$ as the first row of the monodromy of $(X,f)\boxtimes (\bH,\text{proj})$. We can thus replace $(X,f)$ by $(X,f)\boxtimes (\bH, \text{proj})$ and then further reduce to the case $(X,0)\boxtimes (\bH,\text{proj})$, since the latter is equivalent to the former in the \icat $\LFnil{\catC}{-\otimes t}$ by \Cref{lemma:ciucciaN}. In particular, it is enough to identify the two functors  on $(\bH,\text{proj})$ itself, and hence, in view of \Cref{prop:HisKummer}, on the Kummer object.
		\end{rmk}
	
		\subsection{Detection through hearts of weight structures}
		\label{sec:W}
		In this section we recall how certain questions about an \icat can be reduced to questions about a \subicat if the latter `is large enough' and has some self-orthogonality properties.  
		This employs Bondarko's theory of weight structures, as recast in the \icategorical setting by Sosnilo.
		However, here we will need to know very little of that theory and we refer to~\cite{MR2746283,Sosnilo,Sosnilo2,Aoki:weight-complex} for a more thorough development.
		
		In this section, $\catC$ will be of a different nature than before:
		We let $\catC$ be a stable \icat.
		A full \subicat $\heart{}\subseteq \catC$ is called \emph{negative} if the mapping spectrum~$\spMap_{\catC}(X,Y)$ is connective for all $X,Y\in\heart{}$. %

		\begin{dfn}\label{dfn:bws}
			Let $\catC$ be a  stable \icat.
			We say that a full \subicat $\heart{}\subseteq \catC$ is the \emph{heart of a bounded weight structure} if
			\begin{enumerate}
				\item
				\label{it:negative}
				it is negative,
				\item
				\label{it:thick} it is closed under finite coproducts and retracts, and
				\item
				\label{it:generate}
				generates~$\catC$ under finite limits and colimits.
			\end{enumerate}
			If $\catC^\otimes$ is a stable monoidal structure\footnote{This means that the tensor product commutes with finite (co)limits in both variables} on $\catC$ we say that $\heart{}$ is the \emph{heart of a bounded compatible weight structure} if moreover it  contains the tensor unit and is closed under tensor products.
		\end{dfn}
   
		\begin{rmk}
			If   $\heart{}$ is the heart of a {bounded compatible weight structure},  the monoidal structure on~$\catC$ restricts to a monoidal structure on~$\heart{}$ that commutes with finite (co)products in both variables (see \cite[Proposition 2.2.1.1, Remark 2.2.1.2]{HA}).
		\end{rmk}
		\begin{rmk}
			As the name suggests, in the situation of \Cref{dfn:bws} there is then a bounded weight structure on~$\catC$ whose heart coincides with~$\heart{}$.
			See~\cite[Definition 3.1.1]{Sosnilo2} and \cite[Theorem~4.3.2(II.2)]{MR2746283}.
			We will not need this in the sequel.
			
			We also note that the conditions~\ref{it:negative}--\ref{it:generate} can be checked on the homotopy category.
			That is, a full \subicat $\heart{}\subseteq\catC$ is the heart of a bounded weight structure if and only if:
			\begin{enumerate}
				\item $\Hom_{\ho(\catC)}(X,Y[n])=0$ for all $X,Y\in\heart{}$ and $n>0$,
				\item $\ho(\heart{})\subseteq\ho(\catC)$ is closed under finite direct sums and summands,
				\item the smallest triangulated subcategory of~$\ho(\catC)$ that contains~$\ho(\heart{})$ is~$\ho(C)$ itself.
			\end{enumerate}
		\end{rmk}

		\begin{rmk}
			Let $\heart{}$ be the heart of a bounded weight structure on~$\catC$, and let $\catD$ be a stable presentable \icat.
			Every exact functor $\catC\to\catD$ restricts to an additive functor $\heart{}\to\catD$, that is, a functor which preserves finite products (or, equivalently, finite coproducts).
			The canonical restriction functor $\Fun(\catC,\catD)\to\Fun(\heart{},\catD)$ induces an equivalence on the corresponding full \subicats:
			\begin{equation}
				\label{eq:add-vs-ex}
				\Funex(\catC,\catD)\isoto\Funadd(\heart{},\catD).
			\end{equation}
			This is essentially contained in~\cite[Proposition 3.3]{Sosnilo}.
			We now give a monoidal analog. 
		\end{rmk}

		\begin{notn}
			Let $\catC^\otimes$ be a stable monoidal \icat, let $\catD^\otimes$ be a stable presentably monoidal \icat,  and $\heart{}^\otimes$ an additive monoidal \icat.
			(So, in the second case the tensor product preserves small colimits in both variables, and in the third case it preserves finite coproducts.)
			We denote by
			\[
			\Funexo(\catC,\catD),\qquad\Funaddo(\heart{},\catD)
			\]
			the full \subicats of $\Fun^\otimes(\catC,\catD)$ and $\Fun^\otimes(\heart{},\catD)$ respectively (as in~\cite[Definition~2.1.3.7]{HA}), spanned by those monoidal functors whose underlying functors are exact and additive, respectively.
		\end{notn}
		
		\begin{prop}\label{cor:sosnilo@}
			Let $\catC^\otimes$ be a small  stable monoidal \icat, and $\heart{}\subseteq\catC$ the heart of a bounded compatible weight structure.
			Then for every stable presentably monoidal \icat~$\catD$, the restriction functor
			\[
			\Funexo(\catC,\catD)\to\Funaddo(\heart{},\catD)
			\]
			is an equivalence.
		\end{prop}
		\begin{proof}
			By~\cite[Lemma~4.2]{Aoki:weight-complex}, the \icat $\Funadd(\heart{}^{\op},\Sp)$ admits a monoidal structure such that the `additive' Yoneda embedding again induces an equivalence
			\[
			\Ind\catC^\otimes\isoto\Funadd(\heart{}^{\op},\Sp)^{\otimes}.
			\]
			Consider the following commutative square:
			\[
			\begin{tikzcd}
				\FunLo(\Ind\catC,\catD)
				\ar[d]
				&
				\FunLo(\Funadd(\heart{}^{\op},\Sp),\catD)
				\ar[l]
				\ar[d]
				\\
				\Funexo(\catC,\catD)
				\ar[r]
				&
				\Funaddo(\heart{},\catD).
			\end{tikzcd}
			\]
			We just argued that the top horizontal arrow is an equivalence, and the left vertical arrow is an equivalence by the last statement in~\cite[Corollary~4.8.1.14]{HA}.
			It remains to prove that the right vertical arrow is an equivalence.
			
			That functor is induced by precomposition with the `additive' Yoneda embedding and stabilization~\cite[Corollary~2.9]{Aoki:weight-complex}
			\[
			\heart{}^\otimes\to\Funadd(\heart{}^{\op},\Sp_{\geq 0})^\otimes\to\Sp(\Funadd(\heart{}^{\op},\Sp_{\geq 0}))^\otimes\ (\simeq\Funadd(\heart{}^{\op},\Sp)^\otimes).
			\]
			Restriction along the second functor induces an equivalence
			\[
			\FunLo(\Sp(\Funadd(\heart{}^{\op},\Sp_{\geq 0})),\catD)\isoto\FunLo(\Funadd(\heart{}^{\op},\Sp_{\geq 0}),\catD),
			\]
			as follows from~\cite[Proposition~5.4.(iv)]{MR3450758}.
			As observed in~\cite[Remark~2.5]{Aoki:weight-complex}, $\Funadd(\heart{}^{\op},\Sp_{\geq 0})$ can be identified with the cocompletion~$\mathscr{P}^{K'}_{K}(\heart{})$ of~\cite[Proposition~4.8.1.10]{HA} for $K$ the collection of finite discrete simplicial sets and for $K'$ the collection of all small simplicial sets.
			Part~(4) of \textsl{loc.\,cit.}\ says that restriction induces an equivalence of \icats
			\[
			\FunLlax(\mathscr{P}^{K'}_{K}(\heart{}),\catD)\isoto\Funaddlax(\heart{},\catD)
			\]
			between lax monoidal functors whose underlying functors preserve small colimits (resp.\ finite coproducts).
			Since $\heart{}^\otimes\to\mathscr{P}^{K'}_K(\heart{})^\otimes$ is monoidal, this functor restricts to a fully faithful embedding 
			\begin{equation}
				\label{eq:restriction-H}
				\FunLo(\mathscr{P}^{K'}_{K}(\heart{}),\catD)\into\Funaddo(\heart{},\catD).
			\end{equation}
			Now, if $f^\otimes\colon \heart{}^\otimes\to\catD^\otimes$ is monoidal and the underlying functor~$f$ additive then the induced monoidal functor $\mathscr{P}^{K'}(\heart{})^\otimes\to \catD^\otimes$ factors through the localization $\mathscr{P}^{K'}(\heart{})\to\mathscr{P}^{K'}_K(\heart{})$ and we conclude that~\eqref{eq:restriction-H} is an equivalence.
			This completes the proof.
		\end{proof}

		\begin{exm}\label{eg:ws}\label{exm:WS}
			The following are examples of hearts of weight structures that we will  use.
			\begin{enumerate}
				\item\label{exm:ChowWS}
				Let $k$ be a  field.
				We let $\catC=\DA(k)_{\cpt}$ be the \icat of compact motives over~$k$ 
				(see \Cref{sec:RigDA-monodromy} for our conventions), and $\heart{k}$ the full \subicat of Chow motives.
				That is, the idempotent completion of the additive \subicat generated by objects of the form ${\mathsf{M}}(X)\{i\}$ with $i\in\Z$, $\{i\}:=(i)[2i]$, and $X/k$ smooth and proper.
				By~\cite[Proposition~6.5.3]{MR2746283}, this is the heart of a bounded weight structure.
    \footnote{In \textsl{loc.\,cit.}\ this is shown if~$k$ is a perfect field. However, the base change to the perfection is an equivalence (see \Cref{rmk:gr_is_insensitive}) %
					so that the result remains true for general fields~$k$. %
     }
				Since $\mathsf{M}(X)\{i\}\otimes\mathsf{M}(Y)\{j\}\simeq\mathsf{M}(X\times Y)\{i+j\}$, it is clear that it is compatible with the monoidal structure.
				\item\label{exm:WSonComega} If $\widetilde{\heart{}}$ is a full \subicat of compact objects in a stable presentable \icat~$\catC$ such that $\widetilde{\heart{}}$ generates~$\catC$ under shifts and colimits. Then  $\widetilde{\heart{}}$ generates $\catC_{\cpt}$ under shifts, finite colimits and retracts. If~$\widetilde{\heart{}}$ is also negative, then the idempotent completion of its additive envelope  is the heart~$\heart{}$ of a bounded weight structure on~$\catC_{\cpt}$. %
				See~\cite[Corollary~2.1.2]{MR3775347}.
				\item \label{wc}If $\heart{}$ is the heart of a  bounded compatible weight structure on~$\catC^{\otimes}$ then $\Ho(\heart{})$ is the heart of a bounded compatible weight structure on $\mcK^b(\Ho(\heart{}))^\otimes$ (its bounded homotopy category) and the canonical functor $\heart{}\to\Ho(\heart{})$ induces a conservative monoidal functor $w\colon \catC^\otimes\to\mcK^b(\Ho(\heart{}))^\otimes$,  called the weight-complex, see~\cite[Corollary~4.5]{Aoki:weight-complex}.
			\end{enumerate}
		\end{exm}
		
		\begin{cor}
			\label{prop:factorsonwc}
			Let $\heart{}$ be the heart of a bounded compatible weight structure on $\catC^\otimes$ and let $F\colon\catC^\otimes\to\catD^\otimes$ be a monoidal exact functor to a stable presentably monoidal \icat. If $\pi_n\spMap_{\catD}(FX,FY)=0$ for any $X,Y$ in $ \heart{}$ and any $n>0$,  then~$F$ factors as follows:
			$$
			\catC^\otimes\xto{w}\mcK^b(\Ho(\heart{}))^\otimes\xto{\tilde{F}}\catD^\otimes.
			$$
		\end{cor}
		
		\begin{proof}
			In light of \Cref{cor:sosnilo@}, we may as well prove that the restriction of $F$ to $\heart{}$ admits a factorization $\heart{}^\otimes\to\Ho(\heart{})^\otimes\to \catD^\otimes$ (see \Cref{eg:ws}\eqref{wc}). %
			Let $\catD'\subseteq\catD$ be the full \subicat spanned by the essential image of~$F$.
			Since~$F$ is monoidal, the monoidal structure on~$\catD$ restricts to one on~$\catD'$~\cite[Proposition~2.2.1.1, Remark~2.2.1.2]{HA}.
			It will then suffice to show that the corestricted functor $F'\colon \heart{}^\otimes\to(\catD')^\otimes$ factors through $\Ho(\heart{})^\otimes$.
			By our assumption (and~\cite[Remark~2.1.3.8]{HA}), the canonical functor $(\catD')^\otimes\to \Ho(\catD')^\otimes$ is an equivalence of monoidal \icats.
			We are reduced to show that the monoidal functor $\heart{}^\otimes\to\Ho(\catD')^\otimes=\Ho(\catD'^\otimes)$ factors through~$\Ho(\heart{})^\otimes=\Ho(\heart{}^\otimes)$.
			This is clear by the nerve-homotopy category adjunction.
		\end{proof}
		
		We now give an example of the situation above which is akin to the situation considered in~\cite[Remark 2.4.3(1)]{MR2746283}
		\begin{exm}\label{exm:Ai}
			Let $\mcA$ be an abelian category of cohomological dimension $\leq1$ and let $\{\mcA_i\}_{i\in\Z}$ be abelian subcategories of~$\mcA$ closed under subquotients such that, for any $X_i\in\mcA_i$ and $X_j\in\mcA_j$ one has $\Hom(X_i,X_j)=0$ if $i\neq j$ and $\Ext^1(X_i,X_j)=0$ if $j>i+1$. 
			We say that $C\in\mcD^b(\mcA)$ is \emph{pure} if $\Hm_i(C)\in\mcA_i$ for every $i$. Note that in this case 
			$$
			\begin{aligned}
				\pi_n\spMap_{\mcD(\mcA)}(C,C')&=\Hom_{\mathcal{D}(\mcA)}(C,C'[-n])\\&\simeq\bigoplus_{i,j}\Hom_{\mathcal{D}(\mcA)}(\Hm_iC[i],\Hm_jC'[j-n])\\&=\bigoplus_{i,j} \Ext^{j-i-n}(\Hm_iC,\Hm_jC')=0
			\end{aligned}
			$$ whenever $C,C'$ are pure and $n>0$. In particular, we deduce that if $\heart{}$ is the heart of a bounded compatible weight structure on $\catC^\otimes$ and $F\colon\catC^\otimes\to\mathcal{D}^b(\mathcal{A})^\otimes$ is  monoidal, exact and such that $F(X)$ is pure for any $X\in \heart{}$ then~$F$ factors through~$\mcK^b(\Ho(\heart{}))^\otimes$.
		\end{exm}
		
		\begin{rmk}\label{rmk:weight_filtration}
			Here we provide a reinterpretation of a result of Bondarko~\cite[Theorem~2.3.2]{MR2746283}.
			Let $\catH$ be an ordinary additive category and denote by $\Ch^b(\catH)$ and $\mcK^b(\catH)$ its ordinary category of (bounded) chain complexes and its homotopy \icat as before.
			Given an object $C\in\Ch^b(\catH)$ we may filter it by its stupid truncations from the left, that is, $C(i)=\sigma_{\leq i}C$, thus obtaining a filtered chain complex $C(\bullet)\in\Fun(\Z,\mcK^b(\catH))$
			\[
			\cdots C(i-1)\to C(i)\to C(i+1)\to\cdots
			\]
			Given an exact functor $F\colon \mcK^b(\catH)\to\catD$ between stable \icats assume that~$\catD$ is equipped with a t-structure whose heart is the abelian category~$\mcA$.
			The spectral sequence associated with a filtered object (see~\cite[\S\,1.2.2]{HA}) produces a convergent spectral sequence in~$\mcA$ with signature
			\[
			E^1_{p,q}=\Hm_{q}(F(C_p))\ \Rightarrow \ \Hm_{p+q}(F(C)).
			\]
			The functor $\Ch^b(\catH)\to \mathrm{SSeq}_{\geq 1}$ to (the ordinary category of) spectral sequences starting at page~$1$ descends to a functor $\mcK^b(\catH)\to\mathrm{SSeq}_{\geq 2}$~\cite[Theorem~2.3.2.IV]{MR2746283}.

			We will be particularly interested when $\catH=\Ho(\heart{})$ is the homotopy category of the heart of a weight structure and $C$ the weight complex of an object under consideration.
			If $\catD=\mathcal{D}^b(\mcA)$ and $F(X)$ is pure for every $X\in\heart{}$ (as in \Cref{exm:Ai}) then this spectral sequence degenerates at page~$E^2$ and the associated graded objects $\gr^{w}_i \Hm_nF(C)$ lie in~$\mcA_i$, i.e., this canonical filtration is the ``weight filtration'' of the homology groups.
		\end{rmk}
		
		\section{Analytic motives and their realizations}
		\label{sec:realizations}
		We now apply the results of the previous sections to the specific case of (rigid analytic) motives over a non-archimedean field $K$. In this section, we assume the following.
		\begin{notn}\label{not:K}
			Let $K$ be a complete non-archimedean field with a perfect residue field and a value group of (rational) rank $1$. 
			We denote as usual by $k$ its residue field and by $\mcO_K$ the valuation ring. We say that an element $\varpi\in K$ is a \emph{pseudo-uniformizer} if it is topologically nilpotent and not zero.\footnote{For $K$ as above, $\mcO_K$ is a complete height-$1$ valuation ring, so this amounts to~$\varpi$ being a non-zero element of the maximal ideal of $\mcO_K$. In particular, we have $\mcO_K[1/\varpi]=K$. %
   }
		\end{notn}
                \begin{rmk}
		    Typical examples of $K$ satisfying the hypotheses above are discretely valued fields with a perfect residue field, their complete algebraic closures, and any other immediate extension of them (such as the their spherical completions).
		\end{rmk}
		
		\subsection{Rigid motives as motives with monodromy}\label{sec:RigDA-monodromy}
		We can now prove \Cref{thm:intro1} of the introduction.
		
		\begin{dfn} We will consider the following \icats of motivic sheaves.\label{dfn:motives}
			\begin{enumerate}
				\item \label{eq:algebraic_motives}
				Let $S$ be a scheme locally of finite Krull dimension. We write $\DA(S) = \DAet(S;\Q)$ for the \icat of hypercomplete \'etale motives over~$S$, with rational coefficients as defined in~\cite[Section 3]{AyoubEt} by putting~$\Lambda=\Q$. It is constructed from the full \subicat $\DA^{\eff}_{\et}(S;\Q)$ of $\mathrm{Shv}_{\et}(\Sm/S; \mcD(\Q))$  spanned by those objects which are local with respect to the collection of maps $\Q(\A^1_Y) \to \Q(Y)$, by formally inverting the tensor product operation with the Tate twist $\one(1)$.  
				\item\label{eq:analytic_motives} Let $S$ be a rigid analytic space locally of finite Krull dimension, in the sense of \cite[Definition II.2.2.18]{fujiwara-kato}. We write $\RigDA(S) = \RigDA_{\et}(S; \Q)$ for the \icat of hypercomplete rigid analytic motives over~$S$ with respect  to the étale topology, as defined in~\cite[Definition 2.1.15]{agv}. Its construction is analogous to~\eqref{eq:algebraic_motives}, by replacing~$\Sm/S$ with~$\RigSm/S$, and the affine line~$\A^1$ by the closed unit disk~$\B^1$.  %
			\end{enumerate}
			
			These are stable, compactly generated monoidal \icats, underlying a six-functor formalism, see~\cite{ayoub-th1,ayoub-th2,agv}.
			If $Y$ is a scheme or a rigid analytic space which is smooth over $S$, we sometimes write $\mathsf{M}_S(Y)$ (or $\mathsf{M}(Y)$ for brevity) to denote its image under the Yoneda functor in the corresponding motivic \icat. 
		\end{dfn}
		\begin{rmk}
			In the setting of \eqref{eq:analytic_motives}, we will mostly be interested in the case where $S= \Spf(\mcO_K)^{\rig}$, for~$K$ and~$\mcO_K$ as in \Cref{not:K}. We will simply write $\RigDA(K)$ for $\RigDA(\Spf({\mcO_K)^{\rig}})$. Note that in this case by \cite[Lemma 2.4.18]{agv} the non-hypercomplete version of the construction of $\RigDA$ agrees with the one considered above, and the reader can simply ignore the difference between the two.
		\end{rmk}
		
		\begin{dfn}
			We let $\RigDA_{\gr}(K)$ be the full stable \subicat of $\RigDA(K)$ closed under colimits spanned by $\mathsf{M}(\mathfrak{X}^{\rig})$, where $\mathfrak{X}$ runs through smooth formal schemes over~$\mcO_K$. We call it the \icat of motives ``of good reduction'' or ``unipotent'' motives.
		\end{dfn}
		\begin{rmk}
			As observed in \cite[Proposition~3.29]{pWM} or \cite[proof of Theorem~2.5.34]{ayoub-rig}, the \icat $\RigDA_{\gr}(K)$ contains much more than simply motives of  varieties having a smooth formal model over~$\mcO_K$. In fact,  every formal scheme $f\colon\mathfrak{X}\to\mcO_K$ with~$f$ a \emph{pluri-nodal fibration} in the sense of \cite[Definition 1.1]{berk-contr} satisfies $\mathsf{M}(\mathfrak{X}^{\rig}) \in \RigDA_{\gr}(K)$. We remark that semistable formal schemes are pluri-nodal. The same
			is true for polystable formal schemes in the sense of \cite[Definition 1.2]{berk-contr}.
		\end{rmk}
		
		\begin{rmk}
			In case $K$ is algebraically closed, then $\RigDA_{\gr}(K)=\RigDA(K)$ (see \cite[Theorem 3.7.21]{agv}). This implies that, in the general case, for any compact motive $M$ in $\RigDA(K)$ there exists a finite field extension $L/K$ for which $M_L$ ``becomes unipotent'' in the sense that it lies in $\RigDA_{\gr}(L)$. This follows from the following formula in $\Prloo$: $$\RigDA(C)\simeq\varinjlim_{L\subseteq C, [L:K]<\infty}\RigDA(L)$$
			where $C$ is the complete algebraic closure of $K$, as proved in \cite[\S 5]{vezz-tilt4rig}, see also~\cite[Theorem~2.8.15]{agv}.
		\end{rmk}
		
		\begin{rmk}\label{rmk:gr_is_insensitive}
			As remarked in~\cite[Proposition 3.23]{pWM} the \icat of unipotent rigid analytic motives is insensitive to totally ramified extensions and is compatible with the tilting equivalence of \cite{vezz-fw}. Similarly, by semi-separatedness, we note that $\DA(k)\simeq\DA(k^{\perf})$ for any field~$k$ \cite[Th\'eor\`eme~1.2]{AyoubEt} so that the hypothesis on~$k$ of \Cref{not:K} doesn't impose any restriction. %
		\end{rmk}
		\begin{cons}
			\label{cons:chi}
			Let $j\colon \Spec(K)\to \Spec(\mcO_K)$ be the open immersion, and let $\iota\colon \Spec(k) \to \Spec(\mcO_K)$ be the complementary closed immersion. We denote by 
			\[\chi^{\alg} \colon \DA(K) \to \DA(k)\]
			the functor $\iota^* j_*$.  
			We can also define a functor
			\[
			\chi \colon\RigDA(K) \to \DA(k)
			\]
			as the right adjoint to the monoidal generic fiber functor, or ``Monsky-Washnitzer'' functor:
			\[\xi\colon \DA(k) \simeq \FDA(\mcO_K) \xrightarrow{(-)^{\rig}} \RigDA(K), \]
			where the first equivalence is \cite[Corollaire 1.4.29]{ayoub-rig}, and $(-)^{\rig}$ is the functor induced by sending an $\mcO_K$-formal scheme $\mathfrak{X}$ to its associated rigid space (see~\cite[Notation 3.8.14]{agv}). 
			Recall now that there is a monoidal analytification functor
			\[ \An^* \colon \DA(K) \to \RigDA(K),  \]
			see~\cite[Proposition 2.2.13]{agv}, and a natural transformation $\rho_K \colon \chi^{\alg} \to \chi \circ \An^*$. In view of~\cite[Corollary~3.8.18 and Theorem 3.8.19]{agv}, this natural transformation is invertible, and we will identify $\chi \An^*$ and $\chi^{\alg}$ as functors $\DA(K)\to \DA(k)$.
			In particular, $\chi\one\simeq\chi^{\alg}\one$ as objects in $\CAlg(\DA(k))$.
		\end{cons}

		We collect some important properties of $\RigDA_{\gr}(K)$ proved in \cite{agv} (see \cite{ayoub-rig} for the equi-characteristic zero analogue) in the following Proposition.
		\begin{prop}\label{prop:facts_of_life} Let $K$ be as in \Cref{not:K}. %
			\begin{enumerate}
				\item The generic fiber functor $\xi$ refines to an equivalence of monoidal \icats
				\[\tilde{\xi}\colon\Mod_{\chi\one}(\DA(k)) \xrightarrow{\simeq} \RigDA_{\gr}(K).\]
				\item The choice of a pseudo-uniformizer $\varpi\in\mcO_K$ identifies $\chi\one\in\CAlg(\DA(k))$ with the split square-zero extension $\chi\one\simeq\one\oplus \one(-1)[-1]=\Gmd$.
			\end{enumerate}
			\end{prop}
			
			\begin{proof}
				The first statement is a special case of~\cite[Theorem 3.3.3(1)]{agv}.
				The second statement follows immediately from~\cite[Theorem 3.13]{ayoub-weil}.%
			\end{proof}

  \begin{rmk}
      More generally, \cite[Theorem 3.13]{ayoub-weil} can be used to see that the category $\RigDA_{\gr}(K)$ (which equals $\RigDA(K)$ if $K$ is algebraically closed)   is completely determined by the residue field and the rational rank of the value group (assuming it is finite). In particular, it is insensitive to immediate extensions.
  \end{rmk}
 
		The following definition is inspired by the category of $(\varphi,N)$-modules considered in \Cref{sec:phiNmod}.
		\begin{dfn}
			We denote by  $\DAN$ the \icat $\LFnil{\DA(k)}{-\otimes\one(-1)}$ (see \Cref{dfn:LFnil}). Informally, its objects are given by pairs $(M,N_M)$ with $M$ an object of $\DA(k)$ and $N_M$ a map (the \emph{monodromy}) $N_M\colon M\to M(-1)$ which is ind-nilpotent. 
		\end{dfn}
  We note that on compact objects the condition on ind-nilpotency is automatically fullfilled.
		\begin{prop}\label{cor:DANil_is_indDAN}
			The \icat $\DAN$ is equivalent to $\Ind\left((\LF{\DA(k)_{\cpt})}{-\otimes\one(-1)}\right)$. In particular, if $M$ is compact in $\DA(k)$, any map $N\colon M\to M(-1)$ is nilpotent. 
		\end{prop}
		
		\begin{proof}
			We let $M=\mathsf{M}_k(X)$ be the motive of a proper smooth variety over $k$. As these motives are in the heart $\heart{k}$ of a bounded weight structure on $\DA(k)$ and since, if  $t^{\otimes n}\otimes M=(M(-n)[-2n])[2n]$ is a shift of an object in $\heart{k}$, we deduce that $\Hom(M,t^{\otimes n}\otimes M)=0$ whenever $n>0$. We then conclude using the criterion of \Cref{lem:CTNil-vsIndCat}.
		\end{proof}
		
		\begin{cor}\label{cor:RigDAisDAN}
			Let $K$ be as in \Cref{not:K}. 
				A choice of a pseudo-uniformizer $\varpi\in\mcO_K$ determines an equivalence of monoidal \icats $\RigDA_{\gr}(K)\simeq\DAN$.
			\end{cor}
			\begin{proof}
				Use \Cref{sta:Mod-fix} and \Cref{prop:facts_of_life}.
			\end{proof}
			
			\begin{dfn}
				We denote the equivalence $\RigDA_{\gr}(K)\simeq \DAN$ by $M\mapsto (\Psi M,N_{\Psi M})$. In particular, we denote the composite $\RigDA_{\gr}(K)\simeq\DAN\xto{\pi}\DA(k)$ by $\Psi$.
			\end{dfn}
			
			\begin{rmk}
				\label{rmk:nearby-cycles} Assume that $K$ is discretely valued and $\varpi$ is a uniformizer.
				If we denote by $\DA_{\gr}(K)$  the full \subicat of $\DA(K)$ given by the inverse image of $\RigDA_{\gr}(K)$ along the analytification functor $\An^*$ then the composite
				$$\DA_{\gr}(K)\xto{\An^*} \RigDA_{\gr}(K)\xto{\Psi}\DA(k)$$
				can be shown to coincide with the functor $\Psi_{\varpi}=\Upsilon_{\varpi}$ of (unipotent) motivic nearby cycles (in equi-characteristic~$0$ this is proved in \cite[Scholie 1.3.26]{ayoub-rig}).
			\end{rmk}
			
			\begin{rmk}
				\label{rmk:choice?}
				If the residue field $k$ is contained in~$\bar{\F}_p$ then there is a \emph{canonical}  augmentation $\varepsilon\colon\chi\one\to\one$ inducing a \emph{canonical} equivalence $\RigDA_{\gr}(K)\simeq\LFnil{\DA(k)}{T}$ where $T$ is now the functor $\otimes\fib(\varepsilon)[1]$. Indeed, such augmentations form a torsor under the group $\Hom_{\DA(k)}(\one,\one(1)[1])\cong k^*\otimes_{\Z}\Q$ (see e.g. \cite[Corollary 4.2]{mvw} or \cite[Corollaire 11.4]{AyoubEt}) which is trivial. %
			\end{rmk}
			
			\begin{rmk}\label{rmk:changepi?}
				Replacing $\varpi$ with another pseudo-uniformizer of the form $\lambda\varpi^{q}$ with $q\in \Q_{>0}$ and $\lambda\in\mcO_K^*$ induces the automorphism of $\one\oplus\one(-1)[-1]$ given by the matrix $\left(\begin{smallmatrix}
					1&\lambda\otimes1\\0&q
				\end{smallmatrix}\right)$ under the equivalence $\Hom(\one(-1)[-1],\one)\cong k^*\otimes_{\Z}\Q$.   
				This can be deduced from its construction~\cite[Corollary~3.8.32]{agv}: an exponentiation map $T\mapsto T^n$ on $\G_{m,k}$ induces the automorphism $\left(\begin{smallmatrix}
					1&0\\0&n
				\end{smallmatrix}\right)$  on $\mathsf{M}(\G_m)=\one\oplus\one(1)[1]$ (see the proof of \Cref{thm:KummerisTate} below) and the multiplication by $\lambda$ extends to an automorphism of~$\A^1$. In particular, if $k\subseteq\bar{\F}_p$, in which case $k^*\otimes\Q=0$, then the isomorphism $\chi\one\cong\one\oplus\one(-1)[-1]$ in $\Ho(\DA(k))$ actually depends only on the \emph{valuation}~$|\varpi|$ and two different choices would lead to a multiplication of the monodromy operator by a rational constant. 
				Whenever $K$ is of mixed characteristic $(0,p)$, one can consider the ``canonical''  choice $\varpi=p$. %
			\end{rmk}
			
			\begin{rmk}
				As observed in \cite[Lemme 11.5]{AyoubEt}, one has $\Alt^2(\one(-1))=0$ in $\Ho(\DA(k))$ and moreover compact objects in $\DA(k)$ are strongly dualizable (see e.g. \cite{riou-dual}). In particular, in light of \Cref{rmk:hypholds}, \Cref{hyp:F-package} holds in  case $\catC=\DA(k)$ and $t\simeq\one(-1)$. We can then use freely the results of \Cref{sec:detection-principles} in what follows.
			\end{rmk}

			\begin{rmk}\label{rmk:chi_isFibN_xi_is_pirn}
                        Fixing an equivalence $\RigDA_{\gr}(K)\simeq\DAN$ we may extend the table of \Cref{rmk:tabsummary} as follows,  where~$\mathscr{L}\text{og}^\vee$ is as in \Cref{rmk:kummerUDA}.
                        \renewcommand*{\arraystretch}{1.5}
				\begin{table}[H]
					\centering
					\begin{tabular}{| >{\centering\arraybackslash}m{5cm}| >{\centering\arraybackslash}m{5cm}|}\hline

                                          $\RigDA_{\gr}(K)$&$\DAN$\\
                                          \hline{}
                                          &\\[-2.5ex]
						{$
                                          \begin{aligned}
                                            \xi&\dashv \chi\\
                                            \Psi&\dashv \xi\otimes\mathscr{L}\text{og}^{\vee}
                                          \end{aligned}$}
                                               &
                                                {$
                                                 \begin{aligned}
                                                   (N=0) &\dashv \fib(N) \\
                                                   \pi&\dashv\pirn
                                                 \end{aligned}$}
\\\hline
					\end{tabular}
                                        \end{table}
                                        \renewcommand*{\arraystretch}{1}
			\end{rmk}
			
			\subsection{The Kummer motive and the Tate curve}
			\label{sec:KummerinRig}
			We now show that under the equivalence $\RigDA_{\gr}(K)\simeq\DAN$ of \Cref{cor:RigDAisDAN}, the Kummer motive (see \Cref{def:kummerobj}) appears as the $\Hm^1$-component of a Tate curve. This will be important to deduce information on the monodromy of an arbitrary rigid analytic motive from the monodromy of the Tate curve.
			
			\begin{prop}\label{thm:KummerisTate}
				Fix an identification $\RigDA_{\gr}(K)\simeq \DAN$  induced by a pseudo-uniformizer~$\varpi$ and let $\mcK$ be the Kummer motive in $\RigDA_{\gr}(K)$.  Let $\mathcal{T}_\varpi$ be the Tate curve~$\G_{m,K}^{\an}/(\varpi^\Z)$. The motive $\mathsf{M}_{K}(\mathcal{T}_\varpi)\in\RigDA_{\gr}(K)$ admits a decomposition
				\begin{equation}\label{eq:Tate_curve_splits}
				\mathsf{M}_{K}(\mathcal{T}_\varpi)\simeq \one \oplus \mcK(1)[1]\oplus \one(1)[2].
				\end{equation}
			\end{prop}

			\begin{proof}
				The proof will be divided into three steps.\\[.2cm]
				{\it Step 1}: We first reduce to the equi-characteristic case. {In light of \Cref{rmk:gr_is_insensitive}} we may and do assume that $K$ is perfectoid and that it has all the $p$-th roots of $\varpi$ which assemble to an element~$\varpi^\flat$ in $K^\flat\cong \varprojlim_{x\mapsto x^p} K$. Let $\mathsf{M}(\G_{m,K}^{\an})_{h\Z}$ be  the  colimit of the diagram $B\Z\to\RigDA_{\gr}(K)$ induced by the automorphism of~$\mathsf{M}(\G_{m,K}^{\an})$ given by the multiplication by $\varpi$. By étale descent, it coincides with the motive of $\mathcal{T}_\varpi$.
        Under the equivalence $\RigDA_{\gr}(K)\simeq \DAN\simeq \RigDA_{\gr}(K^\flat)$ the motive of $\G^{\an}_{m,K}$ corresponds to the motive of $\G^{\an}_{m,K^{\flat}}$ and the multiplication by $\varpi$ corresponds to the multiplication by $\varpi^\flat$ (see \cite[Propositions 7.5 and 7.6]{vezz-fw}).  Hence, the  motive $\mathsf{M}_{K}(\mathcal{T}_\varpi)$ corresponds to $\mathsf{M}_{K^{\flat}}(\mathcal{T}_{\varpi^\flat})$. We are then reduced to proving the statement for $K^\flat$ which is an extension of $k(\!(\varpi^\flat)\!)$. We may then prove the statement for $K=k(\!(T)\!)$ and $\varpi=T$.\\[.2cm]
				{\it Step 2}: Write now $\mathsf{M}(\G_m)_{h\Z}$ for the colimit of the diagram $B\Z\to\DA(k[T^{\pm1}])$ induced by the automorphism~$T$ of~$\mathsf{M}(\G_m)$. We show that $\mathsf{M}(\G_m)_{h\Z}$ decomposes in $\DA(k[T^{\pm1}])$ as in \eqref{eq:Tate_curve_splits}. In this category, the multiplication by $T$ on $\mathsf{M}(\G_{m})$ can be factored as follows:
				$$
				\mathsf{M}(\G_{m})=\mathsf{M}(\G_m)\otimes\one\xto{\id\otimes T}\mathsf{M}(\G_{m})\otimes \mathsf{M}(\G_m)\xto{\mu}\mathsf{M}(\G_{m})
				$$
				where  $\mu $ is the multiplication map. Using the decomposition $\mathsf{M}(\G_m)=\one\oplus\one(1)[1]$ induced by the rational point $1\in\G_m$, giving rise to a decomposition $\mathsf{M}(\G_m)\otimes \mathsf{M}(\G_m)\cong \one\oplus \one(1)[1]\oplus\one(1)[1]\oplus\one(2)[2]$ the map above is represented by the matrix (for the matrix of $\mu$, see \cite[Proposition 2.6]{AyoubEt})
				$$
				\begin{pmatrix}
					\id&0&0&0\\0&\id&\id&0
				\end{pmatrix}
				\begin{pmatrix}
					\id&0\\T&0\\0&\id\\0&T
				\end{pmatrix}=\begin{pmatrix}
					\id&0\\T&\id
				\end{pmatrix}
				$$
				where here we let $T$ represent the morphism $\one\to\one(1)[1]$ induced by $T$ under the isomorphism $\Hom(\one,\one(1)[1])\cong k[T^{\pm1}]^*\otimes_{\Z}\Q$. Since $B\Z=S^1=\Delta^1/\partial\Delta^1$, %
				 the object $\mathsf{M}(\G_m)_{h\Z}$ coincides with the  coequalizer  of~$\id $ and the multiplication by~$T$ on~$\mathsf{M}(\G_m)$.  As such, it fits in an exact triangle in the homotopy category,
				$$
				\one\oplus\one(1)[1]\xto{\left(\begin{smallmatrix}
						0&0\\T&0
					\end{smallmatrix}\right)}\one\oplus\one(1)[1]\to \mathsf{M}(\G_m)_{h\Z}\to ,
				$$
				giving rise to an isomorphism $\mathsf{M}(\G_m)_{h\Z}\cong\one\oplus\one(1)[2]\oplus {\mcK}(1)[1]$ with ${\mcK}$ the extension 
				$$\one\to{\mcK}\to\one(-1)\to$$
				induced by $T$. Note that this object is the Kummer object of $\UDA(k)$, see \Cref{rmk:kummerUDA}.\\[.2cm]	
				{\it Step 3:} We now conclude the proof. Let $\alpha$ be the map $\Spec k(\!(T)\!)\to \Spec k[T^{\pm1}]$. We note that $\An^* \alpha^*(\mathsf{M}(\G_m)_{h\Z})$ in $\RigDA(k(\!(T)\!))$  is $\mathsf{M}(\G_m^{\an})_{h\Z}$. From the previous steps we then deduce an equivalence $\mathsf{M}(\mathcal{T}_T)\simeq \mathsf{M}(\G_m^{\an})_{h\Z} \simeq \one\oplus\one(1)[2]\oplus\An^*\alpha^*{\mcK}$. Note that $\An^*\alpha^*$ induces an equivalence $\UDA(k)\simeq\RigDA_{\gr}(k(\!(T)\!))$ which is compatible with the equivalence $\RigDA_{\gr}(k(\!(T)\!))\simeq \DAN$ fixed in the statement. This shows that $\mcK$ is sent to the Kummer object via $\An^*\alpha^*$, as wanted.
			\end{proof}

			\begin{rmk}Note that every Tate curve is an elliptic curve with nodal reduction.  In fact, every elliptic curve $E$ over $K$ that satisfies $|j(E)|>1$ is isomorphic after a finite extension to a Tate curve, and (Tate's theorem) an elliptic curve is isomorphic to a Tate curve if and only if it has (split) multiplicative reduction (see e.g. \cite[Theorem 5.8]{silverman_adva}).
			\end{rmk}	
			
			\begin{rmk}
				Let $C$ be the completion of the algebraic closure of $k(\!(T)\!)$. Recall from~\cite[Proposition 2.34]{pWM} that the \icat $\RigDA(C)$ is a localization of the \icat of log motives~$\log\DA(k^0)$, where $k^0$ denotes the field $k$ with canonical log structure. From 
				a log perspective, the motive of the Tate curve corresponds to the log motive of its reduction over the log point $k^0$. %
				Since the Tate curve over an algebraically closed field has nodal reduction, its normalization is $\P^1$.  Under a suitable $cdh$ descent property, the log motive of the reduction of the Tate curve should be a Tate motive built out of $\mathbb{P}^1$ minus three points. 
			\end{rmk}

			\subsection{The weight structure on rigid motives}\label{sec:WSonRig}
			We now find the heart of a bounded compatible weight structure on $\RigDA_{\gr}(K)$ (see \Cref{sec:W}) expanding some arguments of \cite[Appendix]{pWM}. We recall that we denote by $\heart{k}\subset\DA(k)$  the \icat of Chow motives over~$k$, which is the heart of a bounded compatible weight structure on~$\DA(k)$ (see \Cref{exm:WS}\eqref{exm:ChowWS}).
			\begin{prop}\label{prop:WSonRig}
				Let  $\heart{K}\subset\RigDA_{\gr}(K)$ be 
				the full \subicat spanned by the image of~$\heart{k}$ under the functor~$\xi$.
				\begin{enumerate}
					\item \label{WSR1}$\heart{K}$ is the heart of a bounded compatible weight structure on $\RigDA_{\gr}(K)_{\cpt}$. 
					\item \label{WSR2} The functors $\Psi$ and $\xi$ define quasi-inverse monoidal equivalences $\Ho(\heart{k})\simeq \Ho(\heart{K})$
				\end{enumerate}
			\end{prop}
			
			\begin{proof}The image of $\xi$ clearly generates $\RigDA_{\gr}(K)$ under colimits, and $\heart{k}$ generates $\DA(k)$ under shifts and colimits. This proves that $\heart{K}$ generates $\RigDA(K)$ under shifts and colimits.   As $\xi $ is monoidal, $\heart{K}$ is also closed under tensor products, and contains the unit $\one$. 
				Observe that for every $M,N\in\DA(k)$,
				\begin{align*}
					\spMap_{\RigDA(K)} (\xi M, \xi N)&
					\simeq     \spMap_{\DA(k)} (M, \chi\xi N)\\
					&\simeq\spMap_{\DA(k)}(M,N)\oplus \spMap_{\DA(k)}(M,N(-1)[-1]).
				\end{align*}
				and the functors $\xi$ resp. $\Psi$ induce the canonical inclusion resp. the projection to the direct summand $\spMap_{\DA(k)}(M,N)$ as they are induced by the unit $\one\to \one\oplus\one(-1)[-1]$ resp. the augmentation $\one\oplus\one(-1)[-1]\to\one$.
				
				If we apply this to $M,N$ in $\heart{k}$, and use that $N(-1)[-1]=(N(-1)[-2])[1]$ we see that in this case both summands are connective spectra so that $\heart{K}$ is negative. As the second spectrum is even  1-connective,   we deduce~\eqref{WSR2}. Note that this implies that $\heart{K}$ is (additive and) idempotent-complete, thus proving~\eqref{WSR1} by means of \Cref{exm:WS}\eqref{exm:WSonComega}.
			\end{proof}

	\begin{rmk}
	The fact that motives of the form $\xi M$ with $M$ a Chow motive over $k$ are not only (compact) generators of  $\RigDA_{\gr}(K)_\omega$, but that they even generate the heart of a weight structure (that is, their orthogonality property) admits very concrete applications. For instance, one can attach to any compact motive in  $\RigDA_{\gr}(K)$
 a functorial weight complex  in $\mcK^b(\Ho(\heart{}))$ (see \Cref{eg:ws}\eqref{wc}). Moreover, as we will see, one can produce convergent  spectral sequences on many ``classical''  cohomological realizations, thus giving rise to (functorial) weight filtrations on cohomology groups, which are even compatible (in a suitable sense) with the monodromy operator (see \Cref{rmk:weight_filtration} and \Cref{eg:mononWC}).   %
	\end{rmk}

			\subsection{Motivic $(\varphi,N)$-modules}\label{sec:motivicphiN}
			In this section, we assume that $k$ has positive characteristic~$p$, and use \Cref{notn:K_0}.
			We continue to denote by~$\varphi\colon K_0\to K_0$ and by $\varphi\colon k(\!(T)\!)\to k(\!(T)\!)$ an arithmetic Frobenius with reduction~$\varphi:k\to k$. We will show that there is a canonical enrichment $\RigDA_{\gr}(K)\to \DA_{(\varphi,N)}(k)$ where the target is the \icat of \emph{motivic} $(\varphi,N)$-modules (that we will introduce below) similar to \Cref{sec:phiNmod}.
			
	\begin{dfn}\label{dfn:Dphi}
	    Let $\catD$ be a monoidal stable compactly generated \icat equipped with a monoidal auto-equivalence $\varphi\colon \catD\isoto\catD$. We denote by  $\hF{\catD}{\varphi}$   the equalizer of~$\id$ and~$\varphi$ computed in~$\Prloo$. Thus its compact objects are given by pairs $(M,f)$ with $M$ a compact object of~$\catD$ and $f\colon M\isoto\varphi M$. In particular,   we let $\hF{\DA(k)}{\varphi}$ be $\hF{\DA(k)}{\varphi^*}$.
	\end{dfn}

   		The following is the algebraic analogue of~\cite[\S\,2.3]{LBV} and~\cite[Remark~2.32]{pWM}.
			\begin{rmk}
				\label{rmk:relative-Frob-modp}
				We use the terminology of~\cite[\href{https://stacks.math.columbia.edu/tag/0CC6}{Tag 0CC6}]{stacks-project}.
				The relative Frobenius defines a natural transformation $X\to \varphi^*X$ on (smooth) $k$-schemes which induces a natural transformation $\id\to \varphi^*$ of endofunctors of $\DA(k)$.
				Since all Frobenius maps are universal homeomorphisms of schemes, this natural transformation is an equivalence, see~\cite[Proposition~6.3.16]{MR3477640}.
				In particular, we obtain a monoidal functor $\rF\colon \DA(k)\to\hF{\DA(k)}{\varphi}$ inducing a monoidal functor 
    \begin{equation}\label{eq:enrich}
 \LFnil{\DA(k)}{-\otimes\one(-1)}
 \xto{\rF} 
 \LFnil{(\hF{\DA(k)}{\varphi})}{-\otimes \rF(\one(-1))}\simeq \Mod_{\rF(\chi\one)}(\DA(k)^\varphi)
    \end{equation}
    where the last equivalence (depending on a choice of a pseudo-uniformizer) follows from the fact that $\rF(\chi\one)$ is the free symmetric algebra on $\rF(\one(-1)[-1])$ (see \Cref{prop:Gmissplit}).
			\end{rmk}

 We thank an anonymous referee for suggesting the following.
\begin{lemma}\label{lemma:K0phinew}
    Under the canonical equivalence $\varphi^*\one(-1)\simeq\one(-1)$, the object $\rF( \one(-1))$ in $\hF{\DA(k)}{\varphi}$ is given by the map $$p^{-1}\colon \one(-1)\to  \varphi^*\one(-1)\simeq\one(-1).$$
 \end{lemma}

  \begin{proof}
      Using the identification $\Gmd\simeq\one\oplus\one(-1)[-1]$, it is enough to identify the action of the relative Frobenius on $\Gmd$ or, by duality, on $\mathsf{M}(\G_m)\simeq\one\oplus\one(1)[1]$. At this level,   the relative Frobenius on $\mathbb{G}_m$ is represented by the matrix $\left(\begin{smallmatrix}
  					\id&0\\0&p
			\end{smallmatrix}\right)$ hence the claim.
  \end{proof}

	In light of \Cref{lemma:K0phinew}, we may introduce the following notation.
			\begin{dfn}\label{cons:DAphiN}
    We    denote the \icat $\LFnil{(\hF{\DA(k)}{\varphi})}{-\otimes \rF(\one(-1))}$ by $\DA_{(\varphi,N)}(k)$, whose objects are ``motivic $(\varphi,N)$-modules''. Informally, its compact objects are given by commutative squares of compact objects in $\DA(k)$:
       			$$
        			\begin{tikzcd}
        				M \arrow[r,"\sim", "\alpha"'] \arrow[d,"N"]&\varphi^* M\arrow[d,"\varphi^* N"]\\
        				M(-1) \arrow[r,"\sim","\alpha\cdot\frac{1}{p}"'] &\varphi^* M (-1).
        			\end{tikzcd}
				$$
    \begin{prop}\label{prop:K0phinew}
         The choice of a pseudo-uniformizer  of $K$ determines a functor in $\Prloo$:
				$$\RigDA_{\gr}(K)\simeq\RigDA_{\gr}(K_0)\to \DA_{(\varphi,N)}(k).$$
    \end{prop}

    \begin{proof}  The  domain of the canonical functor  $$
 \LFnil{\DA(k)}{-\otimes\one(-1)}
 \xto{\rF} 
 \LFnil{(\hF{\DA(k)}{\varphi})}{-\otimes \rF(\one(-1))}=\DA_{(\varphi,N)}(k)
    $$
    is identified, using the chosen pseudo-uniformizer, with $\RigDA_{\gr}(K)$ via \Cref{cor:RigDAisDAN}. 
    \end{proof}
				Informally, the functor above sends $M$ to the square
				$$
				\begin{tikzcd}
					\Psi M \arrow[r,"\sim", "\alpha"'] \arrow[d,"N"]&\varphi^* \Psi M\arrow[d,"\varphi^* N"]\\
					\Psi M(-1) \arrow[r,"\sim","\alpha\cdot \frac{1}{p}"'] &\varphi^* \Psi M (-1)
				\end{tikzcd}
				$$
				where $\alpha$ is now induced by the relative Frobenius map. Note that such a diagram expresses at the motivic level the relation $ N \varphi \simeq p\varphi N$ that holds in the classical abelian category  of $(\varphi, N)$-modules over $K_0$, see \Cref{sec:phiNmod} including \Cref{rmk:choice-of-phi^*}.
			\end{dfn}

\alberto{

}

                \subsection{Extending realization functors}\label{sec:extend}
			\Cref{cor:RigDAisDAN} gives a way to define, starting from any (co)homology theory ``on the special fiber'', a (co)homology theory ``on the generic fiber'', and this will be equipped with a monodromy operator. We now explain how, and give some examples. For the definition of $\LFnil{\catD}{-\otimes u}$ and the functors $\pi,(N=0)$ we refer to \Cref{sec:nilpotent-op}.
			
			\begin{cor}\label{cor:RGammahat}
				Let $F\colon\DA(k)\to \catD$ be a  colimit-preserving monoidal functor between stable presentably monoidal   \icats. Let $u=F(\one(-1))\in\catD$. The choice of a pseudo-uniformizer~$\varpi\in\mcO_K$ gives rise to a monoidal extension
				$$
				\widehat{F}\colon\RigDA_{\gr}(K)\to \LFnil{\catD}{-\otimes u}
				$$
				endowed with equivalences $\widehat{F}\circ\xi\simeq (N=0)\circ F$ and $\pi\circ\widehat{F}\simeq F\circ \Psi$. Informally, it sends the motive $M$ to $(F(\Psi M), F(N_{\Psi M}))$. \qed
			\end{cor}
   \begin{proof}
As in \Cref{notation:comparison_monod}, the functor $F$ induces a  monoidal functor $\widehat{F}\colon \DAN \to \LFnil{\catD}{-\otimes u}$. We conclude by \Cref{cor:RigDAisDAN}.
   \end{proof}

			\begin{rmk}
			Note that, since colimits are reflected by $\pi$, the functor $\widehat{F}$ is in $\CAlg(\Prl)$. Similarly, since $\pi$ reflects compact objects, if $F$ is a functor in $\Prloo$ then so is~$\widehat{F}$.
			\end{rmk}
			
			\begin{exm}\label{exm:NonWC}
				Let $w\colon\DA(k)_{\cpt}\to \mcK^b(\Ho(\heart{k}))$ be the (monoidal) weight complex functor. We deduce that there is a monoidal functor $\widehat{w}\colon \RigDA_{\gr}(K)_{\cpt}\to \LFnil{\mcK^b(\Ho(\heart{k}))}{(-1)}$ with $(-1)$ being the functor $C\mapsto (C\{-1\})[2]$ (i.e. $C(-1)_n=C_{n-2}(-1)[-2]$). Note that the underlying functor $\pi\circ\widehat{w}\simeq w\circ\Psi$ coincides with the weight complex functor of~$\RigDA_{\gr}(K)_{\cpt}$ (since $\Psi$ induces an equivalence on the weight complexes, see \Cref{{prop:WSonRig}}). In other words, the weight complex $w X$ of a motive~$X$ in~$\RigDA_{\gr}(K)_{\cpt}$ is  equipped with a monodromy operator $w X\to w X(-1)$, represented by a map of complexes levelwise given by $(wX)_n\to (wX)_{n-2}(-1)[-2]$.
			\end{exm}
			
			\begin{exm}\label{eg:mononWC}
				We keep the notation of \Cref{exm:NonWC}.    Suppose that $F$ as in \Cref{cor:RGammahat} factors over the weight complex functor (on compact objects) as	$$\DA(k)_{\cpt}\xto{w}\mcK^b(\Ho(\heart{k}))\xto{\widetilde{F}}\catD
				$$
				(for example, under the hypotheses of \Cref{exm:Ai}). We deduce that the functor~$\widehat{F}$ (on compact objects) factors as
				$$
				\RigDA_{\gr}(K)_{\cpt}\xto{\widehat{w}}\LFnil{\mcK^b(\Ho(\heart{k}))}{(-1)} \xto{\widetilde{F}}\LFnil{\catD}{(-1)}.$$
    where we write $(-1)$ for $-\otimes u$ in $\catD$, for simplicity. 
				We now place ourselves in the setting of \Cref{exm:Ai,rmk:weight_filtration}. %
				Note that in this case the monodromy operator $N$ defines a map between the two weight spectral sequences for any~$X$ in~$\RigDA_{\gr}(K)_{\cpt}$:
				\[ E^1_{pq}=\Hm_p(F(w X)_q)\Rightarrow \Hm_{p+q}(\widehat{F}X), \quad E^1_{pq}=\Hm_p(F(w( X(-1)))_q)\Rightarrow \Hm_{p+q}(\widehat{F}X(-1))\]
				given by 
				\[\Hm_p(F(wX)_q)\to \Hm_p(F(w (X(-1)))_q)\cong \Hm_{p+2}(F(wX)_{q-2})(-1).\]
				As such, it induces a map between the induced graded pieces for the weight filtration $\gr^{w}_i\Hm_n\widehat{F}(X)\to \gr^{w}_{i+2}\Hm_n\widehat{F}(X)(-1)$.
			\end{exm}

			\begin{cor}\label{cor:RGammahat2}Fix the hypotheses of \Cref{cor:RGammahat}. Assume moreover the following:
				\begin{enumerate}
					\item $k$ is a perfect field of characteristic $p>0$.
                    \item $\catD$ is compactly generated and $F$ preserves compact objects.
					\item $\catD$  is equipped with a monoidal auto-equivalence $\varphi\colon\catD\xto{\sim}\catD$.
					\item $F$ is $\varphi$-equivariant, i.e., there is an invertible transformation $F\circ\varphi^*\simeq \varphi\circ F$ of monoidal functors.
				\end{enumerate}
				Then there is a monoidal extension (depending on the choice of a pseudo-uniformizer~$\varpi\in\mcO_K$)
				$$
				\widehat{F}\colon\RigDA_{\gr}(K)\to \LFnil{(\hF{\catD}{\varphi})}{-\otimes u}
				$$
				where $u=F\left(\one(-1)\xrightarrow[\rF]{\sim}\varphi^{*} \one(-1)\right)\simeq\left(F\one(-1)\xto{\sim}\varphi(F\one(-1))\right)$. 
				Informally, it sends the motive $M$ to $(F(\Psi M), F(\alpha_M), F(N_{\Psi M}))$ (see \Cref{cons:DAphiN}). 
			\end{cor}
			\begin{proof}
                        The functor $\hF{F}{\varphi}\colon \hF{\DA(k)}{\varphi}\to \hF{\catD}{\varphi}$ induces a functor $\widehat{F}^\varphi\colon \DA_{(\varphi,N)}(k)\to   \LFnil{(\hF{\catD}{\varphi})}{-\otimes u}$ as in \Cref{notation:comparison_monod}. We can then pre-compose with the functor $\RigDA_{\gr}(K)\to \DA_{(\varphi,N)}(k)$ of \Cref{prop:K0phinew}.
			\end{proof}

			\begin{cor}\label{cor:RGammas}
				Let $R\Gamma\colon\RigDA_{\gr}(K)_{\cpt}\to \catD$ be a monoidal exact functor to a stable presentably monoidal \icat. If $\pi_n\spMap_{\catD}(R\Gamma\xi X,R\Gamma \xi Y)=0$ for  any $X,Y$ in $ \heart{k}$ and any $n>0$,  then $R\Gamma\simeq \pi \circ \widehat{R\Gamma\circ \xi}$ as monoidal functors. 
			\end{cor}
			\begin{proof}
				By construction, the right-hand side is equivalent to $R\Gamma \circ \xi\circ\Psi$, so that it is enough to show  $R\Gamma \simeq R\Gamma \circ \xi\circ\Psi$.   In light of \Cref{prop:factorsonwc} this follows from the fact that $\xi$ and $\Psi$ induce an equivalence $\Ho(\heart{k})\simeq\Ho(\heart{K})$ (see \Cref{prop:WSonRig}) and hence an equivalence $\mcK^b(\Ho(\heart{k}))\simeq\mcK^b(\Ho(\heart{K})).$
			\end{proof}
			
			\begin{rmk}
				The equivalence of \Cref{cor:RGammas} can be suggestively restated by the formula $R\Gamma(X) \simeq R\Gamma(\Psi X)$ for every (unipotent) rigid analyitic motive $X$ over $K$. This is in accordance with the formulas given for the three cohomology theories considered in the introduction (see for example \cite{illusie,mokrane}). 
			\end{rmk}

		\subsection{A ``new'' Hyodo--Kato realization}
			\label{sec:dR-realization}
			Let $K$ be as in \Cref{not:K}, and assume moreover that the residue field~$k$ is a perfect field of characteristic $p>0$. As in \Cref{sec:phiNmod} we let $K_0$ be $W(k)[1/p]$ and we  denote by $\Dp$ (resp.\ $\DpN$) the Ind-completion of the bounded derived \icat of $\varphi$-modules (resp.\ of $(\varphi,N)$-modules). 
			Recall the overconvergent de\,Rham realization~$\dR\colon \RigDA(K_0)\to\QCoh(\Spa K_0)^{\op}$ from~\cite[Corollary 4.39]{LBV}. 
			
			\begin{dfn}\label{def:hom_real_from_coho}
			\begin{enumerate}
			    \item Assume $\car K=0$.    We let $R\Gamma_{\dR,K}^{\dagger}\colon \RigDA(K)_{\cpt}\to\mcD( K)$ be the covariant realization defined on compact objects by pre-composing $\dR$ with the canonical duality $M\mapsto M^\vee$.\footnote{Here, we use that every compact object in~$\RigDA(K)$ is dualizable, as follows from~\cite[Proposition~2.31]{ayoub-neww} and~\cite{riou-dual}.} We still denote by $R\Gamma_{\dR}^{\dagger}$ its monoidal, compact-preserving, colimit-preserving extension to $\RigDA(K)\simeq\Ind(\RigDA(K)_{\cpt})$ Informally, $R\Gamma_{\dR}^{\dagger}(\varinjlim M_i)=\varinjlim \dR(M_i^\vee)$.
       \item  We let $R\Gamma_{\rig}$ be the composite
				$$
R\Gamma_{\rig}\colon\DA(k)\xto{\xi}\RigDA(K_0)\xto{R\Gamma^\dagger_{\dR,K_0}} \mcD( K_0).$$
			\end{enumerate}	
			\end{dfn}
			
                          \begin{rmk}
                          As the functor $R\Gamma_{\rig}$ is $\varphi$-equivariant, it extends to a functor (\Cref{prop:DphiN})
	$$
 \DA(k)^\varphi\xto{R\Gamma_{\rig}} \mcD( K_0)^\varphi\simeq\mcD_\varphi( K_0).$$
                              We point out that the image of $\rF(\one(-1))$ under this functor is equivalent to~$K_0(-1)\in\mcD_{\varphi}(K_0)$, that is, the $\varphi$-module~$K_0$ with the morphism given by $p\cdot\varphi\colon K_0\isoto \varphi_*K_0$, where $\varphi$ is the Frobenius on~$K_0$ (see \Cref{ex:phimods}). %
                              The fact that this is the image of $\rF(\one(-1))$ follows from (the proof of) \Cref{lemma:K0phinew}, see also \Cref{rmk:choice-of-phi^*}.

                              In particular, we can make the following definition.
                          \end{rmk}

			\begin{dfn}
				 We let $\widehat{R\Gamma}_{\rig}$ be the functor
				\[
				\widehat{R\Gamma}_{\rig}\colon \RigDA_{\gr}(K)\to \LFnil{\Dp}{T} \simeq \DpN
				\]
				induced by a pseudo-uniformizer $\varpi\in W(k)$, as in \Cref{cor:RGammahat2}, where $T$ is the Tate twist endofunctor $D\mapsto D(-1)$ (see \Cref{prop:DphiN}).
			\end{dfn}
			\begin{rmk}As observed in
				\Cref{rmk:changepi?}, a change in the choice of the pseudo-uniformizer~$\varpi$ has the effect of multiplying the monodromy by a rational constant. For example, if $\varpi$ is a uniformizer of a finite extension $K$ of $\Q_p$, then $\widehat{R\Gamma}{}^\varpi_{\rig}(X)$ and ${\widehat{R\Gamma}}{}^p_{\rig}(X)$ agree as $\varphi$-modules, but the monodromy operator on the latter is given by the  monodromy operator on the former divided by the absolute ramification index $e_K$. Choosing $p$ is coherent with the classical normalization of the monodromy to make it compatible with the base-change. See \cite[Section 3A, p. 1722]{nekovar_niziol}.
			\end{rmk}

			The cohomology $\widehat{R\Gamma}_{\rig}$ is our ``new'' Hyodo--Kato cohomology for rigid analytic varieties.
			Technically, the cohomology theory $R\Gamma_{\rig}$ constructed above is the so-called Monsky--Washnitzer cohomology, based on the choice of a local weak-formal lift for a scheme over $k$. It is well-known to compare to  rigid cohomology (see also \cite{besser}, \cite{vezz-MW}) as we recall here below, thus justifying our notation.
			\begin{prop}
				\label{prop:ispure}
The following holds.
				\begin{enumerate}
					\item The functor $\pi\circ \widehat{R\Gamma}_{\rig}\circ\xi  \colon \DA(k)\to \mcD_{\varphi}(K_0)$ computes rigid cohomology, equipped with its canonical $\varphi$-structure.
					\item If $k$ is finite, then for any $M\in \heart{k}$ the $\varphi$-module $R^i\Gamma_{\rig}(M)$ is pure of weight $i$.
				\end{enumerate}
			\end{prop}
			\begin{proof}
				The fact that rigid cohomology can be computed in terms of Monsky-Washnitzer cohomology is well-known, see e.g. \cite[Theorem 7.7.2]{fvdp}, \cite[Proposition 1.10]{berth-fin}. The $\varphi$-structure is induced by the absolute Frobenius on the special fiber, see e.g.~\cite[\S\,2.4]{van_derPut_MW}.
				
				The second statement is a consequence of the Weil conjectures for crystalline cohomology, proved by Katz-Messing \cite{km}, following Deligne \cite{weil1}.
			\end{proof}

			\begin{rmk}
				If $K$ is a local field, as remarked in \cite[Corollary 3.35]{pWM}, by Galois descent, we may extend  $\widehat{R\Gamma}{}_{\rig}^\varpi$ to a realization functor (in $\Prloo$) for arbitrary rigid analytic motives:
				$$
				\widehat{R\Gamma}{}^\varpi_{\rig}\colon \RigDA(K)\to \mcD_{(\varphi,N,\Gamma)}(W(\bar{k})[1/p])
				$$
				where the target is the derived \icat of $(\varphi, N,\Gamma)$-modules in the sense of Fontaine (see e.g. \cite[\S 4.2]{FontaineRepSem}) with $\Gamma=\Gal(\bar{K}/K)$.
			\end{rmk}

\subsection{Hyodo-Kato isomorphism}
\label{sec:hyodo-kato-iso}

As anticipated in the introduction, the comparison with de\,Rham cohomology (the ``Hyodo--Kato isomorphism'') will follow formally from our definition and the weight structure on $\RigDA_{\gr}(K)$.
However, to run this argument we need to land in the right target category.
  It turns out that the category of $\varphi$-modules is suited, so our first goal is to enrich the de\,Rham cohomology with a $\varphi$-module structure.

  For this we will need an alternative interpretation of 
the category of $(\varphi,N)$-modules (\Cref{cons:DAphiN}) $ \DA_{(\varphi,N)}(k)\simeq\Mod_{\rF\chi\one}(\hF{\DA(k)}{\varphi})$ ``on the generic fiber'', as follows.

                \begin{prop}\label{prop:newnew}
                    There is an equivalence of monoidal \icats $$ \Mod_{\rF\chi\one}(\hF{\DA(k)}{\varphi})\simeq\RigDA_{\gr}(k(\!(T)\!))^{\varphi}.$$
                \end{prop}

                \begin{proof}
                    The generic fiber functor $\xi\colon\DA(k)\to\RigDA(k(\!(T)\!))$ is compatible with Frobenius and therefore induces a monoidal functor $\hF{\xi}{\varphi}\colon\hF{\DA(k)}{\varphi}\to\hF{\RigDA(k(\!(T)\!))}{\varphi}$. %
				Since $\chi\colon\RigDA(k(\!(T)\!))\to\DA(k)$ is also compatible with Frobenius, it  induces a functor$$\hF{\chi}{\varphi}\colon\hF{\RigDA(k(\!(T)\!))}{\varphi}\to\hF{\DA(k)}{\varphi}$$ which is easily seen to be right adjoint to~$\hF{\xi}{\varphi}$.
				Thus a factorization of~$\hF{\xi}{\varphi}$ through the monoidal functor
				\[
				\hF{\tilde{\xi}}{\varphi}\colon \Mod_{\hF{\chi}{\varphi}\one}(\hF{\DA(k)}{\varphi})\to\hF{\RigDA(k(\!(T)\!))}{\varphi}
				\]
    which induces an equivalence
		\[ \Mod_{\hF{\chi}{\varphi}\one}(\hF{\DA(k)}{\varphi})\xto{\sim}\hF{\RigDA_{\gr}(k(\!(T)\!))}{\varphi}.
				\]
		Indeed, it is enough  to prove that the right adjoint $\hF{\chi}{\varphi}$ satisfies the hypotheses of the monoidal Barr--Beck--Lurie theorem~\cite[Proposition~5.29]{MR3570153}. It suffices to do this after the application of the conservative, colimit-preserving functor $\hF{\DA(k)}{\varphi}\to \DA(k)$. At this point, the result follows from the monadic properties of the functor $\chi\colon \RigDA_{\gr}(k(\!(T)\!))\to \DA(k)$ \cite[Theorem 3.3.3 (1)]{agv}, recalled in \Cref{prop:facts_of_life}.

                We now claim that the object $\chi^\varphi\one\in\hF{\DA(k)}{\varphi}$ is given by $\rF\chi\one$.
                Once we know that, the proof of the proposition will be complete as $\rF\chi\one$ is the free commutative algebra on $\rF(\one(-1)[-1])$.
                Recall (\Cref{rmk:relative-Frob-modp}) that $\rF$ is the relative Frobenious $\rFactual\colon\id\to\varphi^*$.
                For the claim, by adjunction, it is therefore equivalent to show the following diagram commutes.
               \[
               \begin{tikzcd}[ampersand replacement=\&]
                \xi\chi\one
                \ar[d, "\rFactual\one" swap]
                \ar[r, "\xi \rFactual\chi\one"]
                \ar[rd, "\rFactual\xi\chi\one"'{pos=.81} swap,outer sep=-3pt]
                \&
                \xi\varphi^*\chi\one
                \ar[d]
                \\
                \xi\chi\varphi^*\one
                \ar[d]
                \&
                \varphi^*\xi\chi\one
                \ar[dl]
                \\
                \varphi^*\one
               \end{tikzcd}
                \]
                Here, the unlabeled arrows at the bottom are induced by the counit of the adjunction $\xi\dashv\chi$ while the right vertical arrow is induced by the canonical identification $\xi\varphi^*\simeq\varphi^*\xi$.
                The top triangle commutes as $\xi$ sends relative Frobenii to relative Frobenii.
                The bottom triangle also commutes since both paths can be identified with the composite $\xi\chi\one\to\one\xto{\rFactual}\varphi^*\one$.
                \end{proof}

\begin{rmk}
   As remarked by a referee, the statement above does \emph{not} hold if we considered $K_0$ (and its canonical Frobenius lift $\varphi$) in place of $k(\!(T)\!)$ as indeed in this case the object $\chi^\varphi\one$ is given by the \emph{identity} map on $\mathsf{M}^{\mathrm{coh}}(\G_m)$ (as the uniformizer $p$ is fixed under $\varphi$).
\end{rmk}
Nonetheless, we can ``tilt'' the previous description in characteristic zero in the following way.

\begin{dfn}
    Let $\Ftilt$ be the automorphism of $\RigDA_{\gr}(K_0)$ induced by $\varphi^*$ under the equivalences $\RigDA_{\gr}(K_0)\simeq \DA_N(k) \simeq \RigDA_{\gr}(k(\!(T)\!))$ provided by \Cref{cor:RigDAisDAN} with the `canonical' choices $\varpi=p$ resp.\ $\varpi=T$.
\end{dfn}

We can then restate \Cref{prop:newnew} in mixed characteristic, as follows. 
\begin{cor}
\label{prop:new!}
        The functor $\DA(k)\to \RigDA_{\gr}(K_0)$ is equivariant with respect to $\varphi^*$ and $\Ftilt$, and  it induces an equivalence
        \[%
        \Mod_{\rF\chi\one}(\hF{\DA(k)}{\varphi})\simeq\RigDA_{\gr}(K_0)^{\Ftilt}.%
       \]
       In particular, there exists a functor 
       \begin{equation}\label{frobinmixchar}\RigDA_{\gr}(K_0) \to \RigDA_{\gr}(K_0)^{\Ftilt}\end{equation}
       which, under the equivalence above, coincides with \eqref{eq:enrich}. \qed
    \end{cor}

\begin{rmk}
There is a ``geometric'' interpretation of $\Ftilt$ given by the identification $\RigDA(\Spa K_0)\simeq\RigDA(\mathrm{Spd}\,K_0)$ (see \cite[Theorem 5.13]{LBV}). More precisely,   let $K_0^{\cyc}$ be the perfectoid pro-\'etale cover of $K_0$ given by adding all the $p$-th roots of unity to $K_0$. The  Galois group $\Gal(K_0^{\cyc}/K_0)$ is $\Z_p^*$ and the tilt $K_0^{\cyc\flat}$ is isomorphic to $k(\!(T)\!)^{\Perf}$ (see e.g. \cite[Example 6.2.4]{berkeley}. 
Consider the following canonical equivalences 
$$
\RigDA(K_0)\simeq\RigDA(K_0^{\cyc})^{\Z_p^*}\simeq \RigDA(k(\!(T)\!)^{\Perf})^{\Z_p^*}\simeq  \RigDA(k(\!(T)\!))^{\Z_p^*}
$$
obtained by pro-\'etale descent (see \cite[Theorem 2.15]{LBV}) and the tilting equivalence of \cite{vezz-fw}.  Under the equivalence above, $\Z_p^*$ is acting on $\RigDA(k(\!(T)\!))$ via its action on $k(\!(T)\!)^{\Perf}$ given by $\gamma\cdot T=(T+1)^\gamma-1$. %
We may identify $\Ftilt $ with (the restriction of) the automorphism induced on $\RigDA(K_0)$ by the Frobenius automorphism $\varphi^*$ acting on $\RigDA(k(\!(T)\!)^{\Perf})$ (it commutes with the $\Z_p^*$-action). 
\end{rmk}

   We can now enrich the functor $R\Gamma^\dagger_{\dR}$ with a $\varphi$-module structure, using the functor \eqref{frobinmixchar}.

\begin{prop}\label{prop:HKprequel}
    The functor $R\Gamma^\dagger_{\dR,K_0}\colon\RigDA_{\gr}(K_0)\to\mcD(K_0)$ is equivariant with respect to the functor $\Ftilt$ on $\RigDA_{\gr}(K_0)$ and the functor $\varphi^*$ on $\mcD(K_0)$. In particular, the functor \eqref{frobinmixchar}  induces an enrichment
    $$
    R\Gamma^\dagger_{\dR, K_0}\colon \RigDA_{\gr}(K_0)\to\RigDA_{\gr}(K_0)^{\Ftilt}\to\mcD_{\varphi}(K_0).
    $$
\end{prop}

\begin{proof}
    We first point out that the rigid realization is (canonically) $\varphi$-equivariant, i.e.\ there is an invertible natural transformation $R\Gamma_{\rig}\circ\varphi^*\simeq \varphi^*\circ R\Gamma_{\rig}$ of monoidal functors. (This can be deduced, e.g.\ from the canonical isomomorphism of DG-algebras $\varphi^*\Omega_{X^\dagger/K_0}\simeq \Omega_{\varphi^*X^\dagger/K_0}$ for any smooth dagger affinoid variety $X^\dagger/K_0$, see~\cite[Theorems~1.16, 1.21, 2.12 and 2.16]{ayoub-weil}). As such, we may and do identify the $\DA(k)^\otimes$-algebra structures on $\mcD(K_0)$ given by the monoidal functors $R\Gamma_{\rig}\circ\varphi^*$ and $\varphi^*\circ R\Gamma_{\rig}$, and denote them by $\varphi\mcD(K_0)$. 

    Moreover, observe that the functor $R\Gamma_{\rig} \colon \DA(k)\to \mcD(K_0)$ sends compact objects to compact generators, hence its right adjoint $G$ is conservative and preserves all colimits. Thus it is monadic, inducing an equivalence $\mcD(K_0) \simeq \Mod_{G(K_0)}(\DA(k))$ (see the proof of \Cref{sta:classification-monodromy} and \Cref{rmk:hypholds}). In particular,   we have the following equivalences:
    \begin{align*} \Map_{\Prloost_{\DA(k)^\otimes/-}} (\RigDA_{\gr}(K_0), \varphi\mcD(K_0)) &\simeq \Map_{\CAlg(\DA(k))}(\free(\one(-1)[-1]) ,  \varphi_*G(K_0)) \\ 
&\simeq \Map_{\mcD(K_0)}(\varphi^*R\Gamma_{\rig}(\one(-1)[-1]), K_0 ) \\&\simeq \Map_{\mcD(K_0)}(K_0[-1],K_0). \end{align*}
Since the $\pi_0$ of this latter space is trivial,  we can choose an invertible natural transformation   between $R\Gamma^\dagger_{\dR, K}\circ \varphi^\sharp$ and $\varphi^*\circ R\Gamma^\dagger_{\dR, K}$, as claimed.
(We use \Cref{prop:new!} to interpret the former as a point in the space.)
\end{proof}

We can now produce an equivalence (``Hyodo-Kato isomorphism'') between  de Rham cohomology (of the generic fiber) and  rigid cohomology (of the nearby cycle) as $\varphi$-modules, as follows.

			\begin{thm}\label{thm:HKiso} Assume that $k\subseteq\barFp$. 
				There is an  identification of monoidal functors from $\RigDA_{\gr}(K_0)$ to $\mcD_\varphi(K_0)$, $$\pi\circ \widehat{R\Gamma}_{\rig}\simeq R\Gamma^\dagger_{\dR,K_0}.$$
			\end{thm}
			\begin{proof}
                        It is enough to verify that the assumptions of \Cref{cor:RGammas} are satisfied, with $R\Gamma$ being the functor $R\Gamma_{\dR,K_0}^{\dagger}\colon\RigDA_{\gr}(K_0)\to\mcD_{\varphi}(K_0)$ of \Cref{prop:HKprequel}. Assume first that $k$ is finite. Then, for any $M$ in $\heart{k}$ we have $R^i\Gamma_{\dR}^\dagger(\xi M)\cong \Hm^i_{\rig}(M)$ which is pure of weight $i$ by \Cref{prop:ispure}. Note that $\varphi$-modules form an abelian category of cohomological dimension at most $1$, and that there are no maps nor extensions between pure $\varphi$-modules of different weights (see \Cref{rmk:phi-modules-coh-dim}). As pointed out in  \Cref{exm:Ai} we can then apply \Cref{cor:RGammas} to conclude.

                        If $k=\overline{\F}_p$, we let $\catD$ be the colimit $\varinjlim(\Derp(K'_0))$ in $\Prlost$, where $K'_0$ runs through the set of finite unramified extensions of $\Q_p$. Note that since $\RigDA_{}(\breve{\Q}_p) = \varinjlim\RigDA_{}(K'_0)$ (see \cite[Theorem~2.8.15]{agv}), the realization $R\Gamma^\dagger_{\dR, \breve{\Q}_p}$ factors through $\catD$, and that any two $X,Y\in \mathcal{H}_k$ are actually defined over a common finite subfield $\F_q$. We are then reduced to the previous case, and we can conclude.
			\end{proof}

			\begin{cor}
				\label{cor:HKiso}
					{Assume that $k\subseteq\barFp$.} There is an  identification of monoidal functors from $\RigDA_{\gr}(K)$ to $\mcD(K)$:
					$$  \widehat{R\Gamma}_{\rig}\widehat{\otimes}_{K_0}K\simeq R\Gamma^\dagger_{\dR,K}.$$
				\end{cor}
				\begin{proof}
					Since the base change induces an equivalence $\RigDA_{\gr}(K_0)\simeq\RigDA_{\gr}(K)$ and the overconvergent de\,Rham realization is compatible with base change, we may and do suppose that $K=K_0$, and the result follows from \Cref{thm:HKiso}. %
				\end{proof}
\begin{rmk}
   \label{rmk:josephHK} If we drop the hypothesis that $k\subseteq \barFp$, we can still find an identification   $\widehat{R\Gamma}_{\rig}\widehat{\otimes}_{K_0}K\simeq R\Gamma^\dagger_{\dR,K}$ as monoidal functors to $\mcD(K)$ using the equivalence $$\Map_{\Prloost_{\DA(k)^\otimes/-}} (\RigDA_{\gr}(K), \mcD(K))\simeq \Map_{\mcD(K)}(K[-1],K)$$
   as we have done in the proof of \Cref{prop:HKprequel}. (This can also be obtained using \cite[Theorem 3.21]{ayoub-weil}.) 
   Note that, contrarily to \Cref{cor:HKiso}, in this generality the existence of such an equivalence relies on a certain space of functors being connected, rather than on explicit identifications between the two functors built using the weight structure formalism.
   \end{rmk} 

   \begin{rmk}\label{rmk:mononWCHK}
					Note that in the proof of \Cref{thm:HKiso}, we even proved that the functor $\pi\circ \widehat{R\Gamma}_{\rig}$ factors through the weight complex functor and that $\varphi$-modules fit into the setting of \Cref{exm:Ai}. In particular, the cohomology groups $\widehat{\Hm}{}^n_{\rig}(X) = \Hm_{-n} (\widehat{R\Gamma}_{\rig}(\mathsf{M}(X)^\vee))$ can be computed via a canonical, convergent, weight spectral sequence, degenerating at page $E_2$, giving rise to the weight filtration. The monodromy operator acts on the weight complex, and hence on this spectral sequence, so that the monodromy operator restricts to a map between the graded pieces $\gr^i \widehat{\Hm}{}^n(X)\to \gr^{i-2}\widehat{\Hm}{}^n(X)(-1)$. See (the cohomological version of) \Cref{eg:mononWC}.%
				\end{rmk}

\subsection{Comparison with the ``classical'' Hyodo--Kato}
\label{sec:comp-with-classical}

				We now show that our $\widehat{R\Gamma}_{\rig}$ computes the ``classical'' Hyodo--Kato cohomology.
				\begin{dfn}If $\car K=0$  we denote by $R\Gamma_{\HK}^p\colon \RigDA_{\gr}(K) \to \mcD_{(\varphi,N)}(K_0)$ the motivic Hyodo--Kato realization functor $R\Gamma_{\HK}$ of \cite[Proposition 3.30]{pWM}. 
					If $\car K=p$ we denote by $R\Gamma_{\HK}^\varpi\colon \RigDA_{\gr}(K) \to \DpN$ the motivic Hyodo--Kato realization functor of \cite[Corollary 3.32(2)]{pWM}. 
				\end{dfn}
				The realizations above are defined using the classical definition of Hyodo--Kato, i.e., based on a log-de\,Rham complex of some appropriate lifts of the log-smooth special fiber of a variety with semi-stable reduction.
				
				\begin{cor}\label{cor:HKisHK} Assume $k\subseteq\barFp$ and fix a pseudo-uniformizer $\varpi$ of $K$. If $\car K=0$, we assume $\varpi=p$.  There is an equivalence ${\widehat{R\Gamma}}{}^\varpi_{\rig}\simeq 
					R\Gamma^\varpi_{\mathrm{HK}}$ as monoidal functors $\RigDA_{\gr}(K)\to \DpN$.
				\end{cor}
				\begin{proof}%
					We want to apply \Cref{prop:unicity}. By the compatibility of both realization functors with  unramified base extensions and the tilting equivalence, we may and do assume that $K$ is in mixed characteristic and  $K=K_0$ with $\varpi=p$.  Under these assumptions, we can use  the Hyodo--Kato isomorphism (see e.g. \cite[Formula 5.16]{CN19}, \cite[Lemma 2.21]{DN}, \cite[Definition 4.3]{EY19}) which yields an equivalence of monoidal functors. We then know that the underlying $\varphi$-module of $R\Gamma_{\HK}$ coincides with the $\varphi$-module underlying the overconvergent de\,Rham cohomology,
					i.e., there is an equivalence of monoidal functors $\pi\circ R\Gamma_{\mathrm{HK}}\simeq  
					R\Gamma^\dagger_{\dR}$. On the other hand, we proved in \Cref{thm:HKiso} that $ R\Gamma^\dagger_{\dR}\simeq (\pi\circ \widehat{R\Gamma}_{\rig})=R\Gamma_{\rig}\circ\Psi$. This gives an equivalence $\pi\circ R\Gamma_{\HK}\simeq R\Gamma_{\rig}\circ \Psi$.
					
					By construction, one also has $R\Gamma_{\HK}\circ\xi\simeq (N=0)\circ R\Gamma_{\rig}$ (see \cite[Definition 3.3]{pWM} using \cite[Definition 3.17 and Remark 3.19]{EY19}). 
					
					We are left to show that the monodromy operators on $R\Gamma_{\HK}(\mcK)$ and $\widehat{R\Gamma}_{\rig}(\mcK)$ agree. In light of \Cref{thm:KummerisTate} this amounts to computing the monodromy on the $\varphi$-module $\Hm^1_{\HK}(\G_m^{\an}/p^\Z)\cong  K_0\oplus K_0(-1)$. This is well-known to be represented by the matrix $\left(\begin{smallmatrix}
						0&1\\0&0
					\end{smallmatrix}\right)$, see e.g. \cite[Section~9]{lestum-HKc}.  This agrees with the realization of the Kummer motive via $\widehat{R\Gamma}_{\rig}$ using \Cref{rmk:monodromy_Kummer_motive}.%
				\end{proof}
				
				\begin{rmk}
					In the category of $(\varphi,N)$-modules, there are only two extensions of $\one$ and $\one(-1)$ up to rescaling of the monodromy operator: the trivial one (with trivial monodromy) and the  one given by the monodromy matrix $\left(\begin{smallmatrix}
						0&1\\0&0
					\end{smallmatrix} \right)$. By \Cref{prop:unicity}, in order to prove the corollary up to a rescaling factor, it then suffices to show that the underlying $\varphi$-modules of the two realization functors agree, and that the monodromy of the Hyodo--Kato realization of the Kummer motive is not zero.  In light of \Cref{sta:classification-monodromy} this second condition amounts to showing that the monodromy of the Hyodo--Kato realization  is not \emph{always} zero, which is obvious.
				\end{rmk}

				We point out that our alternative definition of $R\Gamma_{\HK}$ gives a direct proof of the compatibility of the Hyodo--Kato cohomology with the tilting equivalence, which is one of the crucial ingredients of the proof of the $p$-adic weight monodromy conjecture for complete intersections \cite[Corollary 3.13]{pWM}.
				\begin{cor}\label{cor:shortcut}
					The composition
					$$
					\RigDA(\C^\flat_p)\simeq \RigDA(\C_p)\xto{{R\Gamma}_{\HK}}\mcD_{(\varphi,N)}(\breve{\Q}_p)
					$$
					coincides with ${R\Gamma}_{\HK}^{p^\flat}$.
				\end{cor}
				\begin{proof}
					This is clear as they both agree with $\widehat{R\Gamma}_{\rig}$ on $\DA_{(\varphi,N)}(k)$ by \Cref{cor:HKisHK}.
				\end{proof}
				
					\subsection{Hodge and $\ell$-adic realizations}
				\label{sec:H-and-l}
				Our framework is well-suited to define and recover more classical realizations. We briefly sketch their construction here.
				
				\begin{rmk}(The $\ell$-adic case) Let $K$ be a local field, $\ell$ be a prime number invertible in $k$ and let $$R\Gamma_{\ell}\colon\DA(k)\to \mcD_{\proet}(k,\Q_\ell)$$%
					be the homological \'etale $\ell$-adic realization functor (see \cite[\S 9]{AyoubEt}). Using \Cref{cor:RGammahat} we can define a canonical extension
					$$
					\widehat{R\Gamma}_\ell\colon\RigDA_{\gr}(K)\to\LFnil{\mcD_{\proet}(k,\Q_\ell)}{T}
					$$
					A compact object of the target is informally given by an equivariant map $V\to V(-1)$ of continuous $\Gal(k)$-representations on a finite dimensional (complex of) $\Q_\ell$-vector spaces, where the twist $(1)$ is given by the $\ell$-adic cyclotomic character.  
					Using the same techniques as in \Cref{sec:dR-realization} one can prove the following:
					\begin{enumerate}
						\item An equivalence of $(\Gal(k),N)$-representations $R\Gamma_{\ell}(X)\simeq R\Gamma_{\ell}(\Psi X)$ where on the left we consider the $\ell$-adic realization of (unipotent) rigid analytic varieties (see \cite[\S 3.1]{bamb-vezz}). Note that the action of $N$ corresponds to the logarithm of the (unipotent) action of a generator of the pro-cyclic group $I/P$, where $I$ is the inertia group and $P$ is its pro-$p$-Sylow subgroup. In particular, the motivic monodromy coincides with the classical monodromy operator of Galois representations. This recovers the compatibility result of Ayoub \cite[Th\'eor\`eme 11.17]{AyoubEt} for $S=\Spec \mcO_K$.
						\item 
						A factorization of ${R\Gamma}_\ell\colon \DA(k)\to \mcD_{\proet}(k,\Q_\ell)$ through $\mcK(\Ho(\heart{k}))$ giving rise to a factorization of $\widehat{R\Gamma}_\ell$ through $\LF{\mcK(\Ho(\heart{k}))}{T}$. This induces a canonical weight spectral sequence converging to $\Hm_n(X,\Q_\ell)$,  for any $X$ in $\RigDA_{\gr}(K)$. The induced filtration is the weight filtration, and the monodromy operator restricts to a map $\gr_i\to \gr_{i+2}(-1)$.
					\end{enumerate}  
					 Note that we can simply consider the (\'etale) realization induced by base change along~$\bar{k}/k$ and~$R\Gamma_\ell$:
					$$
					R\Gamma_{\ell}\colon\DA(k)\to \mcD(\Q_{\ell})
					$$
					taking values in  $\Q_\ell$-modules, and extend it directly to 
					$$
					\widehat{R\Gamma}_\ell\colon\RigDA_{\gr}(K)\to\mcD_{(\varphi,N)}(\Q_{\ell})
					$$
					where on the right we consider the derived \icat of $(\varphi,N)$-modules over~$\Q_\ell$. In this case we recover the Weil--Deligne representation attached to the \'etale cohomology of rigid analytic varieties (of  polystable reduction). 
				\end{rmk}
				
				\begin{rmk}(The Hodge case) 
    Suppose that $K=\C(\!(T)\!)$. We may consider the Hodge realization %
     $$R\Gamma_{\Hodge}\colon \DA(\C)\to\Ind\mcD^b(\MHS^p_\Q)$$ 
     taking values in  the Ind-derived \icat of rational {polarizable} mixed Hodge structures (see e.g.~\cite[Theorem~V.2.3.10]{levine_motives}, \cite[Theorem~1.2]{drew:motivic-hodge}), the latter being a semi-simple abelian category (see~\cite[Corollary~I.2.12]{hodge-book}). 	
     Following the convention in Hodge theory, we say that  a polarizable mixed Hodge structure $(F^\bullet V_\C, V_\Q)$  is pure of weight $n$ if $V_{\C}^{p,q}=0$ for $p+q\neq n$, where $V_{\C}^{p,q} = F^p V_\C \cap \overline{F^q V_\C}$.
                                                                                
				 By construction, the Tate object $\one(1)\in \DA(\C)$ is sent by $R\Gamma_{\Hodge}$ to the object $\Q_{\Hodge}(1)$, given by a pair $(F^\bullet V_\C, V_\Q)$, where $V_\Q = {(2\pi i)} \Q$, $V_\C = \C\otimes_\Q V_\Q \cong \C$ and the filtration is $0=F^0 \subset F^{-1} = V_\C$. This corresponds to the classical Tate (pure) Hodge structure of weight~$-2$, rank~$1$ and grading~$(-1,-1)$. 	 Note that the realization of any  motive in $\heart{\C}$ gives rise to a pure Hodge structure, and that the family, for varying~$i$, of pure polarizable rational Hodge structures of weight~$i$ are orthogonal in the sense of \Cref{exm:Ai}.

     If we let 
    $$
    \widehat{R\Gamma}_{\Hodge}\colon \RigDA_{\gr}(\C(\!(T)\!))\to \LF{\Ind\mcD^b(\MHS_\Q^p)}{T}
    $$
   be the induced extension, where now $T$ is the twist by~$\Q_{\Hodge}(-1)$, we then deduce that the image lies in the \subicat generated by pure Hodge structures under colimits and shifts, and that moreover:
     \begin{enumerate}
        \item One has the formula $\widehat{R\Gamma}_{\Hodge}(X)\simeq {R\Gamma}_{\Hodge}(\Psi X)$.
        \item ${R\Gamma}_{\Hodge}$ factors over $\mcK(\Ho(\heart{\C}))$. Therefore $\widehat{R\Gamma}_{\Hodge}$ factors over $\LF{\mcK(\Ho(\heart{\C}))}{T}$. This gives rise to a  canonical weight spectral sequence converging to $\Hm_n\widehat{R\Gamma}{}_{\Hodge}(X)$,  for any $X$ in $\RigDA_{\gr}(\C(\!(T)\!))$. The induced filtration is the weight filtration, and the monodromy operator restricts to a map $\gr_i\to \gr_{i+2}(-1)$. 
    \end{enumerate}
					This allows us to call $\widehat{R\Gamma}{}_{\Hodge}$  the limit-Hodge realization following Steenbrink. Note that the weight-monodromy conjecture  (proved by Saito) implies that $N^k\colon \gr_{d-k}\to\gr_{d+k}(-k)$ is an equivalence, whenever $X$ is the motive of the analytification of a smooth and projective variety over~$\C(T)$.
				\end{rmk}
				
				\begin{rmk}
     The monodromy operator  $N$ acting on the singular cohomology $V = \Hm^{*}(X_t, \Q)$ of a smooth generic fiber $X_t$ of a  semi-stable degeneration $X\to \Delta$ of complex K\"ahler manifolds is classically used to \emph{define} a mixed Hodge structure (the limit Hodge structure) on the graded vector space $V$ \cite[Section 2]{morrison}, which is then only \textsl{a posteriori} compared with the mixed Hodge structure on the cohomology of the nearby cycle. From our perspective, it is more natural to look at the monodromy  as an extra operator that acts directly on the motivic nearby cycle, and thus on its realization. 
				\end{rmk}
				
			\subsection{The Clemens--Schmid diagram}
				\label{sec:CS}
				We now define a motivic Clemens-Schmid diagram that induces by any (appropriate) realization the two Clemens-Schmid chain complexes. Specializing to the case of the Hyodo-Kato realization, this provides a proof of a conjecture of Flach--Morin \cite[Conjecture 7.15]{flach-morin}. 
				
				\begin{prop}
 Let $K$ be as in \Cref{not:K}.
    Denote by $j\colon \Spec K\to \Spec\mcO_K$ the natural open immersion and by $\iota\colon\Spec k\to\Spec\mcO_K$ its closed complementary. 
					\begin{enumerate}
						\item\label{eq:chi_is_fib_CS} Let $M$ be an object of $\DA(K)$ such that $\An^*M$ lies in $\RigDA_{\gr}(K)$. There is a functorial fiber sequence in  $\DA(k)$:
						$$
						\chi^{\mathrm{alg}}M \to \Psi  \An^*M \xto{N} \Psi  \An^*M (-1).
						$$
						\item\label{eq:CS_triangles} Let $M$ be an object of $\DA(\mcO_K)$ such that $\An^*j^*M$ lies in $\RigDA_{\gr}(K)$. There are two functorial fiber sequences of in  $\DA(k)$:
						\begin{align*}
							& \chi^{\mathrm{alg}}j^* M\to \Psi  \An^* j^*M  \xto{N} \Psi  \An^* j^*M(-1) 
							\\
							\iota^! M\to \iota^*M \to  & \chi^{\mathrm{alg}}j^* M.
						\end{align*}
					\end{enumerate}
				\end{prop}
					\begin{proof}
						Part~\eqref{eq:chi_is_fib_CS} follows directly from \Cref{rmk:chi_isFibN_xi_is_pirn}, since the functor $\chi$ is identified with the functor that computes the fiber of the monodromy.
The first fiber sequence in~\eqref{eq:CS_triangles} is a special case of that.
And for the second, it is enough to apply $\iota^*$ to the localization fiber sequence:
\[ \iota^*( \iota_* \iota^! M \to M \to j_*j^*M) = \iota^!M \to \iota^* M\to \chi^{\mathrm{alg}}j^*M.\qedhere
\]
					\end{proof}

				\begin{rmk}\label{rmk:calotte}Following Saito's depiction \cite[Remarque 5.2.2]{saito-mdhp}, we may concatenate the two fiber sequences in \eqref{eq:CS_triangles} using their common vertex  $\chi^{\mathrm{alg}}j^* M$ as the ``lower half'' of a octahedron (on the right) which induces automatically the ``upper half'' of the octahedron (on the left)  here below, where solid arrows are part of fiber sequences and $\Phi$ denotes the cofiber of the induced map $\iota^*\to \Psi \An^* j^*$ (the motivic \emph{vanishing cycles}).
$$
\adjustbox{max width=\textwidth}{
\begin{tikzcd}
                \iota^*\ar[dd,"+1"] &&\iota^!\ar[ll,dashed]\ar[dl]
                &
			\iota^*
			\arrow[dd, dashed, "+1"]
			\ar[dr]
			&&
                \iota^!
                \ar[ll]
\\
&\Phi[-1]\ar[ul]\ar[dr]
&&
& \chi^{\mathrm{alg}} j^*\ar[ur,"+1"]\ar[dl,"+1"]
\\
\Psi \An^* j^*[-1]\ar[rr,"N",dashed]\ar[ur]&&\Psi \An^*j^*(-1)[-1]\ar[uu,"+1"]
&
\Psi \An^* j^*[-1]\ar[rr,"N"]&&\Psi \An^*j^*(-1)[-1]\ar[ul]\ar[uu,"+1",dashed]
		\end{tikzcd}}
  $$
				\end{rmk}
				
				\begin{rmk}
					Vertical maps have also a geometric interpretation: the canonical open inclusion $(\mcX^\wedge)_\eta\subseteq \mcX^{\An}_K$ for any smooth scheme $\mcX$ over $\mcO_K$ induces a natural transformation $\xi\iota^*\to \An^*j^*$ (see  \cite[Remark 4.5]{vezz-bang}) and hence a natural transformation $\iota^*\simeq \Psi\xi\iota^*\to \Psi\An j^*$.
				\end{rmk}
				
				\begin{rmk}\label{rmk:calotteself-dual}
				If $K$ is discretely valued,	note that the diagrams above are ``self-dual'' in the following sense. For any $a\colon S\to \mcO_K$ of finite type and any  $M$ in $\DA(S)$ we let  $\mathbb{D}_S(M)$ be the motive $\uhom(M,a^!\one)$. We remind the reader that, by absolute purity, the  functor  $\mathbb{D}_k$ agrees with $\uhom(-,\one)(-1)[-2]$ while $\mathbb{D}_K$ agrees with $\uhom(-,\one)$ (since $j^!=j^*)$. %
     
					Assume now that $M$ is compact. In this case, we have  canonical identifications $\mathbb{D}_k(\iota^!M)=\iota^*\mathbb{D}(M)$ and $\mathbb{D}_k(\iota^*M)=\iota^!\mathbb{D}(M)$.\footnote{More generally, $\mathbb{D}$ ``exchanges'' $*$'s with $!$'s and vice-versa, see \cite[Th\'eor\`eme 2.3.75]{ayoub-th1}.} Also, note that $j^*M$ is compact (hence dualizable). Now, further assume that $\An^*j^*M$ lies in $\RigDA_{\gr}(K)$. One has then $\An^*j^*\mathbb{D}(M)\cong \uhom(\An^*j^*M,\one)$. This proves that $\An^*j^*\mathbb{D}(M)$ also lies in $\RigDA_{\gr}(K)$.  In this situation, we have canonical identifications		
     $$\resizebox{\textwidth}{!}{$
					\uhom(-,\one)\left(
					\Psi\An^*j^*M\xto{N}\Psi\An^*j^*M(-1)
					\right)\cong
					\left(\Psi\An^*j^*\mathbb{D}(M)\xto{N}\Psi\An^*j^*\mathbb{D}(M)(-1)\right)(1)
     $}
					$$
					and
					$$
					\uhom(-,\one)\left(
					\iota^!M\xto{}\iota^*M
					\right)\cong
					\left(
					\iota^!\mathbb{D}(M)\xto{}\iota^*\mathbb{D}(M)\right)(1)[2].
					$$
					In particular, the square diagrams of \Cref{rmk:calotte}  applied to $\mathbb{D}(M)$ agree with the duals of the  diagrams applied to $M$, up to $(-1)[-2]$. 
				\end{rmk}

				\begin{rmk}\label{rmk:CShom} Let $R\Gamma\colon\DA(k)\to \catD$ be a functor in $\Prlmst$ and assume $\catD$ is endowed with a t-structure. Let $\widehat{R\Gamma}\colon\RigDA_{\gr}(K)\to \LFnil{\catD}{T}$ be the induced extension. We denote by $\Hm_*$ resp. $\widehat{\Hm}_*$ the homology objects $\Hm_*R\Gamma$ resp. $\Hm_*\pi\widehat{R\Gamma}\cong \Hm_*R\Gamma\Psi$ (where $\pi$ is the functor that ``forgets monodromy'', see \Cref{cons:lax-equalizer}). Let also $M$ be in $\DA(\mcO_K)$ such that $\An^*j^*M$ is in $\RigDA_{\gr}(K)$. 
					By intertwining the long exact sequences in homology of the diagrams in \Cref{rmk:calotte}, we deduce the existence of two canonical chain complexes (there is a drop by two in the indices, giving rise to \emph{two} parallel complexes):
					$$
					\cdots\! \to \Hm_*(\iota^*M)\to \widehat{\Hm}_*(\An^*j^*M)\xto{N}\widehat{\Hm}_*(\An^*j^*M)(-1)\to \Hm_{*-2}(\iota^!M)\to \Hm_{*-2}(\iota^*M)\to\!\cdots 
					$$
				\end{rmk}
				
				We now want to give a cohomological version of Clemens-Schmid complexes, under the semi-stable reduction hypothesis, which is the case typically considered in the literature. To this aim, we recall that for any map of schemes $p\colon X\to S$ which is locally of finite type there are induced adjunctions (see e.g. \cite{ayoub-th1, ayoub-ICM})
				$$
				p^*\colon \DA(S)\to \DA(X)\colon p_*\qquad 
				p_!\colon \DA(X)\to \DA(S)\colon p^!.
				$$
				The object $\mathsf{M}_S^{\mathrm{coh}}(X)\colonequals p_*\one=p_*p^*\one$ is the so-called ``co-homological'' motive of $X$ over $S$ whereas the object $\mathsf{M}_S(X)\colonequals p_!p^!\one$ is the so-called ``homological'' motive of $X$ over $S$. The same can be said for a map of rigid analytic spaces, see \cite{agv}. From now on, we abbreviate~$\{i\}\colonequals (i)[2i]$.

				\begin{rmk}\label{rmk:CSpre}
				Let $K$ be discretely valued and let $M=\mathsf{M}^{\mathrm{coh}}_{\mcO_K}(\mcX)$ be the cohomological motive of a flat, projective regular scheme $p\colon \mcX\to \mcO_K$ of relative dimension $d$ with a smooth generic fiber, for which $\An^*j^*M$ lies in $\RigDA_{\gr}(K)$ (e.g.\ if $\mcX$ has polystable reduction). In this case, by absolute purity (see e.g. \cite[Theorem A(II)]{dfjk}), we know that $p^!\one\cong \one\{d\}$ which implies that $$\iota^!M=\iota^!p_*p^*\one\cong p_*\iota^!p^!\one\{-d\}\cong p_!p^!\iota^!\one\{-d\}=p_!p^!\one\{-d-1\}\cong \mathsf{M}_k(\mcX_k)\{-d-1\}.$$
					The outer square of the diagrams in \Cref{rmk:calotte} is then formed by the maps
					$$
					\mathsf{M}_k(\mcX_k)\{-d-1\}\to \mathsf{M}^{\mathrm{coh}}_k(\mcX_k)\to\Psi \mathsf{M}_K^{\mathrm{coh}}(\mcX_K^{\An})\xto{N} \Psi \mathsf{M}_K^{\mathrm{coh}}(\mcX_K^{\An})(-1)\to \mathsf{M}_k(\mcX_k)\{-d-1\}[2].
					$$
     Note that  the motives $\mathsf{M}_K(\mcX_K)$, $\mathsf{M}_K(\mcX_K^{\An})$ and $\mathsf{M}_k(\mcX_k)$ are compact and hence strongly dualizable (see e.g. \cite{riou-dual}, \cite[Proposition 2.31]{ayoub-neww}). 
				\end{rmk}
				\begin{cor}\label{cor:CS_realization}
                               Let $K$ be discretely valued. Let $R\Gamma\colon\DA(k)\to \catD$ be a functor in $\Prlmst$ and assume that $\catD$ is endowed with a t-structure compatible with tensor products and with heart~$\mcA$.
                                Let $\widehat{R\Gamma}\colon\RigDA_{\gr}(K)\to \LFnil{\catD}{T}$ be the induced extension. Let  $\mcX$ be a flat, proper, regular scheme of essential relative dimension $d$ over $\mcO_K$ which is generically smooth and has polystable reduction over $k$ (e.g. it has semi-stable reduction). %
						There are two canonical chain complexes in $\mcA$ (for even and for odd indices):
						$$%
						\cdots \to \Hm^*(\mcX_k)\to \widehat{\Hm}{}^*(\mcX_K^{\An})\xto{N} \widehat{\Hm}{}^*(\mcX_K^{\An})(-1) \to \Hm_{2d-*}(\mcX_k)(-d-1) \to  \Hm^{*+2}(\mcX_k)\to\cdots 
						$$%
					where we let $\Hm_*(M)$ resp. $\widehat{\Hm}_{*}(M)$   be the objects $\Hm_*R\Gamma(M)$ resp. $\Hm_*\pi\widehat{R\Gamma}(M)\cong \Hm_*R\Gamma(\Psi M)$  and we let  $\Hm^{*}(M)$ resp. $\widehat{\Hm}{}^{*}(M)$ be their duals.
				\end{cor}
				
				\begin{proof}As $\mathsf{M}^{\mathrm{coh}}_k(\mcX_k)=(\mathsf{M}(\mcX_k))^\vee$ and $\mathsf{M}^{\mathrm{coh}}_K(\mcX_K^{\An})=(\mathsf{M}(\mcX_K^{\An}))^\vee$ we have $\Hm_n(\mathsf{M}^{\mathrm{coh}}_k(\mcX_k))=\Hm^{-n}(\mcX_k)$, $\widehat{\Hm}_n(\mathsf{M}^{\mathrm{coh}}_K(\mcX_K^{\An}))=\widehat{\Hm}{}^{-n}(\mcX_K^{\An})$ and $\widehat{\Hm}_n(\mathsf{M}^{\mathrm{coh}}_K(\mcX_K^{\An})(-1))=\widehat{\Hm}{}^{-n}(\mcX_K^{\An})(-1)$. We may then  apply \Cref{rmk:CShom} to the motive $M=\mathsf{M}_{\mcO_K}^{\text{coh}}(\mcX)$ using \Cref{rmk:CSpre}.
				\end{proof}

				\begin{rmk}
					In the Hodge, $\ell$-adic or $p$-adic case, it is known that the Clemens-Schmid complexes associated to $\mcX$ as in \Cref{rmk:CSpre} are exact whenever the corresponding weight-monodromy conjecture holds for $\mcX_K$ (see e.g.\ \cite[\S 5]{chiar-tsu-CS}). In particular, they are always exact in the equi-characteristic case,  in case $\mcX_K$ has dimension $\leq2$ or is a smooth complete intersection in a smooth projective toric variety.
				\end{rmk}
				
				\begin{exm}\label{exm:WCofiotas}
					Let us explicitly describe the maps induced on the weight complexes by the Clemens-Schmid octahedron of \Cref{rmk:calotte}. Let $\mcX$ be as in \Cref{cor:CS_realization} and denote by $\{D_i\}_{i\in I}$ the set of (proper, smooth) irreducible components of the special fiber $\mcX_k$. For any $J\subseteq I$ we let $D_J$ be the proper, smooth variety $\bigcap_{j\in J} D_j$. By $h$-descent, the weight complex of $\mathsf{M}(\mcX_k)$ is given by
					$$
					\mathsf{M}(D_I) \to \bigoplus_{|J|=|I|-1} \mathsf{M}(D_{J}) \to \cdots \to \bigoplus_{i\neq j} \mathsf{M}(D_{ij}) \to \bigoplus_i \mathsf{M}(D_i)  
					$$
					where the right-most term is in degree~$0$ and with differentials given by the alternating sums of maps induced by inclusion. We deduce that the weight complex of $\mathsf{M}^{\text{coh}}(\mcX_k)$ is given by its dual, that is,
					$$
					\bigoplus_i \mathsf{M}(D_i)^\vee \to   \bigoplus_{i\neq j} \mathsf{M}(D_{ij})^\vee \to \cdots \to \bigoplus_{|J|=|I|-1} \mathsf{M}(D_{J})^\vee \to \mathsf{M}(D_I)^\vee
					$$
					which by purity can be expressed as
					$$
					\bigoplus_i \mathsf{M}(D_i)\{-d\} \to   \bigoplus_{i\neq j} \mathsf{M}(D_{ij})\{-d+1\} \to \cdots \to \bigoplus_{|J|=|I|-1} \mathsf{M}(D_{J})\{-1\}  \to \mathsf{M}(D_I)
					$$
					where the left-most term is in degree~$0$ and the differentials are alternating sums of Gysin maps.
					
					The map $\mathsf{M}_k(\mcX_k)=\iota^!\mathsf{M}^{\mathrm{coh}}(\mcX)\{d+1\}\to \iota^*\mathsf{M}^{\mathrm{coh}}(\mcX)\{d+1\}= \mathsf{M}_k^{\mathrm{coh}}(\mcX_k)\{d+1\}$ is given on weight complexes by (we omit $\mathsf{M}(-)$ and $0$'s for brevity)
					$$
					\xymatrix{
						D_I \ar[r]  &  \cdots \ar[r] & \bigoplus_{i\neq j} D_{ij} \ar[r] & \bigoplus D_i  \ar[d]\\
						&&& \bigoplus D_i\{1\} \ar[r] &     \bigoplus_{i\neq j} D_{ij}\{2\}\ar[r]&\cdots\ar[r]& D_I\{d+1\}
					}
					$$
					where the vertical map is induced by $\bigoplus D_i\to \mcX\to \bigoplus D_j\{1\}$ (cfr. \cite[Theorem 4.2.1]{DDO}).
     The cone of this map (which is simply the concatenation of the two complexes) computes the weight complex of $\chi^{\alg} j^* \mathsf{M}^{\mathrm{coh}}(\mcX)\{d+1\}=\chi^{\alg} \mathsf{M}(\mcX_K)\{1\}$.
				\end{exm}
				
				\begin{exm}
					We point out that it is also possible to use rigid analytic varieties to give an explicit description of the objects in the lower part of the diagrams in   \Cref{rmk:CSpre} for  $M=\mathsf{M}_{\mcO_K}^{\mathrm{coh}}(\mcX)$ as in \Cref{cor:CS_realization}. We use the same notation as in \Cref{exm:WCofiotas}.
					
					We can use the specialization map $\mathrm{sp}\colon |\mcX_K^{\An}|\to |\mcX_k|$ to define open subsets of the domain given by $(\mathrm{sp}^{-1}(D_J))^\circ=]D_J[$. This yields a presentation $\mathsf{M}(\mcX_K^{\An})=\colim \mathsf{M}(]D_J[)$. To fix our ideas let us assume there are only two irreducible components $D_1$, $D_2$. We have $\mathsf{M}(]D_i[)\cong \mathsf{M}(]D_i\setminus D_{12}[)$ (see~\cite[Theorem~5.1]{ais}) which is the motive of good reduction $\xi \mathsf{M}(D_i\setminus D_{12})$ %
     with weight complex $(\,D_i \to D_{12}\{1\}\,)$ in degrees~$0$ and~$-1$. Note that, by considering an admissible formal blow up centered in $D_{12}$ we can find another  model $\mcX'$ for $\mcX^{\An}_K$ lying above $\mcX$, whose special fiber has as irreducible components $D_1,D_2$ and $\P^1_{D_{12}}$,  for which $\mathsf{M}(]D_{12}[)=\mathsf{M}(]\P^1_{D_{12}}[)\cong \mathsf{M}(]\G_{m,D_{12}}[)$. Its weight complex is $(\,D_{12}\xto{0} D_{12}\{1\}\,)$. We deduce that the weight complex of $\mathsf{M}(\mcX_K^{\An})$ (which is the same as the one of $\Psi \mathsf{M}(\mcX_K^{\An})$) is
					$$
					\Tot\left(
					\vcenter{\xymatrix{D_{12}\ar[r]\ar[d] & D_1\oplus D_2\ar[d]\\
							D_{12}\{1\}\ar[r] & (D_{12}\{1\})^{\oplus 2}
					}}
					\right) \simeq \Big(D_{12}\to D_1\oplus D_2 \to D_{12}\{1\} \Big) 
					$$
                                        in degrees~$[-1,1]$.
					A similar description can be given using h-descent and the formulas for~$\Psi$ given in~\cite[Theorem~6.1]{ais}.

					In the general case (with more than two irreducible components), we expect similar simplifications to hold giving an analogous description of the weight complex $C_\bullet$ of $ \mathsf{M}(\mcX^{\An}_K)$  akin to the one given by Bloch--Gillet--Soul\'e \cite{bgs}.  That is, $C_\bullet$ should coincide with the total complex  of the bi-complex
					$
					(X^{r,s}=X^{(r-s+d)}\{s\}, \del',\del'')_{r,s}
					$
					where $X^{(a)}=\bigoplus_{|J|=-a+d+1} D_J$ (it is of dimension $a$)  and where the differentials $X^{r,s}\to X^{r+1,s}$, $X^{r,s}\to X^{r,s+1}$ are alternating sums of inclusion/Gysin maps. Note that the (bi-)complex is self-dual, up to applying $\{d\}$ which is compatible with the formula $\mathsf{M}(\mcX_K)^\vee\simeq \mathsf{M}(\mcX_K)\{-d\}$ (see \cite[Scholium]{agv}). This would be the motivic analog to the formulas for~$\Psi$ given by Steenbrink, Rapoport--Zink and Mokrane.
					
 				The monodromy operator  is induced by the canonical map on the weight complex
					$$N\colon C_\bullet\to C_\bullet(-1)=C_{\bullet}\{-1\}[2]=C_{\bullet -2}\{-1\}$$
					given by the identity map on each component $X^{(r)}\{s\}$ appearing on both sides (and $0$ elsewhere). One can check directly that the cone of $N$ agrees with the weight complex of $\chi \mathsf{M}(\mcX_K)$ computed above (up to twist and shift).
     \end{exm}

 \def\MR#1{}
\providecommand{\MRhref}[2]{%
\href{http://www.ams.org/mathscinet-getitem?mr=#1}{#2}
}
\providecommand{\href}[2]{#2}

			\end{document}